\documentclass{amsart}
\usepackage{amsmath,amssymb,amsfonts,amsthm}
\usepackage{mathtools}
\usepackage{mathrsfs}
\usepackage{mathabx}

\usepackage{graphicx}
\usepackage{subfigure}
\usepackage{float}

\usepackage{tikz}
\usetikzlibrary{calc,patterns}
\usepackage{pgfplots}
\pgfplotsset{compat=1.17}

\usepackage{multicol}

\usepackage{xcolor}

\usepackage[normalem]{ulem}

\usepackage{cite}
\usepackage{hyperref}
\numberwithin{equation}{section}

\allowdisplaybreaks

\newtheorem{thm}{Theorem}[section]
\newtheorem{prop}[thm]{Proposition}
\newtheorem{lem}[thm]{Lemma}
\newtheorem{cor}[thm]{Corollary}
\newtheorem{lemma}[thm]{Lemma}
\newtheorem{thmy}{Theorem}

\makeatletter
\newcommand*{\rom}[1]{\expandafter\@slowromancap\romannumeral #1@}
\makeatother
\renewcommand{\theequation}
{\thesection.\arabic{equation}}

\DeclareMathOperator\supp{supp}

\newtheorem{remark}[thm]{Remark}

\def\mr{\mathbb{R}}

\def\Lp{L^{p}}
\def\L1{L^{1}}

\date{}
\title{Two-parameter variational estimates for averages over tori}

\begin{document}
\author[Juyoung Lee]{Juyoung Lee}
\author[Sanghyuk Lee]{Sanghyuk Lee}
\author[Feng Zhang]{Feng Zhang}
\author[Shuijiang Zhao]{Shuijiang Zhao}

\address{Korea Institute for Advanced Study, 85
Hoegiro, Dongdaemun-gu, Seoul 02455, Republic of Korea}
\email{juyounglee@kias.re.kr}

\address{Department of Mathematical Sciences and RIM, Seoul National University, Seoul 08826, Republic of Korea}
\email{shklee@snu.ac.kr}

\address{School of Mathematical Sciences, Xiamen University, Xiamen 361005,
P. R. China}
\email{fengzhang@stu.xmu.edu.cn}

\address{Research Institute of Mathematics, Seoul National University, Seoul 08826, Republic of Korea}
\email{shuijiangzhao25@snu.ac.kr}

\makeatletter
\@namedef{subjclassname@2020}{\textup{2020} Mathematics Subject Classification}
\makeatother


\begin{abstract}
One-parameter variational inequalities are well developed, whereas
their multi-parameter counterparts remain much less understood.
We explore two-parameter variational inequalities for averages over
tori in $\mathbb{R}^3$. To capture the underlying two-parameter
structure, we introduce a local two-parameter $r$-variation norm that
combines rectangular increments with variations along the boundary.
The resulting variation operator pointwise dominates the corresponding
two-parameter local maximal function and, unlike the maximal function,
also captures oscillation across the two parameters. We establish
sharp $L^p$--$L^q$ bounds for this variation operator up to endpoints.
For comparison, we also obtain sharp $L^p$ bounds up to
endpoints for the corresponding local one-parameter variation operator,
revealing a genuine difference between the one- and two-parameter
boundedness regions. The proof combines square function estimates for
two-parameter propagators with local smoothing estimates through
mixed-norm interpolation.
\end{abstract}

\maketitle


\section{Introduction}\label{sec-intro}
Variational inequalities have been widely studied in probability, ergodic theory, and harmonic analysis because they imply pointwise convergence results and strengthen maximal estimates for families of operators. A landmark result is due to L\'epingle \cite{Lep}, who established variational inequalities for martingales. Later, Pisier--Xu \cite{PX} and Bourgain \cite{Bou} provided alternative proofs of L\'epingle's result. Bourgain \cite{Bou} then used L\'epingle's result to obtain variational inequalities for ergodic averages, initiating a systematic study of variational inequalities for various operators; see \cite{BORSS,CJRW1,CJRW2,GRY,JSW,OSTTW,BMSW,MST} and the references therein.

Although one-parameter variational inequalities have been studied in various settings, their multi-parameter counterparts remain far less understood. In particular, it is not yet clear what the appropriate notion of multi-parameter variation should be or whether L\'epingle's variational inequality extends to the multi-parameter setting. We refer the reader to \cite{JRW,Dab,KLMP,BMSW2,MSW} for related work on variational and oscillation inequalities in this setting.

Since variational inequalities strengthen maximal estimates, the theory of multi-parameter maximal operators provides a natural starting point. Maximal operators associated with multi-parameter averages over hypersurfaces have received considerable attention. The lacunary setting was studied in \cite{RS}, while two-parameter maximal operators associated with homogeneous surfaces were investigated in \cite{MR,MRZ}. Related results under suitable Fourier decay assumptions were obtained in \cite{Cho,Heo}. Strong spherical maximal operators were considered in \cite{Cho,Erd}, with more recent progress made in \cite{LLO1,LLO2,CGY,Zah,HZ}.

Of particular relevance here, the first and second authors studied maximal operators associated with two-parameter averages over tori in $\mr^3$ in \cite{LL}. Let $0<c_0<1$ be fixed. For $0<s<c_0t$, define the smooth parametrization
\begin{align*}
	\Phi_{t,s}(\theta,\phi):=\big((t+s\cos \theta)\cos \phi,(t+s\cos \theta)\sin \phi, s\sin \theta\big),
\end{align*}
and set $\mathbb{T}_{t,s}:=\{\Phi_{t,s}(\theta,\phi):\theta,\phi\in [0,2\pi)\}$. The measure $\sigma_{t,s}$ on $\mathbb{T}_{t,s}$ is given by
\begin{align*}
	\langle f,\sigma_{t,s}\rangle:=\int_{[0,2\pi)^{2}}f(\Phi_{t,s}(\theta,\phi))\,d\theta\,d\phi.
\end{align*}
We consider the two-parameter averaging operator
\[\mathcal{A}f(x,t,s):=f\ast\sigma_{t,s}(x).\]

For the two-parameter maximal operator
\[\mathcal{M}f(x):=\sup_{0<s<c_0t}|\mathcal{A}f(x,t,s)|,\]
the supremum is taken over the set $\{(t,s):0<s<c_0t\}$ to ensure that $\mathbb{T}_{t,s}$ remains a torus. It was shown in \cite{LL} that $\mathcal{M}$ is bounded on $\Lp(\mr^3)$ if and only if $p>2$. Moreover, nearly sharp $L^p$--$L^q$ estimates were obtained for the maximal operator
\[\mathcal{M}_cf(x):=\sup_{(t,s)\in \mathbb{J}_0}|\mathcal{A}f(x,t,s)|,\]
where $\mathbb{J}_0:=[1,2]\times [2^{-3},2^{-2}]$.
More precisely, let $\mathcal{Q}$ be the union of the open quadrilateral with vertices
\begin{align}\label{ali-max-vert}
    P_1:=(0,0), \ P_2:=(\tfrac{3}{7},\tfrac{1}{7}), \ P_3:=(\tfrac{5}{11},\tfrac{2}{11}), \ P_4:=(\tfrac{1}{2},\tfrac{1}{2}),
\end{align}
and the half-open line segment $[P_1,P_4)$. Then $\mathcal{M}_c$ is bounded from $\Lp(\mr^3)$ to $L^q(\mr^3)$ if $(1/p,1/q)\in \mathcal{Q}$, and fails to be bounded if $(1/p,1/q)\notin \overline{\mathcal{Q}}\setminus\{P_4\}$.

Motivated by these maximal estimates, we explore two-parameter variational inequalities for averages over tori. As a first step toward a broader multi-parameter theory, we introduce a local variation operator that pointwise dominates $\mathcal{M}_c$ while also measuring oscillation across the two parameters. We determine its sharp $L^p$--$L^q$ boundedness range up to endpoints.

\subsection{One-parameter variation}
To put our results in context, we first recall the definition of one-parameter variation. For a function $a$ defined on an interval $\mathbb{I}\subset\mr$, the one-parameter $r$-variation $[a]_{r,1}$ is given by
\begin{align}\label{ali-def-one-param-var}
	[a]_{r,1}:=\sup_{L\in \mathbb{N}} \sup_{\substack{t_{1}<\cdots<t_{L}\\ \{t_j\}_j\subset \mathbb{I}}}\Big(\sum_{j=1}^{L-1} |a(t_{j+1})-a(t_{j})|^{r}\Big)^{\frac{1}{r}}
\end{align}
for any $1\leq r<\infty$. For $r=\infty$, $[a]_{\infty,1}$ is defined as above, with the $\ell^r$ norm replaced by the supremum.

Various variational inequalities for spherical averages have been established in \cite{JSW,BORSS,Whe}. In particular, we recall the local variation bounds for the circular averaging operator
\begin{align*}
    \mathcal{A}^{cir}f(x,t):=\int_{\mathbb{S}^1} f(x-ty)d\sigma(y),
\end{align*}
where $d\sigma$ is the normalized surface measure on the unit circle $\mathbb{S}^1$. For the family of circular averages $\{\mathcal{A}^{cir}f(x,t)\}_{t\in [1,2]}$, the associated variation operator is given by $V_r(\mathcal{A}^{cir}f)(x):=[\mathcal{A}^{cir}f(x,\cdot)]_{r,1}$. Jones--Seeger--Wright \cite{JSW} showed that $V_r(\mathcal{A}^{cir})$ is bounded on $\Lp(\mr^2)$ if either $2<p\leq4$ and $r>2$, or $p>4$ and $r>p/2$. Moreover, the $L^p(\mr^2)$ boundedness fails if $r<\max\{2,p/2\}$. Recently, Beltran--Carbery--Roncal--Seeger \cite{BCRS} showed that the $L^p(\mr^2)$ boundedness also fails in the endpoint case $r=2$.

\subsection{Two-parameter variation}
In the one-parameter case, when $r=1$, the quantity in \eqref{ali-def-one-param-var} coincides with the total variation of a function of one variable. This suggests seeking a two-parameter analogue among the classical notions of total variation for functions of two variables. Several such notions have been studied previously; see \cite{AC,Owe,CA}. Their relations were considered by Clarkson and Adams \cite{CA}. For further information, see, for example, \cite{Chi1,Chi2,DW}.

We now introduce a two-parameter analogue of the one-parameter $r$-variation.
Let $\mathcal{A}$ be a function on $\mathbb{J}_0$. For a rectangle $Q:=[t_{1},t_{2}]\times [s_{1},s_{2}]\subset\mathbb{J}_0$, we define
\begin{align*}
	\mathcal{A}(Q):=\mathcal{A}(t_2,s_2)-\mathcal{A}(t_1,s_2)+\mathcal{A}(t_1,s_1)-\mathcal{A}(t_2,s_1).
\end{align*}
The assignment of signs to the vertices of $Q$ is shown in Figure~\ref{Fig.rectangle}.
We then define
\begin{align*}
	\{\mathcal{A}\}_{r,2}:=\sup_{ \{I_{j}\}_j,\{J_{k}\}_k}\Big(\sum_{j,k} |\mathcal{A}(I_{j}\times J_{k})|^{r}\Big)^{\frac{1}{r}}
\end{align*}
for any $1\leq r<\infty$, where the supremum is taken over all possible partitions of $[1,2]$ into $\{I_{j}\}_{j}$ and $[2^{-3},2^{-2}]$ into $\{J_{k}\}_{k}$. Replacing the $\ell^r$ norm by a supremum gives the definition of $\{\mathcal{A}\}_{\infty,2}$. The two-parameter variation norm of $\mathcal{A}$ is defined by
\begin{align*}
	[\mathcal{A}]_{r,2}:=\{\mathcal{A}\}_{r,2}+[\mathcal{A}(1,\cdot)]_{r,1}+[\mathcal{A}(\cdot,2^{-3})]_{r,1}+|\mathcal{A}(1,2^{-3})|.
\end{align*}

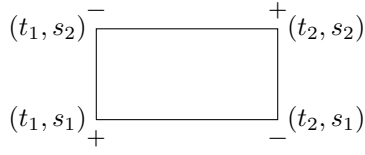
\begin{figure}[t]
	\centering
	\begin{tikzpicture}[scale=0.4]
		\draw  (0,0) -- (6,0) -- (6,3) -- (0,3) -- (0,0);
		\node[right] at (6,3) {$(t_2,s_2)$};
		\node[right] at (6,0) {$(t_2,s_1)$};
		\node[left] at (0,0) {$(t_1,s_1)$};
		\node[left] at (0,3) {$(t_1,s_2)$};
		\node [above] at (6,3) {$ +$};
		\node [above] at (0,3) {$ -$};
		\node [below] at (0,0) {$ +$};
		\node [below] at (6,0) {$ -$};
	\end{tikzpicture}
		\caption{Signs assigned to the vertices of $Q$}
	\label{Fig.rectangle}
\end{figure}

We define the $r$-variation operator $V_r({\mathcal{A}})$ by
\begin{equation}
\label{ali-def-two-parameter-variation}
\begin{aligned}
	V_{r}(\mathcal{A}f)(x):=&\widetilde{V}_{r}(\mathcal{A}f)(x)+[\mathcal{A}f(x,1,\cdot)]_{r,1}
	\\&\qquad+[\mathcal{A}f(x,\cdot,2^{-3})]_{r,1}+|\mathcal{A}f(x,1,2^{-3})|,
\end{aligned}
\end{equation}
where $\widetilde{V}_{r}(\mathcal{A}f)(x):=\{\mathcal{A}f(x,\cdot,\cdot)\}_{r,2}$.

The boundary variation terms in $[\mathcal{A}]_{r,2}$ are essential for the pointwise control of the corresponding maximal function. As shown in Propositions~\ref{prop-Vr-domina-maximal} and \ref{prop-Vr-domina-line}, the full norm controls both the supremum of $\mathcal{A}$ on $\mathbb{J}_0$ and the one-parameter variations of its restrictions to every horizontal and vertical line. In particular,
\[
    \mathcal{M}_cf(x)\leq V_r(\mathcal{A}f)(x).
\]
Thus, $V_r(\mathcal{A})$ pointwise dominates the corresponding two-parameter local maximal function while also measuring oscillation across the two parameters.

The definition of $[\mathcal{A}]_{r,2}$ is related to the Hardy--Krause variation when $r=1$. Its rectangular component $\{\mathcal{A}\}_{r,2}$ is the ``grid-like'' variation introduced in \cite{FV,FV1,FV2}, and it agrees with the Vitali variation when $r=1$. Two other natural notions, based respectively on arbitrary rectangular partitions and coordinatewise ordered sequences, are defined and compared with $[\mathcal{A}]_{r,2}$ in Section~\ref{subsec-def-var}. Lemma~\ref{lem-ali-rela-var-par} and \eqref{ali-compa-seqvar} show that $[\mathcal{A}]_{r,2}$ controls both alternatives after an arbitrarily small increase in their variation exponents. In this sense, $[\mathcal{A}]_{r,2}$ is the strongest of these notions. Another seemingly natural corner-oscillation quantity is unsuitable in continuous settings because its finiteness forces a continuous function to be constant; see Lemma~\ref{lem-osc-continuous-constant}.

Finiteness of this norm also has pointwise regularity consequences analogous to those of one-parameter variation. Proposition~\ref{prop-Vr-continuity} guarantees the existence of limits in each coordinate quadrant and restricts the discontinuity set of $\mathcal{A}$ to a countable union of straight lines parallel to the coordinate axes. Finally, Lemmas~\ref{lem-embed-1-var} and \ref{lem-embedding} control the boundary and rectangular components, respectively, by parameter-space norms of the function and its derivatives. Consequently, the variation estimates reduce to mixed-norm estimates for $\mathcal{A}f$ and its parameter derivatives, which are obtained below from local smoothing and square function estimates.

\subsection{\texorpdfstring{$L^p$ bounds}{Lp bounds}}
Our first result establishes the $L^p$ boundedness of $V_r(\mathcal{A})$.
\begin{thm}\label{thm-two-para}
    Let $\mathfrak{Q}_1$ be the closed pentagon with vertices $P_1,Q_1,P_4,Q_2,Q_4$ (see Figure~\ref{Fig.bound2}), where $Q_1:=(1/2,0)$, $Q_2:=(1/4,1/2)$, and $Q_4:=(1/8,3/8)$.
	Then ${V}_{r}(\mathcal{A})$ is bounded on $\Lp(\mathbb{R}^3)$ if $(1/p,1/r)$ lies either in the interior of $\mathfrak{Q}_1$ or on the half-open line segment $[P_1,Q_1)$. Furthermore, if $(1/p,1/r)\notin \mathfrak{Q}_1\setminus [Q_1,P_4]$, then ${V}_{r}(\mathcal{A})$ is not bounded on $\Lp(\mathbb{R}^3)$.
\end{thm}
\begin{figure}[t]
\centering
\begin{tikzpicture}[scale=0.7]
	\draw[thick,->]  (0,0) -- (7,0);
	\draw[thick,->] (0,0) -- (0,7);
	\node[left] at (0,7) {$\frac{1}{r}$};
	\node[right] at (7,0) {$\frac{1}{p}$};

	\draw (6,-2pt) -- (6,2pt);
	\node[left] at (0,6) {$\frac{1}{2}$};

    \draw (2,-2pt) -- (2,2pt);
	\node[below] at (2,0) {$\frac{1}{6}$};

	\draw (-2pt,4.5) -- (2pt,4.5);
	\node[left] at (0,4.5) {$\frac{3}{8}$};

	\draw (3,-2pt) -- (3,2pt);
	\node[below] at (3,0) {$\frac{1}{4}$};

	\draw (1.5,-2pt) -- (1.5,2pt);
	\node[below] at (1.5,0) {$\frac{1}{8}$};

	\draw (-2pt,6) -- (2pt,6);
	\node[below] at (6,0) {$\frac{1}{2}$};

	\draw[densely dotted] (6,0) -- (6,6) -- (3,6) -- (1.5,4.5) -- (0,0);
        \draw[densely dotted] (3,6) -- (2,6) -- (1.5,4.5);
	\fill[fill=black!18,
    fill opacity=0.75] (6,0) -- (6,6) -- (3,6) -- (1.5,4.5) -- (0,0);

	\node[left] at (4,3) {$\mathfrak{Q}_1$};
	\node[left,font=\footnotesize] at (0,0) {$P_{1}$};
	\node[above right,font=\footnotesize] at (6,0) {$Q_1$};
	\node[above right,font=\footnotesize] at (6,6) {$P_4$};
	\node[above right,font=\footnotesize] at (3,6) {$Q_2$};
	\node[above left,font=\footnotesize] at (1.5,4.5) {$Q_4$};
    \node[above,font=\footnotesize] at (2,6) {$Q_3$};

	\node[circle, fill, inner sep=0pt,minimum size=4pt] at (0,0) {};
	\node[circle,draw=black, fill=white, inner sep=0pt,minimum size=4pt] at (6,0) {};
	\node[circle,draw=black, fill=white, inner sep=0pt,minimum size=4pt] at (3,6) {};
	\node[circle,draw=black, fill=white, inner sep=0pt,minimum size=4pt] at (1.5,4.5) {};
	\node[circle,draw=black, fill=white, inner sep=0pt,minimum size=4pt] at (6,6) {};
    \node[circle,draw=black, fill=white, inner sep=0pt,minimum size=4pt] at (2,6) {};

\end{tikzpicture}
\caption{The pentagon $\mathfrak{Q}_1$.}
\label{Fig.bound2}
\end{figure}

For $t\in [1,2]$, let
\[\mathcal{A}_1f(x,t):=\mathcal{A}f(x,t,c_0t).\]  For simplicity, we take $c_0=2^{-3}$ in what follows, as the precise value of $c_0$ is not essential. As mentioned earlier, $\mathcal M$ is bounded on $L^p$ precisely when $p>2$. The same sharp range holds for the circular maximal function and for the one-parameter maximal operator $f\mapsto\sup_{t\in[1,2]}|\mathcal A_1f(x,t)|$; see \cite{Lee03} and \cite[p.~2]{LL}, respectively. However, the situation for variational inequalities is significantly different.

With a slight abuse of notation, we define the associated variation operator by
\[
V_r(\mathcal{A}_1f)(x):=[\mathcal{A}_1f(x,\cdot)]_{r,1}.
\]
We establish below the $L^p$ bounds for $V_r(\mathcal{A}_1)$. We restrict attention to the local form of $V_r(\mathcal{A}_1)$ in order to compare it directly with $V_r(\mathcal{A})$. In Figure~\ref{Fig.bound2}, the boundedness regions for the circular-average variation operator and for $V_r(\mathcal A_1)$ are, respectively, the trapezoids with vertices $P_1,Q_1,P_4,Q_2$ and $P_1,Q_1,P_4,Q_3$, where $Q_3:=(1/6,1/2)$. By contrast, the boundedness region for $V_r(\mathcal A)$ is the pentagon $\mathfrak Q_1$. The different shapes of these regions highlight a genuine distinction between the one-parameter variational problem and its two-parameter counterpart.

\begin{thm}\label{thm-one-para}
    Let $\mathfrak{Q}_2$ be the closed trapezoid with vertices $P_1,Q_1,P_4,Q_3$ in Figure~\ref{Fig.bound2}. If $(1/p,1/r)$ lies in the interior of $\mathfrak{Q}_2$ or in the half-open line segment $[P_1,Q_1)$, then
    \begin{align}\label{eq-vari-one}
		\|V_r(\mathcal{A}_1f)\|_{\Lp(\mathbb{R}^{3})}\leq C \|f\|_{\Lp(\mathbb{R}^{3})}.
	\end{align}
    If $(1/p,1/r)\notin \mathfrak{Q}_2\setminus[Q_1,P_4]$, then \eqref{eq-vari-one} fails.
\end{thm}

We note that Theorem~\ref{thm-two-para} and \eqref{ali-compa-seqvar} below imply, up to a harmless loss in the variation exponent, that $V_r(\mathcal{A}_1)$ is bounded on $L^p(\mr^3)$ whenever $(1/p,1/r)$ lies in the boundedness region of Theorem~\ref{thm-two-para}.

\begin{remark}\label{rem-two-impli-one}
     \normalfont
     The argument in \cite{BCRS} can be adapted to show that the variational inequalities in Theorems~\ref{thm-two-para} and \ref{thm-one-para} fail at the endpoint \(r=2\), thereby establishing the sharpness of the requirement \(r>2\); we do not include the details here.
\end{remark}

\subsection{\texorpdfstring{$L^p$--$L^q$ bounds}{Lp--Lq bounds}}
To present the main theorem on the $L^p$--$L^q$ bounds for $V_r(\mathcal{A})$, we first recall the restriction inherited from the maximal operator. Since $V_r(\mathcal{A})$ dominates $\mathcal{M}_c$, the unboundedness of $\mathcal{M}_c$ whenever $(1/p,1/q)\notin\overline{\mathcal{Q}}\setminus\{P_4\}$ implies the same for $V_r(\mathcal{A})$. The parameter $r$ imposes additional restrictions within $\mathcal{Q}$, and we therefore define the type set
\[\mathcal{T}(r):=\big\{\big(\tfrac1p,\tfrac1q\big)\in \mathcal{Q}:V_r(\mathcal{A}) \text{ is bounded from~}L^p(\mr^3) \text{ to~} L^q(\mr^3)\big\}.\]

We introduce the following affine functions:
\begin{equation}\label{ali-affine-functions}
\begin{aligned}
    L_1(p,q,r)&:=\frac{3}{q}-\frac{1}{r},\\
    L_2(p,q,r)&:=-\frac{3}{2p}+\frac{5}{2q}-\frac{1}{r}+\frac{1}{2},\\
    L_3(p,q,r)&:=-\frac{2}{p}+\frac{1}{q}-\frac{1}{r}+1,\\
    L_4(p,q,r)&:=-\frac{1}{2p}+\frac{5}{2q}-\frac{2}{r}+\frac{1}{2},\\
    L_5(p,q,r)&:=-\frac{1}{p}+\frac{1}{q}-\frac{2}{r}+1.
\end{aligned}
\end{equation}
The corresponding weak inequalities $L_j(p,q,r)\geq0$ appear as
necessary conditions in Section~\ref{sec-nec-cond}. For
$j=1,\ldots,5$, define
\[
    \mathcal{R}_j(r)
    :=\left\{\left(\frac1p,\frac1q\right)\in\mathcal Q:
    L_j(p,q,r)>0\right\}.
\]

The following theorem characterizes the type set $\mathcal{T}(r)$ up to endpoints.

\begin{thm}\label{thm-Lp-Lq}
For $1\leq r\leq\infty$, let
\[
\mathfrak{P}(r):=\bigcap_{j=1}^{5}\mathcal{R}_j(r).
\]
Then
\[
\mathfrak{P}(r)\subset\mathcal{T}(r)\subset
\Big(
\overline{\mathcal Q}
\cap\{L_1(p,q,r)\geq0\}
\cap\bigcap_{j=2}^5\{L_j(p,q,r)\geq0\}
\Big)\setminus\{P_4\}.
\]
\end{thm}

\subsubsection*{Description of the boundedness regions}
The geometry of $\mathfrak{P}(r)$ changes according to the range of $r$. We describe the set in the following six cases.

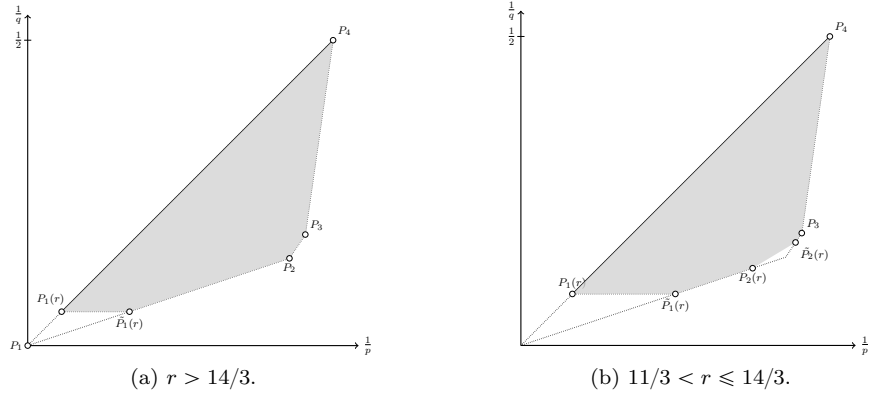
\begin{figure}[H]
    \centering

    \begin{minipage}[t]{0.48\textwidth}
        \centering
        \resizebox{0.82\linewidth}{!}{%
        \begin{tikzpicture}[scale=1.3]
            \draw[thick,->]  (0,0) -- (6.5,0);
            \draw[thick,->] (0,0) -- (0,6.5);
            \node[left] at (0,6.5) {$\frac{1}{q}$};
            \node[left] at (0,6) {$\frac{1}{2}$};
            \node[right] at (6.5,0) {$\frac{1}{p}$};
            \draw (-2pt,6) -- (2pt,6);

            \draw (2/3,2/3) -- (6,6);
            \draw[densely dotted] (2/3,2/3) -- (2,2/3);
            \draw[densely dotted] (0,0) -- (6,6) -- (60/11,24/11) -- (36/7,12/7) -- (0,0);

            \node[left,font=\footnotesize] at (0,0) {$P_{1}$};
            \node[above right,font=\footnotesize] at (6,6) {$P_4$};
            \node[above right,font=\footnotesize] at (60/11,24/11) {$P_3$};
            \node[below,font=\footnotesize] at (36/7,12/7) {$P_2$};
            \node[above,font=\footnotesize] at (0.45,0.7) {$P_1(r)$};
            \node[below,font=\footnotesize] at (2,2/3) {$\tilde{P_1}(r)$};

            \fill[fill=black!18, fill opacity=0.75]
            (2/3,2/3) -- (6,6) -- (60/11,24/11) -- (36/7,12/7) -- (2,2/3) -- cycle;

            \node[circle,draw=black, fill=white, inner sep=0pt,minimum size=4pt] at (0,0) {};
            \node[circle,draw=black, fill=white, inner sep=0pt,minimum size=4pt] at (60/11,24/11) {};
            \node[circle,draw=black, fill=white, inner sep=0pt,minimum size=4pt] at (36/7,12/7) {};
            \node[circle,draw=black, fill=white, inner sep=0pt,minimum size=4pt] at (6,6) {};
            \node[circle,draw=black, fill=white, inner sep=0pt,minimum size=4pt] at (2/3,2/3) {};
            \node[circle,draw=black, fill=white, inner sep=0pt,minimum size=4pt] at (2,2/3) {};
        \end{tikzpicture}
        }
        \par\vspace{-1mm}{\footnotesize (a) $r>14/3$.}\par
    \end{minipage}
    \hfill
    \begin{minipage}[t]{0.48\textwidth}
        \centering
        \resizebox{0.82\linewidth}{!}{%
        \begin{tikzpicture}[scale=1.3]
            \draw[thick,->]  (0,0) -- (6.5,0);
            \draw[thick,->] (0,0) -- (0,6.5);
            \node[left] at (0,6.5) {$\frac{1}{q}$};
            \node[left] at (0,6) {$\frac{1}{2}$};
            \node[right] at (6.5,0) {$\frac{1}{p}$};
            \draw (-2pt,6) -- (2pt,6);

            \draw (1,1) -- (6,6);
            \draw[densely dotted] (1,1) -- (3,1);
            \draw[densely dotted] (0,0) -- (6,6) -- (60/11,24/11) -- (36/7,12/7) -- (0,0);

            \node[above right,font=\footnotesize] at (6,6) {$P_4$};
            \node[above right,font=\footnotesize] at (60/11,24/11) {$P_3$};
            \node[above,font=\footnotesize] at (1,1) {$P_1(r)$};
            \node[below,font=\footnotesize] at (3,1) {$\tilde{P_1}(r)$};
            \node[below,font=\footnotesize] at (9/2,3/2) {$P_2(r)$};
            \node[below right,font=\footnotesize] at (16/3,2) {$\tilde{P_2}(r)$};

            \fill[fill=black!18, fill opacity=0.75]
            (1,1) -- (3,1) -- (9/2,3/2) -- (16/3,2) -- (60/11,24/11) -- (6,6) -- cycle;

            \node[circle,draw=black, fill=white, inner sep=0pt,minimum size=4pt] at (60/11,24/11) {};
            \node[circle,draw=black, fill=white, inner sep=0pt,minimum size=4pt] at (6,6) {};
            \node[circle,draw=black, fill=white, inner sep=0pt,minimum size=4pt] at (1,1) {};
            \node[circle,draw=black, fill=white, inner sep=0pt,minimum size=4pt] at (3,1) {};
            \node[circle,draw=black, fill=white, inner sep=0pt,minimum size=4pt] at (9/2,3/2) {};
            \node[circle,draw=black, fill=white, inner sep=0pt,minimum size=4pt] at (16/3,2) {};
        \end{tikzpicture}
        }
        \par\vspace{-1mm}{\footnotesize (b) $11/3<r\leq14/3$.}\par
    \end{minipage}

    \caption{The regions $\mathfrak{P}(r)$ in cases \textnormal{(a)} and \textnormal{(b)}.}
    \label{Fig.regions-Pr-ab}
\end{figure}

\medskip
\noindent\textbf{(a) $r>14/3$.}
The region $\mathfrak{P}(r)$ in Figure~\ref{Fig.regions-Pr-ab}(a) coincides with $\mathcal{R}_1(r)$ and has vertices $P_1(r), \tilde{P_1}(r)$, $P_2, P_3, P_4$, where $P_2, P_3, P_4$ are given in \eqref{ali-max-vert}, and
\[P_1(r):=(\tfrac{1}{3r},\tfrac{1}{3r}),\  \tilde{P_1}(r):=(\tfrac{1}{r},\tfrac{1}{3r}).\]
The line segment connecting $P_1(r)$ and $\tilde{P_1}(r)$ lies on the line $L_1(p,q,r)=0$. Taking $r=\infty$, we recover the $L^p$--$L^q$ bounds for the maximal operator $\mathcal{M}_c$. As $r\rightarrow\infty$, both $P_1(r)$ and $\tilde{P_1}(r)$ tend to $P_1$. Compared with the boundedness region for $\mathcal{M}_c$, the region below the segment $[P_1(r),\tilde{P_1}(r)]$ is excluded.

\medskip
\noindent\textbf{(b) $11/3<r\leq14/3$.}
Compared with case \textnormal{(a)}, the region below the segment $[P_2(r),\tilde{P_2}(r)]$ is excluded; see Figure~\ref{Fig.regions-Pr-ab}(b). The region $\mathfrak{P}(r)$ equals $\mathcal{R}_1(r)\cap\mathcal{R}_2(r)$ and has vertices $P_1(r), \tilde{P_1}(r), P_2(r), \tilde{P_2}(r), P_3, P_4$, where
\[P_2(r):=(\tfrac34-\tfrac{3}{2r},\tfrac14-\tfrac{1}{2r}),\ \tilde{P_2}(r):=(\tfrac{4}{9r}+\tfrac13,\tfrac{2}{3r}).\]
The segment $[P_2(r),\tilde{P_2}(r)]$ is given by $L_2(p,q,r)=0$. Note that $P_2(r)=\tilde{P_2}(r)=P_2$ when $r=14/3$, and $\tilde{P_2}(r)=P_3$ when $r=11/3$.

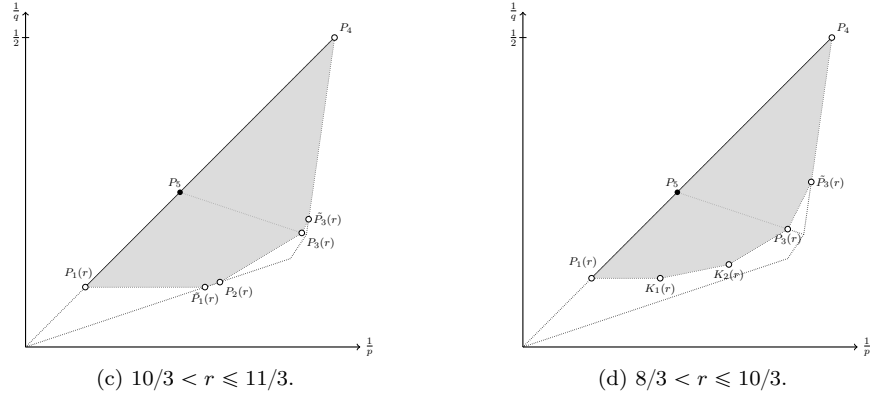
\begin{figure}[H]
    \centering

    \begin{minipage}[t]{0.48\textwidth}
        \centering
        \resizebox{0.82\linewidth}{!}{%
        \begin{tikzpicture}[scale=1.3]
            \draw[thick,->]  (0,0) -- (6.5,0);
            \draw[thick,->] (0,0) -- (0,6.5);
            \node[left] at (0,6.5) {$\frac{1}{q}$};
            \node[left] at (0,6) {$\frac{1}{2}$};
            \node[right] at (6.5,0) {$\frac{1}{p}$};
            \draw (-2pt,6) -- (2pt,6);

            \draw (9/8,9/8) -- (6,6);
            \draw[densely dotted] (0,0) -- (6,6) -- (60/11,24/11) -- (36/7,12/7) -- (0,0);
            \draw[densely dotted] (36/31,36/31) -- (108/31,36/31) -- (117/31,39/31) --  (1164/217,480/217) -- (852/155,384/155);
            \draw[densely dotted] (3,3) -- (60/11,24/11);

            \node[above right,font=\footnotesize] at (6,6) {$P_4$};
            \node[above,font=\footnotesize] at (32/31,38/31) {$P_1(r)$};
            \node[above,font=\footnotesize] at (2.9,3) {$P_5$};
            \node[below,font=\footnotesize] at (108/31,36/31) {$\tilde{P_1}(r)$};
            \node[below right,font=\footnotesize] at (117/31,39/31) {$P_2(r)$};
            \node[below right,font=\footnotesize] at (1164/217,480/217) {$P_3(r)$};
            \node[right,font=\footnotesize] at (852/155,384/155) {$\tilde{P_3}(r)$};

            \fill[fill=black!18, fill opacity=0.75]
            (36/31,36/31) -- (108/31,36/31) -- (117/31,39/31) -- (1164/217,480/217) -- (852/155,384/155) -- (6,6) -- cycle;

            \node[circle,draw=black, fill=white, inner sep=0pt,minimum size=4pt] at (6,6) {};
            \node[circle, fill, inner sep=0pt,minimum size=4pt] at (3,3) {};
            \node[circle,draw=black, fill=white, inner sep=0pt,minimum size=4pt] at (36/31,36/31) {};
            \node[circle,draw=black, fill=white, inner sep=0pt,minimum size=4pt] at (108/31,36/31) {};
            \node[circle,draw=black, fill=white, inner sep=0pt,minimum size=4pt] at (117/31,39/31) {};
            \node[circle,draw=black, fill=white, inner sep=0pt,minimum size=4pt] at (1164/217,480/217) {};
            \node[circle,draw=black, fill=white, inner sep=0pt,minimum size=4pt] at (852/155,384/155) {};
        \end{tikzpicture}
        }
        \par\vspace{-1mm}{\footnotesize (c) $10/3<r\leq11/3$.}\par
    \end{minipage}
    \hfill
    \begin{minipage}[t]{0.48\textwidth}
        \centering
        \resizebox{0.82\linewidth}{!}{%
        \begin{tikzpicture}[scale=1.3]
            \draw[thick,->]  (0,0) -- (6.5,0);
            \draw[thick,->] (0,0) -- (0,6.5);
            \node[left] at (0,6.5) {$\frac{1}{q}$};
            \node[left] at (0,6) {$\frac{1}{2}$};
            \node[right] at (6.5,0) {$\frac{1}{p}$};
            \draw (-2pt,6) -- (2pt,6);

            \draw (4/3,4/3) -- (6,6);
            \draw[densely dotted] (0,0) -- (6,6) -- (60/11,24/11) -- (36/7,12/7) -- (0,0);
            \draw[densely dotted] (3,3) -- (60/11,24/11);
            \draw[densely dotted] (4/3,4/3) -- (8/3,4/3) -- (4,8/5) -- (36/7,16/7) -- (28/5,16/5);

            \node[above right,font=\footnotesize] at (6,6) {$P_4$};
            \node[above,font=\footnotesize] at (2.9,3) {$P_5$};
            \node[above,font=\footnotesize] at (3.5/3,4.3/3) {$P_1(r)$};
            \node[below,font=\footnotesize] at (8/3,4/3) {$K_1(r)$};
            \node[below,font=\footnotesize] at (4,8/5) {$K_2(r)$};
            \node[below,font=\footnotesize] at (36/7,16/7) {$P_3(r)$};
            \node[right,font=\footnotesize] at (28/5,16/5) {$\tilde{P_3}(r)$};

            \fill[fill=black!18, fill opacity=0.75]
            (4/3,4/3) -- (8/3,4/3) -- (4,8/5) -- (36/7,16/7) -- (28/5,16/5) -- (6,6) -- cycle;

            \node[circle, fill, inner sep=0pt,minimum size=4pt] at (3,3) {};
            \node[circle,draw=black, fill=white, inner sep=0pt,minimum size=4pt] at (6,6) {};
            \node[circle,draw=black, fill=white, inner sep=0pt,minimum size=4pt] at (4/3,4/3) {};
            \node[circle,draw=black, fill=white, inner sep=0pt,minimum size=4pt] at (8/3,4/3) {};
            \node[circle,draw=black, fill=white, inner sep=0pt,minimum size=4pt] at (4,8/5) {};
            \node[circle,draw=black, fill=white, inner sep=0pt,minimum size=4pt] at (36/7,16/7) {};
            \node[circle,draw=black, fill=white, inner sep=0pt,minimum size=4pt] at (28/5,16/5) {};
        \end{tikzpicture}
        }
        \par\vspace{-1mm}{\footnotesize (d) $8/3<r\leq10/3$.}\par
    \end{minipage}

    \caption{The regions $\mathfrak{P}(r)$ in cases \textnormal{(c)} and \textnormal{(d)}.}
    \label{Fig.regions-Pr-cd}
\end{figure}

\medskip
\noindent\textbf{(c) $10/3<r\leq11/3$.}
The region $\mathfrak{P}(r)$ in Figure~\ref{Fig.regions-Pr-cd}(c) coincides with $\bigcap_{j=1}^{3}\mathcal{R}_j(r)$ and has vertices $P_1(r), \tilde{P_1}(r), P_2(r), P_3(r), \tilde{P_3}(r), P_4$, where
\[P_3(r):=(\tfrac{4}{7}-\tfrac{3}{7r},\tfrac{1}{7}+\tfrac{1}{7r}), \ \tilde{P_3}(r):=(\tfrac{1}{5r}+\tfrac{2}{5},\tfrac{7}{5r}-\tfrac{1}{5}). \]
$P_3(r)$ lies on the segment connecting $P_3$ and $P_5:=(1/4,1/4)$. The segments $[P_2(r),P_3(r)]$ and $[P_3(r),\tilde{P_3}(r)]$ are determined by $L_2(p,q,r)=0$ and $L_3(p,q,r)=0$, respectively. Note that $P_3(r)=\tilde{P_3}(r)=P_3$ when $r=11/3$, and that $\tilde{P_1}(r)=P_2(r)$ when $r=10/3$.

\medskip
\noindent\textbf{(d) $8/3<r\leq10/3$.}
The region $\mathfrak{P}(r)$ in Figure~\ref{Fig.regions-Pr-cd}(d) equals $\bigcap_{j=1}^4\mathcal{R}_j(r)$ and has vertices $P_1(r), K_1(r), K_2(r), P_3(r), \tilde{P_3}(r), P_4$, with
\[K_1(r):=(1-\tfrac{7}{3r},\tfrac{1}{3r}),\ K_2(r):=(\tfrac{1}{r},\tfrac{1}{r}-\tfrac{1}{5}).\]
The segments $[P_1(r),K_1(r)]$, $[K_1(r),K_2(r)]$, and $[K_2(r),P_3(r)]$ are determined by $L_1(p,q,r)\allowbreak=0$, $L_4(p,q,r)=0$, and $L_2(p,q,r)=0$, respectively. If $r=10/3$, then $K_1(r)=K_2(r)$, and this point coincides with $\tilde{P_1}(r)=P_2(r)$ in Figure~\ref{Fig.regions-Pr-cd}(c). If $r=8/3$, then $P_1(r)=K_1(r)$, and this point coincides with $P_5(r):=(\tfrac{1}{r}-\tfrac{1}{4},\tfrac{1}{r}-\tfrac{1}{4})$ in Figure~\ref{Fig.regions-Pr-ef}(e).

\begin{figure}[H]
    \centering

    \begin{minipage}[t]{0.48\textwidth}
        \centering
        \resizebox{0.82\linewidth}{!}{%
        \begin{tikzpicture}[scale=1.3]
            \draw[thick,->]  (0,0) -- (6.5,0);
            \draw[thick,->] (0,0) -- (0,6.5);
            \node[left] at (0,6.5) {$\frac{1}{q}$};
            \node[left] at (0,6) {$\frac{1}{2}$};
            \node[right] at (6.5,0) {$\frac{1}{p}$};
            \draw (-2pt,6) -- (2pt,6);

            \draw (33/20,33/20) -- (6,6);
            \draw[densely dotted] (0,0) -- (6,6) -- (60/11,24/11) -- (36/7,12/7) -- (0,0);
            \draw[densely dotted] (33/20,33/20) -- (93/20,9/4) -- (681/140,333/140) -- (573/100,411/100);
            \draw[densely dotted] (3,3) -- (60/11,24/11);

            \node[above,font=\footnotesize] at (2.9,3) {$P_5$};
            \node[above right,font=\footnotesize] at (6,6) {$P_{4}$};
            \node[above,font=\footnotesize] at (3/2,35/20) {$P_5(r)$};
            \node[below,font=\footnotesize] at (93/20,9/4) {$K_2(r)$};
            \node[right,font=\footnotesize] at (681/140,333/140) {$P_3(r)$};
            \node[right,font=\footnotesize] at (573/100,411/100) {$\tilde{P_3}(r)$};

            \fill[fill=black!18, fill opacity=0.75]
            (33/20,33/20) -- (93/20,9/4) -- (681/140,333/140) -- (573/100,411/100) -- (6,6) -- cycle;

            \node[circle, fill, inner sep=0pt,minimum size=4pt] at (3,3) {};
            \node[circle,draw=black, fill=white, inner sep=0pt,minimum size=4pt] at (6,6) {};
            \node[circle,draw=black, fill=white, inner sep=0pt,minimum size=4pt] at (33/20,33/20) {};
            \node[circle,draw=black, fill=white, inner sep=0pt,minimum size=4pt] at (93/20,9/4) {};
            \node[circle,draw=black, fill=white, inner sep=0pt,minimum size=4pt] at (681/140,333/140) {};
            \node[circle,draw=black, fill=white, inner sep=0pt,minimum size=4pt] at (573/100,411/100) {};
        \end{tikzpicture}
        }
        \par\vspace{-1mm}{\footnotesize (e) $5/2<r\leq8/3$.}\par
    \end{minipage}
    \hfill
    \begin{minipage}[t]{0.48\textwidth}
        \centering
        \resizebox{0.82\linewidth}{!}{%
        \begin{tikzpicture}[scale=1.3]
            \draw[thick,->]  (0,0) -- (6.5,0);
            \draw[thick,->] (0,0) -- (0,6.5);
            \node[left] at (0,6.5) {$\frac{1}{q}$};
            \node[left] at (0,6) {$\frac{1}{2}$};
            \node[right] at (6.5,0) {$\frac{1}{p}$};
            \draw (-2pt,6) -- (2pt,6);

            \draw (2,2) -- (6,6);
            \draw[densely dotted] (0,0) -- (6,6) -- (60/11,24/11) -- (36/7,12/7) -- (0,0);
            \draw[densely dotted] (3,3) -- (60/11,24/11);
            \draw[densely dotted] (2,2) -- (9/2,5/2) -- (5,3) -- (29/5,23/5);

            \node[above right,font=\footnotesize] at (6,6) {$P_4$};
            \node[above,font=\footnotesize] at (2.9,3) {$P_5$};
            \node[above,font=\footnotesize] at (1.8,2.1) {$P_5(r)$};
            \node[below,font=\footnotesize] at (9/2,5/2) {$\tilde{P_5}(r)$};
            \node[below,font=\footnotesize] at (5.19,3) {$P_4(r)$};
            \node[right,font=\footnotesize] at (29/5,23/5) {$\tilde{P_3}(r)$};

            \fill[fill=black!18, fill opacity=0.75]
            (2,2) -- (9/2,5/2) -- (5,3) -- (29/5,23/5) -- (6,6) -- cycle;

            \node[circle, fill, inner sep=0pt,minimum size=4pt] at (3,3) {};
            \node[circle,draw=black, fill=white, inner sep=0pt,minimum size=4pt] at (6,6) {};
            \node[circle,draw=black, fill=white, inner sep=0pt,minimum size=4pt] at (2,2) {};
            \node[circle,draw=black, fill=white, inner sep=0pt,minimum size=4pt] at (9/2,5/2) {};
            \node[circle,draw=black, fill=white, inner sep=0pt,minimum size=4pt] at (5,3) {};
            \node[circle,draw=black, fill=white, inner sep=0pt,minimum size=4pt] at (29/5,23/5) {};
        \end{tikzpicture}
        }
        \par\vspace{-1mm}{\footnotesize (f) $2<r\leq5/2$.}\par
    \end{minipage}

    \caption{The regions $\mathfrak{P}(r)$ in cases \textnormal{(e)} and \textnormal{(f)}.}
    \label{Fig.regions-Pr-ef}
\end{figure}
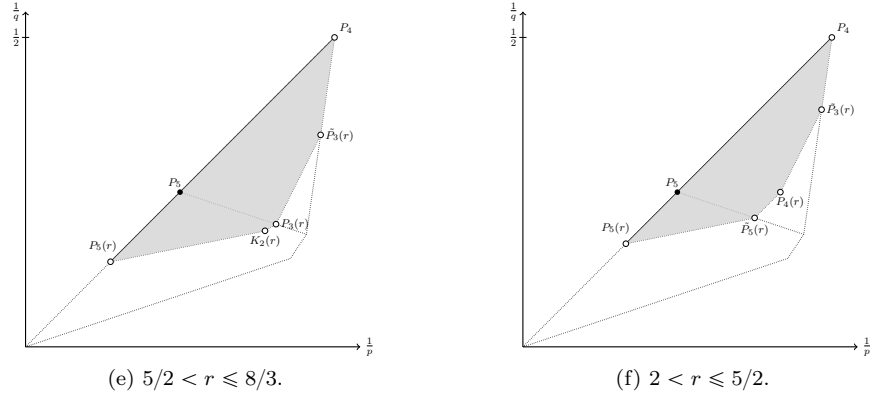

\medskip
\noindent\textbf{(e) $5/2<r\leq8/3$.}
The region $\mathfrak{P}(r)$ in Figure~\ref{Fig.regions-Pr-ef}(e) coincides with $\bigcap_{j=2}^4\mathcal{R}_j(r)$. Its vertices are
\[
P_5(r),\ K_2(r),\ P_3(r),\ \tilde{P_3}(r),\ P_4.
\]
The segment $[P_5(r),K_2(r)]$ is given by $L_4(p,q,r)=0$. If $r=5/2$, then $K_2(r)=P_3(r)$, and this point coincides with $\tilde{P_5}(r):=(1-\tfrac{3}{2r},\tfrac{1}{2r})$ in Figure~\ref{Fig.regions-Pr-ef}(f).

\medskip
\noindent\textbf{(f) $2<r\leq5/2$.}
The region $\mathfrak{P}(r)$ in Figure~\ref{Fig.regions-Pr-ef}(f) equals $\bigcap_{j=3}^5\mathcal{R}_j(r)$ and has vertices $P_5(r), \tilde{P_5}(r), P_4(r):=(\tfrac1r,\tfrac3r-1), \tilde{P_3}(r), P_4$. The segments $[P_5(r),\tilde{P_5}(r)]$, $[\tilde{P_5}(r),P_4(r)]$, and $[P_4(r),\tilde{P_3}(r)]$ are determined by $L_4(p,q,r)=0$, $L_5(p,q,r)=0$, and $L_3(p,q,r)=0$, respectively. If $r=5/2$, then $\tilde{P_5}(r)=P_4(r)$, and this point coincides with $K_2(r)=P_3(r)$ in Figure~\ref{Fig.regions-Pr-ef}(e). As $r\rightarrow2$, both $P_5(r)$ and $\tilde{P_5}(r)$ tend to $P_5$, while $P_4(r)$ and $\tilde{P_3}(r)$ tend to $P_4$.

\subsection{Further directions}
It is natural to study variation over parameter domains beyond the
fixed rectangle $\mathbb J_0$. General parallelograms are not obviously
reducible to $\mathbb J_0$, since the definition of
$[\mathcal A]_{r,2}$ is tied to coordinatewise rectangular increments.
One may also consider strips such as
$[1,2]\times(0,2^{-2}]$ or
$[1,\infty)\times[2^{-3},2^{-2}]$. In the former case,
$\mathbb T_{t,s}$ degenerates to a circle as $s\to0$, so the
boundedness region need not coincide with that of the local operator
$V_r(\mathcal A)$.

The comparisons in Section~\ref{subsec-def-var} point to a more
structural question: which notion of two-parameter variation 
captures joint oscillation, controls the corresponding maximal
function, and admits a L\'epingle-type inequality? The full variation
norm used here has the first two properties and, up to an arbitrarily
small increase in the target variation exponent, controls the notions
based on arbitrary rectangular partitions and coordinatewise ordered
sequences. However, the sequence-based variation does not control the
full norm in the reverse direction, and the corner-oscillation
quantity is unsuitable for continuous parameter families.
We do not know whether an appropriate long-variation
analogue of the present rectangular norm satisfies a L\'epingle-type
inequality.
Suitably formulated multi-parameter jump or oscillation
estimates, distinct from the corner-oscillation quantity considered in
Section~\ref{subsec-def-var}, may provide intermediate substitutes.

The present results concern the fixed relative-scale rectangle
$\mathbb J_0$ and may be viewed as a model for the within-box
short-variation estimates needed in a global theory. A genuinely
two-parameter global result over $0<s<c_0t$ would require uniform
local estimates on dyadic boxes with
$t\sim2^j$, $s\sim2^k$, and $k\leq j+O(1)$, including the regime
$s/t\to0$, together with control of the long variation across this
two-dimensional dyadic index set. A more accessible first problem is
the family obtained by applying the common dilations
$(t,s)\mapsto(2^\ell t,2^\ell s)$ to a fixed relative-scale region.
Here only the inter-box scale $\ell$ is one-dimensional, although the
variation within each box remains two-parameter. Consequently, the
usual one-parameter long--short decomposition and L\'epingle-type
tools provide a plausible approach, but a global estimate does not
follow formally from the present local results.

\subsection{Key ingredients of the proofs}
We briefly explain the key ideas in the proofs of our main theorems. We begin with the $L^p$ case in Theorem~\ref{thm-two-para}, which follows directly from Theorem~\ref{thm-Lp-Lq} by setting $p=q$ and is therefore not proved separately. Nevertheless, this case provides a convenient setting for explaining our method.

The term $|\mathcal{A}f(x,1,2^{-3})|$ is controlled directly by Minkowski's inequality or by the maximal bounds established in \cite{LL}. Thus, to establish the $L^p$ bounds for $V_r(\mathcal{A})$, it suffices to focus on three parts:
\[\widetilde{V}_r(\mathcal{A}f),\quad [\mathcal{A}f(x,1,\cdot)]_{r,1}, \quad \,[\mathcal{A}f(x,\cdot,2^{-3})]_{r,1}.\]
The key term is $\widetilde{V}_r(\mathcal{A}f)$.
The remaining two terms are treated in Section~\ref{sec-one-para} by combining the one-parameter variation embedding with mixed-norm estimates for the corresponding boundary operators, obtained through a Fourier-series argument.

\subsubsection*{Two-parameter propagator}
Since $\widetilde{V}_r(\mathcal{A})\leq \widetilde{V}_p(\mathcal{A})$ for $r\geq p$, the range $r>p$ follows from the case $r=p$.
It therefore suffices to consider $r\leq p$.
By the asymptotic expansion of $\widehat{\sigma_{t,s}}(\xi)$ (see \eqref{ali-expansion-sigmats} and \eqref{asm-B-symbol}) and symmetry, the analysis of $\widetilde{V}_r(\mathcal{A})$ reduces to that of the two-parameter propagator
\begin{align}\label{def-wjk}
    \mathcal{U}_{t,s}^Bf(x):=\int_{\mr^3}e^{i(x\cdot\xi+t|\bar{\xi}|+s|\xi|)}B(\xi,t,s) \widehat{f}(\xi)d\xi,
    \quad \xi=(\bar{\xi},\xi_3)\in\mr^2\times\mr,
\end{align}
where $B(\xi,t,s)\in C^\infty(\mr^3\times \mathbb{J}_0)$ satisfies
\begin{align}\label{ali-condition-on-B}
    \sup_{(t,s)\in\mathbb{J}_0}|\partial_{t,s}^{\gamma}\partial_{\xi}^{\alpha}B(\xi,t,s)|\leq C_{\gamma,\alpha}(1+|\bar{\xi}|)^{-|(\alpha_1,\alpha_2)|}(1+|\xi|)^{-\alpha_3}
\end{align}
for all multi-indices $\gamma$ and $\alpha$. For convenience, we shall suppress the dependence on the symbol $B$ throughout the paper, since all constants $C_{\gamma,\alpha}$ are uniform with respect to $B$. By the embedding lemma (see Lemma~\ref{lem-embedding}) and frequency localization, the $L^p$ bound for $\widetilde{V}_r(\mathcal{A})$ is reduced to estimating the following quantities. Writing $\lambda$ and $h$ for the horizontal and vertical dyadic frequency scales, respectively, we have
\begin{align*}
   \|\mathcal{U}_{t,s}f_{\lambda,h}\|_{L_x^p(\mr^3;L_{t,s}^r(\mathbb{J}_0))},\quad \lambda,h\geq2,
\end{align*}
where
\begin{align*}
    \widehat{f_{\lambda,h}}(\xi):=\widehat{f}(\xi)
    \varphi(|\bar{\xi}|/\lambda)\varphi(|\xi_3|/h).
\end{align*}
Here $\varphi$ is the cutoff fixed in the Notation subsection below.
As in the maximal estimate for $\mathcal{M}_c$, which corresponds to the limiting case $r=\infty$, the diagonal frequency regime $h=\lambda$ is dominant in the $L^p$ analysis. The local smoothing estimates established in \cite[Proposition~2.3]{LL}, together with H\"older's inequality, imply that for $\lambda\geq 2$ and any $\epsilon>0$,
\begin{align*}
    \|\mathcal{U}_{t,s}f_{\lambda,\lambda}\|_{L_x^p(\mr^3;L_{t,s}^r(\mathbb{J}_0))}\lesssim
        \begin{cases}
	    			\lambda^{1-\frac{4}{p}+\epsilon}\|f\|_{\Lp(\mr^3)}, & p\geq 4,\\[2pt]
	    			\lambda^{\epsilon}\|f\|_{\Lp(\mr^3)}, & 2\leq p\leq4.
	    	\end{cases}
\end{align*}
\subsubsection*{Two-parameter square function}
The preceding estimates yield a partial boundedness region for
$\widetilde{V}_r(\mathcal{A})$, namely the trapezoid with vertices
$P_1,Q_1,P_4,Q_2$. To enlarge this region by including the triangle
with vertices $P_1,Q_2,Q_4$, we use the square function
\begin{align}\label{def-square-function}
    \mathcal{S}f(x):=
    \|\mathcal{U}_{t,s}f\|_{L^2_{t,s}(\mathbb{J}_0)}.
\end{align}
The $L^\infty$ square function estimates needed for the $L^p$ bounds
are established in the two intermediate frequency regimes:
\eqref{ali-main-jk} covers
$\lambda^{1/2+\delta}\leq h\leq\lambda$, while the first estimate in
Proposition~\ref{prop-main-j<k<2j} covers
$\lambda\leq h\leq\lambda^2$.

Interpolating the diagonal case $h=\lambda$ of \eqref{ali-main-jk} with the
local smoothing estimate for $p=4$ and applying H\"older's inequality,
we obtain
\[
\|\mathcal{U}_{t,s}f_{\lambda,\lambda}\|
_{L_x^p(\mr^3;L_{t,s}^r(\mathbb{J}_0))}
\lesssim
\lambda^{\frac{1}{2}-\frac{2}{p}+\epsilon}
\|f\|_{\Lp(\mr^3)}
\]
for $p\geq4$ and $r\leq2p/(p-2)$. The relation $r=2p/(p-2)$ determines the new boundary of the
$L^p$ exponent region. Combined with the estimates in the remaining
frequency regimes and the dyadic summation in Section
\ref{sec-Lp-Lq-bounds}, this estimate yields the sharp boundedness
region for $\widetilde{V}_r(\mathcal{A})$ up to endpoints. For
Theorem~\ref{thm-one-para}, the local smoothing estimates alone yield
the corresponding sharp boundedness region up to endpoints for
$V_r(\mathcal{A}_1)$.

The role of square function estimates here differs from that in the
variational theory for circular averages. Let $\mathcal{W}_\pm$ be
the two-dimensional wave operator defined by
\begin{align}\label{ali-def-2d-wave}
    \mathcal{W}_\pm g(y,t)
    :=\int_{\mr^2}e^{i(y\cdot\eta\pm t|\eta|)}
    \widehat{g}(\eta)\,d\eta
\end{align}
for $(y,t)\in\mr^2\times[1,2]$. For global variational inequalities
for circular averages, one uses
\begin{align*}
    \|\mathcal{W}_\pm g\|_{\Lp(\mr^2;L^2([1,2]))}
    \lesssim
    \lambda^{\frac12-\frac2p+\epsilon}\|g\|_{\Lp(\mr^2)}
\end{align*}
for $p\geq4$, any $\epsilon>0$, and
$\supp\widehat g\subset\mathbb A_\lambda$. These bounds follow from
the sharp local smoothing estimate in \cite{GWZ} and H\"older's
inequality, while the optimality of the power of $\lambda$, up to the
$\epsilon$-loss, follows from the duality argument in \cite{LRS}.
For the two-parameter propagator considered here, however, the
available local smoothing estimates alone do not yield the
corresponding sharp square function bounds.

To obtain the sharp $L^p$--$L^q$ bounds for $V_r(\mathcal{A})$ up to
endpoints, we also use the $L^2\rightarrow L^\infty(L^2)$ and
$L^2\rightarrow L^6(L^2)$ square function estimates established in
Section~\ref{sec-square-func-est}. Interpolating these estimates with
the local smoothing estimates in \cite{LL}, we obtain the mixed-norm
bounds needed to control $\widetilde{V}_r(\mathcal A)$. Unlike in the
$L^p$ theory, where the diagonal case $h=\lambda$ is dominant, the
regime $h\approx\lambda^{1/2}$ also plays an essential role in the
$L^p$--$L^q$ analysis. The same anisotropic scale is reflected in the
constructions used to establish necessity; see Proposition~\ref{prop-nece-r>2} (c)
and Proposition~\ref{prop-nec-Lp-Lq}.

\subsection{Organization of the paper}
In Section~\ref{sec-pre}, we compare different notions of two-parameter variation, establish maximal control and pointwise properties, and prove the embedding lemmas that connect variation norms to parameter-space estimates. We also record the asymptotic expansion of $\widehat{\sigma_{t,s}}(\xi)$ and the local smoothing estimates used throughout the paper. Section~\ref{sec-square-func-est} begins with an auxiliary annular estimate and then establishes three types of square function estimates for two-parameter propagators and averages over tori. In Section~\ref{sec-Lp-Lq-bounds}, we interpolate these estimates with local smoothing estimates to obtain the mixed-norm bounds for $\widetilde{V}_r(\mathcal{A})$. Section~\ref{sec-one-para} treats the two boundary variation terms $[\mathcal{A}f(x,1,\cdot)]_{r,1}$ and $[\mathcal{A}f(x,\cdot,2^{-3})]_{r,1}$, thereby completing the proofs of the sufficient parts of Theorems~\ref{thm-two-para} and \ref{thm-Lp-Lq}. It also proves the sufficient part of Theorem~\ref{thm-one-para} using the local smoothing estimates for $\mathcal{A}_1f$. Finally, Section~\ref{sec-nec-cond} establishes the necessary conditions in Theorems~\ref{thm-two-para}, \ref{thm-one-para}, and \ref{thm-Lp-Lq} through explicit constructions.

\subsection{Notation}
\begingroup
\setlength{\leftmargini}{2em}
\begin{itemize}
    \item For $x=(x_1,x_2,x_3)\in\mr^3$, we write $\bar x=(x_1,x_2)$, so that $x=(\bar x,x_3)$. We use the same convention for all other variables in $\mr^3$; for example, $\xi=(\bar{\xi},\xi_3)$.
    \item We denote the Fourier transform and its inverse by $\mathcal F$ and $\mathcal F^{-1}$, respectively. We also use $\widehat{\ }$ and $^\vee$ for the same transforms. When necessary, we indicate the variable with respect to which the Fourier transform or inverse Fourier transform is applied, for example by writing $\mathcal F_x$ or $\mathcal F_\xi^{-1}$.
    \item We write $\mathbb D:=2^{\mathbb Z}$ for the set of dyadic numbers and fix a cutoff function $\varphi\in C_c^\infty((2^{-1},2))$ such that
	\[\sum_{\rho\in\mathbb D}\varphi(s/\rho)=1,\quad s>0.\]
	For $\rho\in\mathbb D$, set $\varphi_\rho(\cdot):=\varphi(\cdot/\rho)$. If $R>0$ and $\star\in\{<,\leq,>,\geq\}$, define $\varphi_{\star R}:=\sum_{\rho\in\mathbb D:\,\rho\star R}\varphi_\rho$. The letters $\lambda,h\in\mathbb D$ denote the horizontal and vertical dyadic frequency scales, respectively, and are also used as indices of frequency-localized objects. For a function $f$ on $\mr^3$, define
	\begin{align*}
	\mathcal{F}(f_{\lambda,h})(\xi)&:=\varphi_\lambda(|\bar{\xi}|)\varphi_h(|\xi_3|)\widehat{f}(\xi),
	\\
	\mathcal{F}(f_{\leq\lambda,\leq h})(\xi)&:=\varphi_{\leq\lambda}(|\bar{\xi}|)\varphi_{\leq h}(|\xi_3|)\widehat{f}(\xi),
	\end{align*}
	and define \(f_{\lambda,\leq h}\), \(f_{\leq\lambda,h}\), \(f_{\lambda,<h}\), and \(f_{\lambda,\geq h}\) similarly. Unless otherwise indicated, all sums in $\lambda$ and $h$ are over $\mathbb D$.
    \item For nonnegative quantities $A$ and $B$, we write $A\lesssim B$, or equivalently $B\gtrsim A$, if there exists a constant $C>0$, independent of $A$ and $B$, such that $A\leq CB$. We write $A\approx B$ if both $A\lesssim B$ and $B\lesssim A$ hold.
	\item We set $\mathbb{I}:=[1,2]$ and $\mathbb{I}^{\circ}:=[0,2]$. For $a\in\mathbb R$, we denote $\mathbb{I}_a:=[2^a,2^{a+1}]$ and $\mathbb{I}_a^{\circ}:=[0,2^{a+1}]$. For $\lambda,h\in\mathbb D$ and $R>0$, we set
\begin{align*}
    \mathbb{A}_\lambda&:=\left\{\eta \in \mathbb{R}^2: \lambda/2 \leq|\eta| \leq 2\lambda\right\},
    &\mathbb{A}_R^{\circ}&:=\left\{\eta \in \mathbb{R}^2:|\eta| \leq 2R\right\},\\
    \mathbb{B}_h&:=\left\{\zeta \in \mathbb{R}: h/2 \leq|\zeta| \leq 2h\right\},
    &\mathbb{B}_R^{\circ}&:=\left\{\zeta \in \mathbb{R}:|\zeta| \leq 2R\right\}.
\end{align*}
    \item The letters $j$ and $k$ are used as integer counting indices when needed; dyadic frequency-localized objects are indexed by $\lambda$ and $h$.
    \item For a set $E$, we write $\chi_E$ for its characteristic function.
\end{itemize}
\endgroup

 \section{Two-parameter variation and analytic preliminaries}\label{sec-pre}

The full variation norm $[\mathcal A]_{r,2}$ consists of a rectangular component measuring joint oscillation, two boundary-variation terms
measuring coordinatewise variation, and a base-value term. We begin by placing its rectangular component among other natural notions of two-parameter variation.

\subsection{Comparison with other notions of two-parameter variation}\label{subsec-def-var}
In addition to the grid-like variation $\{\mathcal{A}\}_{r,2}$ introduced in the Introduction, we consider two other natural notions of two-parameter variation for $1\leq r<\infty$. The first is obtained from finite rectangular partitions of $\mathbb{J}_0$ and is defined by
\begin{align}\label{ali-def-var-par}
    \{\mathcal{A}\}_{r,2}^{par}:=\sup_{\mathscr{P}}\Big(\sum_{Q\in \mathscr{P}} |\mathcal{A}(Q)|^{r}\Big)^{\frac{1}{r}},
\end{align}
where the supremum is taken over all finite partitions $\mathscr{P}$ of $\mathbb{J}_0$ into rectangles whose sides are parallel to the coordinate axes.

The second notion, used in \cite{Dab,JRW,KLMP}, is based on ordered sequences. For $\mathbf{z}_j:=(t_j,s_j)$, we write $\mathbf{z}_j\leq \mathbf{z}_{j+1}$ if $t_j\leq t_{j+1}$ and $s_j\leq s_{j+1}$, and define
\begin{align*}
    \{\mathcal{A}\}_{r,2}^{seq}:=\sup_{L\in \mathbb{N}} \sup_{(1,2^{-3})\leq \mathbf{z}_{1}\leq\cdots\leq\mathbf{z}_{L}\leq (2,2^{-2})}\Big(\sum_{j=1}^{L-1} |\mathcal{A}(\mathbf{z}_{j+1})-\mathcal{A}(\mathbf{z}_j)|^{r}\Big)^{\frac{1}{r}}.
\end{align*}
When $r=1$, this is the Arzel\`a variation. We now compare these two notions with the variation norm used in this paper and then examine another natural, but ultimately inappropriate, definition.

After an affine rescaling of the two coordinate intervals, $\{\mathcal{A}\}_{r,2}$ and $\{\mathcal{A}\}_{r,2}^{par}$ correspond, respectively, to the $r$-variation and the controlled $r$-variation in \cite[Definition~1]{FV}.
Thus, \cite[Theorem~1 (i)--(ii)]{FV} applies to the rectangular-increment convention used here; its constant depends only on $r$ and $\epsilon$, and not on the partition or on $\mathcal{A}$. We record the resulting comparison for later use.
\begin{lemma}\label{lem-ali-rela-var-par}
    For any function $\mathcal{A}$ defined on $\mathbb{J}_0$, one has $\{\mathcal{A}\}_{1,2}=\{\mathcal{A}\}_{1,2}^{par}$. Moreover, for $r\in [1,\infty)$ and $\epsilon>0$, there exists $C_{r,\epsilon}\geq1$ such that
\begin{align}\label{ali-rela-var-par}
    {C_{r,\epsilon}}^{-1}\{\mathcal{A}\}_{r+\epsilon,2}^{par}\leq \{\mathcal{A}\}_{r,2}\leq \{\mathcal{A}\}_{r,2}^{par}.
\end{align}
\end{lemma}

As an immediate consequence of this lemma, we have
\begin{align}\label{ali-compa-seqvar}
    \{\mathcal{A}\}_{r,2}^{seq}\lesssim_{r,\epsilon} [\mathcal{A}]_{r-\epsilon,2}
\end{align}
for any $r\in (1,\infty)$ and $\epsilon\in (0,r-1)$.
Indeed, for an increasing sequence $\{\mathbf z_j\}_{j=1}^{L}\subset \mathbb{J}_0$, with $\mathbf z_j=(t_j,s_j)$, we have
\begin{align*}
    |\mathcal{A}(\mathbf{z}_{j+1})-\mathcal{A}(\mathbf{z}_j)|&\leq |\mathcal{A}([t_j,t_{j+1}]\times [s_j,s_{j+1}])|+|\mathcal{A}([1,t_j]\times [s_j,s_{j+1}])|\\
    &\qquad +|\mathcal{A}([t_j,t_{j+1}]\times [2^{-3},s_j])|+|\mathcal{A}(1,s_{j+1})-\mathcal{A}(1,s_j)|\\
    &\qquad +|\mathcal{A}(t_{j+1},2^{-3})-\mathcal{A}(t_j,2^{-3})|.
\end{align*}
In each of the three rectangular terms, the rectangles indexed by $j$ have pairwise disjoint interiors and can therefore be completed to a finite rectangular partition of $\mathbb{J}_0$. Taking the $\ell^r$ norm and using \eqref{ali-def-var-par} and \eqref{ali-rela-var-par}, together with monotonicity of one-parameter variation in the exponent for the two boundary terms, yields \eqref{ali-compa-seqvar}.

In particular, if $\mathcal{A}_1(t):=\mathcal{A}(t,2^{-3}t)$, then for every fixed $\eta>0$ and $r\geq1$,
\begin{align}\label{ali-diagonal-var-comparison}
    [\mathcal{A}_1]_{r+\eta,1}\lesssim_{r,\eta}[\mathcal{A}]_{r,2}.
\end{align}
For $r>1$, this follows from \eqref{ali-compa-seqvar} by using the exponent $r+\eta$ and the loss $\eta$. For $r=1$, the displayed decomposition above, \cite[Theorem~1 (i)]{FV}, and monotonicity of one-parameter variation give the same conclusion. The dependence of the constant on the fixed pair $(r,\eta)$ is harmless in the limiting argument in Section~\ref{sec-nec-cond}.

However, for any $1\leq r_1, r_2<\infty$, the reverse inequality
\[[\mathcal A]_{r_1,2}\lesssim\{\mathcal A\}_{r_2,2}^{\mathrm{seq}}\]
does not hold. We present a simple example for which $\{\mathcal A\}_{1,2}^{\mathrm{seq}}<\infty$ while $\{\mathcal A\}_{r,2}=\infty$ for every $1\leq r<\infty$. After an affine change of variables, it suffices to work on $[0,1]^2$. Let $\mathcal A(t,s)$ be the characteristic function of the set $\{(t,s)\in[0,1]^2:t+s\geq 1\}$. A direct verification gives $\{\mathcal A\}_{1,2}^{\mathrm{seq}}=1$. On the other hand, for each $m\in\mathbb N$, let
\[I_j:=\Big[\frac{j}{m}, \frac{j+1}{m}\Big], \qquad J_k:=\Big[\frac{k}{m}, \frac{k+1}{m}\Big], \qquad 0\leq j,k\leq m-1.\]
Then $\mathcal{A}(I_j\times J_k)=-1$ when $j+k=m-1$. There are exactly $m$ such pairs $(j,k)$, which implies $\{\mathcal A\}_{r,2}\geq m^{1/r}$. The result now follows by letting $m\to \infty$.

For $1\leq r<\infty$, one might attempt to define $\{\mathcal{A}\}_{r,2}^{osc}$ by
\begin{align*}
		\{\mathcal{A}\}_{r,2}^{osc}:=\sup_{\mathscr{P}}\Big(\sum_{Q\in \mathscr{P}}\max_{c_{Q}} |\mathcal{A}({c_{Q}})-\mathcal{A}({l_{Q}})|^{r}\Big)^{\frac{1}{r}},
\end{align*}
where $c_{Q}$ denotes a corner of $Q$ and $l_{Q}$ denotes the bottom-left corner of $Q$; one could equally choose any fixed vertex in place of the bottom-left corner. However, this definition does not provide a suitable notion of two-parameter
	variation in natural continuous settings (see Lemma~\ref{lem-osc-continuous-constant} below). Roughly speaking, the reason is
	that, after subdividing a rectangle into many thin horizontal or vertical
	strips, the same oscillation along one coordinate direction is counted
	repeatedly, which forces divergence unless $\mathcal{A}$ is constant.

\begin{lem}\label{lem-osc-continuous-constant}
	Let $1\leq r<\infty$. If $\mathcal{A}\in C(\mathbb{J}_0)$, then
	\[
	\{\mathcal{A}\}_{r,2}^{osc}<\infty
	\quad\Longleftrightarrow\quad
	\mathcal{A}\ \text{is constant on }\mathbb{J}_0.
	\]
\end{lem}

\begin{proof}
	The implication from right to left is immediate. For the converse, assume that $\mathcal{A}$ is continuous and nonconstant on $\mathbb{J}_0$ and that $\{\mathcal{A}\}_{r,2}^{osc}<\infty$. Then there exist $(t_{1},s_{1})$, $(t_{2},s_{2})\in \mathbb{J}_0$ such that $\mathcal{A}(t_{1},s_{1})\neq \mathcal{A}(t_{2},s_{2})$. Hence either $\mathcal{A}(t_{1},s_{1})\neq \mathcal{A}(t_{2},s_{1})$ or $\mathcal{A}(t_{2},s_{1})\neq \mathcal{A}(t_{2},s_{2})$.

	By symmetry, it suffices to treat the first case. Interchanging $t_1$ and $t_2$ if necessary, we may assume that $t_1<t_2$. Set
	\[F(s):=\mathcal{A}(t_{2},s)-\mathcal{A}(t_{1},s).\]
	Then $F$ is continuous and $F(s_{1})\neq 0$. Thus there exist $\eta>0$ and an interval
		$I=[\alpha,\beta]\subset \mathbb{I}_{-3}$ containing $s_{1}$ such that $|F(s)|\ge \eta$ for all $s\in I$. For a fixed $N\in \mathbb{N}$, we consider the partition of
	$[t_{1},t_{2}]\times I$ into the horizontal strips
	\[Q_{j}:=[t_{1},t_{2}]\times [y_{j-1},y_{j}],\qquad 1\le j\le N,\]
	where $\alpha=y_{0}<y_{1}<\cdots<y_{N}=\beta$. These strips can be completed to a finite rectangular partition of $\mathbb{J}_0$. Consequently,
	\[\max_{c_{Q_{j}}}\bigl|\mathcal{A}(c_{Q_{j}})-\mathcal{A}(l_{Q_{j}})\bigr|\ge
	\bigl|\mathcal{A}(t_{2},y_{j-1})-\mathcal{A}(t_{1},y_{j-1})\bigr|=|F(y_{j-1})|\ge \eta.\]
	It follows that
	\[\{\mathcal{A}\}_{r,2}^{osc}\ge
	\Bigg(
	\sum_{j=1}^{N}
	\max_{c_{Q_{j}}}
	\bigl|\mathcal{A}(c_{Q_{j}})-\mathcal{A}(l_{Q_{j}})\bigr|^{r}
	\Bigg)^{1/r}\ge N^{1/r}\eta.\]
	Since $N$ is arbitrary, we obtain $\{\mathcal{A}\}_{r,2}^{osc}=\infty$, contradicting the assumed finiteness. Therefore $\mathcal{A}$ must be constant on $\mathbb{J}_0$.
\end{proof}

\begin{remark}
    For the torus averages considered here, $\mathcal A f(x,\cdot,\cdot)$ is continuous on $\mathbb{J}_0$ for every $f\in\mathcal S(\mathbb R^3)$ and $x\in\mathbb R^3$. Lemma~\ref{lem-osc-continuous-constant} shows that its corner-oscillation quantity is finite only when this parameter function is constant. Thus, this quantity is not suitable for measuring parameter oscillation in the present setting.
\end{remark}

\subsection{Maximal control and pointwise properties}
Throughout this subsection, let $1\leq r<\infty$. We first record that the two-parameter maximal function is controlled by the associated two-parameter variation norm.

\begin{prop}\label{prop-Vr-domina-maximal}
	\begin{align*}
		\sup_{(t,s)\in \mathbb{J}_0}|\mathcal{A}(t,s)|\leq [\mathcal{A}]_{r,2}.
	\end{align*}
\end{prop}
\begin{proof}
	Note that
		\begin{align*}	|\mathcal{A}(t,s)|
		    	&\leq |\mathcal{A}(t,s)-\mathcal{A}(1,s)+\mathcal{A}(1,2^{-3})-\mathcal{A}(t,2^{-3})|  \\
			&\qquad +|\mathcal{A}(1,s)-\mathcal{A}(1,2^{-3})|+|\mathcal{A}(t,2^{-3})|
			\end{align*}
		for any $(t,s)\in \mathbb{J}_0$. Thus, we have
		    \begin{align*}
		  |\mathcal{A}(t,s)|  \leq \{\mathcal{A}\}_{r,2}+[\mathcal{A}(1,\cdot)]_{r,1}+[\mathcal{A}(\cdot,2^{-3})]_{r,1}+
		    	|\mathcal{A}(1,2^{-3})|
		=[\mathcal{A}]_{r,2},
		    \end{align*}
	   from which the desired inequality follows.
\end{proof}
In addition, the one-dimensional variation along any horizontal or vertical line is controlled by $[\mathcal{A}]_{r,2}$.

\begin{prop}\label{prop-Vr-domina-line}
	For any $(t,s)\in \mathbb{J}_0$, we have
    \begin{align*}
        [\mathcal{A}(t,\cdot)]_{r,1}\leq [\mathcal{A}]_{r,2}
        \quad\text{and}\quad
        [\mathcal{A}(\cdot,s)]_{r,1}\leq [\mathcal{A}]_{r,2}.
    \end{align*}
\end{prop}
\begin{proof}
	It suffices to estimate $[\mathcal{A}(t,\cdot)]_{r,1}$. For an increasing sequence $\{s_k\}_{k=1}^L\subset \mathbb{I}_{-3}$, the rectangular-partition argument used in the proof of \eqref{ali-compa-seqvar} gives
	   \begin{align*}
	   	 & \Big(\sum_{k=1}^{L-1}|\mathcal{A}(t,s_{k+1})-\mathcal{A}(t,s_{k})|^r\Big)^{1/r} \leq \Big(\sum_{k=1}^{L-1}|\mathcal{A}([1,t]\times[s_k,s_{k+1}])|^r\Big)^{\frac1r}\\
	   	  &\qquad \qquad \qquad +\Big(\sum_{k=1}^{L-1}|\mathcal{A}(1,s_{k+1})-\mathcal{A}(1,s_{k})|^r\Big)^{1/r} \leq [\mathcal{A}]_{r,2}.
	   \end{align*}
    Taking the supremum over all increasing sequences $\{s_k\}_k$ completes the proof.
\end{proof}

If $[\mathcal{A}]_{r,2}<\infty$ with $1\le r<\infty$, then Proposition~\ref{prop-Vr-domina-line} and the corresponding one-parameter variation property imply that the limits
\[
\lim_{s\downarrow 2^{-3}}\mathcal{A}(t_0,s)
\qquad\text{and}\qquad
\lim_{t\downarrow 1}\mathcal{A}(t,s_0)
\]
exist for every $(t_0,s_0)\in \mathbb{J}_0$. In fact, finiteness of this two-parameter variation also guarantees the existence of two-parameter limits in every open quadrant. The following proposition goes one step further: if $[\mathcal{A}]_{r,2}<\infty$, then the discontinuity set of $\mathcal{A}$ is contained in a countable union of straight lines parallel to the axes. This is a natural two-parameter analogue of the fact that a one-variable function with finite $r$-variation has at most countably many discontinuities. Arguments of a similar flavor appear in \cite[Theorem~1]{DW}, although the setting there is different from ours. We include the proof for the reader's convenience.

\begin{prop}\label{prop-Vr-continuity}
	If $[\mathcal{A}]_{r,2}<\infty$ with $1\leq r<\infty$, then for any $(t_{0},s_{0})\in \mathbb{J}_0$, the following quadrant limits exist whenever the corresponding quadrant is nonempty:
	    \begin{align*}
	    	\lim_{t\rightarrow t_{0}^{+},s\rightarrow s_{0}^{+}}\mathcal{A}(t,s), \quad \lim_{t\rightarrow t_{0}^{+},s\rightarrow s_{0}^{-}}\mathcal{A}(t,s), \quad \lim_{t\rightarrow t_{0}^{-},s\rightarrow s_{0}^{+}}\mathcal{A}(t,s), \quad \lim_{t\rightarrow t_{0}^{-},s\rightarrow s_{0}^{-}}\mathcal{A}(t,s).
	    \end{align*}
Furthermore, there exist countable sets $X\subset\mathbb{I}$ and $Y\subset \mathbb{I}_{-3}$ such that $\mathcal{A}$ is continuous at $(t_0,s_0)$ whenever $t_0\notin X$ and $s_0\notin Y$.
\end{prop}
\begin{proof}
	For each of the four quadrant limits, let $R$ denote the corresponding quadrant rectangle.
	The rectangular variation of $\mathcal A$ over $R$ is bounded by $\{\mathcal A\}_{r,2}$, since each pair of coordinate partitions of $R$ can be extended to a pair of partitions of the coordinate intervals of $\mathbb J_0$. Proposition~\ref{prop-Vr-domina-line} controls the two boundary variations associated with $R$, and Proposition~\ref{prop-Vr-domina-maximal} controls its base value. Hence the full variation of $\mathcal A$ over $R$ is bounded by a constant multiple of $[\mathcal A]_{r,2}$.
	After affinely rescaling the coordinate intervals of $R$ and reversing coordinate directions when necessary, it suffices to treat the case $(t_{0},s_{0})=(1,2^{-3})$; in this case, only the first limit is well-defined.
	Suppose that the limit
	\begin{align*}
		\lim_{t\rightarrow t_{0}^{+},s\rightarrow s_{0}^{+}}\mathcal{A}(t,s)
	\end{align*}
	 does not exist. Then there exists $\epsilon>0$ such that, for every $\delta>0$, the oscillation of $\mathcal{A}(t,s)$ in the square $(t_{0},t_{0}+\delta)\times (s_{0},s_{0}+\delta)$ is greater than $\epsilon$. Choose any $t_{1}\in (1,2)$ and $s_{1}\in (2^{-3},2^{-2})$. Note that, by Proposition~\ref{prop-Vr-domina-line},
	     \begin{align*}
	     	[\mathcal{A}(t_{1},\cdot)]_{r,1}+[\mathcal{A}(\cdot,s_{1})]_{r,1}\lesssim [\mathcal{A}]_{r,2}<\infty,
	     \end{align*}
	 which implies that the limits
	 \[
	     	\lim_{s\rightarrow s_{0}^{+}}\mathcal{A}(t_1,s)
	     	\quad\text{and}\quad
	     	\lim_{t\rightarrow t_0^{+}}\mathcal{A}(t,s_1)
	 \]
	 exist. Thus, we can choose a suitably small $\delta_{1}>0$ such that the oscillations of $\mathcal{A}(t_1,\cdot)$ and $\mathcal{A}(\cdot,s_1)$ on $(s_0,s_0+\delta_1)$ and $(t_0,t_0+\delta_1)$, respectively, are both less than $\epsilon/8$. Now choose two points $(x_1,y_1), (x_2,y_2)\in (t_{0},t_0+\delta_1)\times (s_0,s_0+\delta_1)$ such that
	 \begin{align*}
	 	|\mathcal{A}(x_1,y_1)-\mathcal{A}(x_2,y_2)|>\epsilon/2.
	 \end{align*}
		 Let $A_{1}:=\mathcal{A}(t_1,s_1)-\mathcal{A}(t_1,y_1)-\mathcal{A}(x_1,s_1)+\mathcal{A}(x_1,y_1)$ and $A_2:=\mathcal{A}(t_1,s_1)-\mathcal{A}(t_1,y_2)-\mathcal{A}(x_2,s_1)+\mathcal{A}(x_2,y_2)$. Then
		 \begin{align*}
		 	|A_1-A_2|
		 	&\geq |\mathcal{A}(x_1,y_1)-\mathcal{A}(x_2,y_2)|
		 	-|\mathcal{A}(t_1,y_2)-\mathcal{A}(t_1,y_1)|\\
		 	&\quad-|\mathcal{A}(x_2,s_1)-\mathcal{A}(x_1,s_1)| \\
		 	&\geq \epsilon/2-\epsilon/8-\epsilon/8=\epsilon/4,
		 \end{align*}
	 which implies that $|A_1|\geq \epsilon/8$ or $|A_{2}|\geq \epsilon/8$. Thus, we obtain a rectangle $Q_{1}$ such that $|\mathcal{A}(Q_{1})|\geq \epsilon/8$. Suppose that $Q_1,\ldots,Q_m$ have been constructed. Choose $(t_{m+1},s_{m+1})$ sufficiently close to $(t_0,s_0)$ so that $[t_0,t_{m+1}]\times[s_0,s_{m+1}]$ is disjoint from $Q_1,\ldots,Q_m$. We then choose $\delta_{m+1}>0$ so that the corresponding oscillation square lies in $(t_0,t_{m+1})\times(s_0,s_{m+1})$. Repeating the preceding argument produces a rectangle $Q_{m+1}$ in this smaller rectangle such that $|\mathcal A(Q_{m+1})|\geq\epsilon/8$. Hence, for every $N\in\mathbb N$, we obtain pairwise disjoint rectangles $\{Q_k\}_{k=1}^N$ in $\mathbb{J}_0$ satisfying $|\mathcal A(Q_k)|\geq\epsilon/8$.
	 Fix $r<r_1<\infty$. Since these rectangles can be completed to a finite rectangular partition of $\mathbb{J}_0$, Lemma~\ref{lem-ali-rela-var-par} gives
\begin{align*}
    \frac{\epsilon}{8}N^{1/r_1}
    &\leq \Big(\sum_{k=1}^N|\mathcal A(Q_k)|^{r_1}\Big)^{1/r_1}
    \leq \{\mathcal A\}_{r_1,2}^{par}
    \lesssim_{r,r_1}\{\mathcal A\}_{r,2}
    \leq [\mathcal A]_{r,2}.
\end{align*}
Letting $N\to\infty$ gives a contradiction. Therefore, the limit $\lim_{t\to t_0^+,s\to s_0^+}\mathcal{A}(t,s)$ exists, and the first part of the proposition is proved.

	To establish the second part, it suffices to show that, for each $k\geq 1$, every point at which $\mathcal{A}$ has oscillation exceeding $1/k$ lies on one of finitely many lines parallel to the coordinate axes. If this fails, then there exist infinitely many points $(t_i,s_i)$ with distinct $t$-coordinates and distinct $s$-coordinates such that the oscillation of $\mathcal{A}$ at each of these points is greater than $1/k$. For any $N\in\mathbb{N}$, we can find rectangles $\{Q_i\}_{i=1}^N$ centered at $(t_i,s_i)$ such that their projections onto both the $t$-axis and the $s$-axis are pairwise disjoint. Then, for each fixed $1\le i\le N$, there exist points $(\alpha_i,\beta_i),(\gamma_i,\delta_i)\in Q_i$ such that $\alpha_i\le \gamma_i$ and
\[
|\mathcal A(\alpha_i,\beta_i)-\mathcal A(\gamma_i,\delta_i)|>1/k.
\]
If $\beta_i\le \delta_i$, then, by the triangle inequality,
\begin{align*}
	|\mathcal A(\alpha_i,\beta_i)-\mathcal A(\gamma_i,\delta_i)|
	&\le |\mathcal A([\alpha_i,\gamma_i]\times [2^{-3},\beta_i])|
	+|\mathcal A([1,\gamma_i]\times [\beta_i,\delta_i])| \\*
	&\quad +|\mathcal A(\alpha_i,2^{-3})-\mathcal A(\gamma_i,2^{-3})|
	+|\mathcal A(1,\beta_i)-\mathcal A(1,\delta_i)|.
\end{align*}
Similarly, if $\delta_i<\beta_i$, then
\begin{align*}
	|\mathcal A(\alpha_i,\beta_i)-\mathcal A(\gamma_i,\delta_i)|
	&\le |\mathcal A([\alpha_i,\gamma_i]\times [2^{-3},\delta_i])|
	+|\mathcal A([1,\alpha_i]\times [\delta_i,\beta_i])| \\*
	&\quad +|\mathcal A(\alpha_i,2^{-3})-\mathcal A(\gamma_i,2^{-3})|
	+|\mathcal A(1,\beta_i)-\mathcal A(1,\delta_i)|.
\end{align*}
In either case, for any $r<r_1<\infty$, Minkowski's inequality
and Lemma~\ref{lem-ali-rela-var-par} yield
\begin{align*}
	\frac{N^{1/r_1}}{k}
	&\le \Big(\sum_{i=1}^N |\mathcal A(\alpha_i,\beta_i)-\mathcal A(\gamma_i,\delta_i)|^{r_1}\Big)^{1/r_1} \\
	&\le 2\{\mathcal A\}_{r_1,2}^{par}
	+[\mathcal A(1,\cdot)]_{r_1,1}
	+[\mathcal A(\cdot,2^{-3})]_{r_1,1} \\
	&\lesssim_{r,r_1} [\mathcal A]_{r,2}.
\end{align*}
Letting $N\to\infty$ yields a contradiction, and the proof is complete.
\end{proof}
\begin{remark}
	Even if $[\mathcal{A}]_{r,2}<\infty$, the full limit $\lim_{(t,s)\rightarrow (t_0,s_0)}\mathcal{A}(t,s)$ need not exist at a point $(t_0,s_0)\in \mathbb{J}_0$. For example, let $\mathcal{A}$ be the characteristic function of the rectangle $Q$ with vertices
	\begin{align*}
		K_1:=(\tfrac{5}{4},\tfrac{5}{32}), \ K_2:=(\tfrac{7}{4},\tfrac{5}{32}), \ K_3:=(\tfrac{7}{4},\tfrac{7}{32}), \ K_4:=(\tfrac{5}{4},\tfrac{7}{32}).
	\end{align*}
		Then $[\mathcal A]_{r,2}<\infty$, whereas $\mathcal A$ is
		discontinuous on the boundary of $Q$; see Figure~\ref{Fig.def}.
	\begin{figure}[t]
		\centering
		\begin{tikzpicture}[scale=0.6]
			\draw  (0,0) -- (6,0) -- (6,3) -- (0,3) -- (0,0);
			\node[right] at (6,3) {$(2,2^{-2})$};
			\node[right] at (6,0) {$(2,2^{-3})$};
			\node[left] at (0,0) {$(1,2^{-3})$};
			\node[left] at (0,3) {$(1,2^{-2})$};
			\fill[pattern=north east lines] (3/2,3/4) -- (9/2,3/4) -- (9/2,9/4) -- (3/2,9/4) -- (3/2,3/4);
			\node[left] at (3/2,3/4) {$K_1$};
			\node[right] at (9/2,3/4) {$K_2$};
			\node[right] at (9/2,9/4) {$K_3$};
			\node[left] at (3/2,9/4) {$K_4$};
		\end{tikzpicture}
			\caption{A rectangle example showing that finite two-parameter variation need not imply full continuity.}
			\label{Fig.def}
	\end{figure}
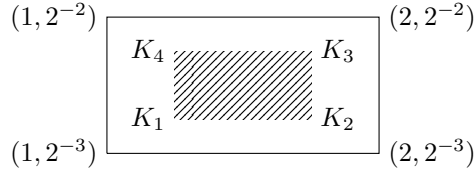
\end{remark}
\begin{remark}
   If the hypothesis $[\mathcal A]_{r,2}<\infty$ in
Proposition~\ref{prop-Vr-continuity} is replaced by
$\{\mathcal A\}_{r,2}^{\mathrm{seq}}<\infty$, then the same argument shows that the limits
\[
\lim_{t\to t_0^+,\, s\to s_0^+}\mathcal A(t,s)
\quad\text{and}\quad
\lim_{t\to t_0^-,\, s\to s_0^-}\mathcal A(t,s)
\]
exist whenever they are well-defined. However, the corresponding limits in the other two coordinate quadrants need not exist, and the set of discontinuities need not have the same axis-parallel structure as in Proposition~\ref{prop-Vr-continuity}. Indeed, let
\[
\mathcal A(t,s):=\sum_{n=1}^\infty 2^{-n}\chi_{\{t^2+8s\ge 2+1/n\}}.
\]
Then the discontinuity set of $\mathcal A$ is a countable union of curves in the plane. On the other hand, since $\mathcal A(\mathbf z_j)\le \mathcal A(\mathbf z_{j+1})$ whenever $\mathbf z_j\le \mathbf z_{j+1}$, we have
\[
\{\mathcal A\}_{1,2}^{\mathrm{seq}}
=\mathcal A(2,2^{-2})-\mathcal A(1,2^{-3})
=1.
\]
\end{remark}

\subsection{Embedding lemmas for variation norms}
We now establish embedding lemmas for the one- and two-parameter variation norms. Recall that
    \begin{align*}
    	[\mathcal{A}]_{r,2}=\{\mathcal{A}\}_{r,2}+[\mathcal{A}(1,\cdot)]_{r,1}+[\mathcal{A}(\cdot,2^{-3})]_{r,1}+|\mathcal{A}(1,2^{-3})|.
    \end{align*}
For the one-parameter variation seminorm, we have the following embedding lemma.
\begin{lem}\label{lem-embed-1-var}
	Suppose that $\mathbb I$ is a closed interval in $(0,\infty)$ and $a\in C^1(\mathbb I)$. Then, for any $r\in (1,\infty)$ and any $N>0$,
	    \begin{align*}
	    	[a]_{r,1}\lesssim N^{1/r}\|a\|_{L^r(\mathbb I)}+N^{1/r-1}\|a^{\prime}\|_{L^r(\mathbb I)}.
	    \end{align*}
\end{lem}
\begin{proof}
	We recall the following inequality established in \cite[p.~6729]{JSW}:
	    \begin{align*}
	    	[a]_{r,1}\lesssim \|a\|_{L^r(\mathbb I)}^{1/r^{\prime}}\|a^{\prime}\|_{L^r(\mathbb I)}^{1/r}.
	    \end{align*}
	Writing the right-hand side as $(N^{1/r}\|a\|_{L^r(\mathbb I)})^{1/r'}(N^{1/r-1}\|a^{\prime}\|_{L^r(\mathbb I)})^{1/r}$ and applying Young's inequality proves the claim.
\end{proof}

Next, we prove the following embedding lemma for $\{\mathcal{A}\}_{r,2}$.
\begin{lem}\label{lem-embedding}
	Suppose $\mathcal{A}\in C^{2}(\mathbb{J}_0)$. Then for any $r\in (1,\infty)$ and any $N_1,N_2>0$,
	\begin{align*}
		\{\mathcal{A}\}_{r,2}&\lesssim (N_1N_2)^{1/r}\|\mathcal{A}\|_{L^r(\mathbb{J}_0)}+N_{1}^{1/r-1}N_2^{1/r}\|\partial_{t}\mathcal{A}\|_{L^r(\mathbb{J}_0)}\\
		&\qquad +N_{1}^{1/r}N_2^{1/r-1}\|\partial_{s}\mathcal{A}\|_{L^r(\mathbb{J}_0)} +\left(N_{1}N_2\right)^{1/r-1}\|\partial_{t}\partial_{s}\mathcal{A}\|_{L^r(\mathbb{J}_0)}.
	\end{align*}
\end{lem}
\begin{proof}
	Assume that $[1,2]=\bigcup_{j=1}^{m}[t_{j},t_{j+1}]$ and $[2^{-3},2^{-2}]=\bigcup_{k=1}^{n}[s_{k},s_{k+1}]$. For $s\in[2^{-3},2^{-2}]$, let $F_{j}(s):=\mathcal{A}(t_{j+1},s)-\mathcal{A}(t_j,s)$. Note that
	\[ \Big(\sum_{j=1}^{m}\sum_{k=1}^{n}|\mathcal{A}([t_j,t_{j+1}]\times [s_k,s_{k+1}])|^r\Big)^{1/r}=\Big(\sum_{j=1}^{m}\sum_{k=1}^{n}|F_{j}(s_{k+1})-F_{j}(s_k)|^r\Big)^{1/r}.\]

	Thus, we have
	    \begin{align}\label{eq-two-var}
		    	\Big(\sum_{j=1}^{m}\sum_{k=1}^{n}|\mathcal{A}([t_j,t_{j+1}]\times [s_k,s_{k+1}])|^r\Big)^{1/r}&\leq \Big(\sum_{j=1}^{m}[F_j]_{r,1}^{r}\Big)^{1/r}.
	    \end{align}
	 By Lemma~\ref{lem-embed-1-var}, for any $N_{2}>0$,
	     \begin{align*}
	     	[F_j]_{r,1}^r\lesssim N_2\int_{2^{-3}}^{2^{-2}}|F_{j}(s)|^r\,ds+N_2^{1-r}\int_{2^{-3}}^{2^{-2}}|\partial_s F_{j}(s)|^r\,ds.
	     \end{align*}
	This gives
\[  	\Big(\sum_{j=1}^{m}[F_j]_{r,1}^{r}\Big)^{1/r}
			    	\lesssim N_{2}^{1/r}I_{1}^{1/r}+N_{2}^{1/r-1}I_{2}^{1/r}, \]
where
\[
			I_1=\sum_{j=1}^{m}\int_{2^{-3}}^{2^{-2}}|F_{j}(s)|^r\,ds, \quad I_2=\sum_{j=1}^{m}\int_{2^{-3}}^{2^{-2}}|\partial_s F_{j}(s)|^r\,ds.
	     	\]

	   It remains to estimate $I_{1}$ and $I_{2}$.
	   Applying Lemma~\ref{lem-embed-1-var} again to $t\mapsto\mathcal A(t,s)$, we obtain, for any $N_{1}>0$,
	     \begin{align*}
	     	I_{1}
	     	\leq \int_{2^{-3}}^{2^{-2}}[\mathcal{A}(\cdot,s)]_{r,1}^{r}\,ds
	     	\lesssim N_1 \int_{\mathbb{J}_0}|\mathcal{A}(t,s)|^r\,dt\,ds+N_1^{1-r}\int_{\mathbb{J}_0}|\partial_{t}\mathcal{A}(t,s)|^{r}\,dt\,ds.
	     \end{align*}
	 Similarly, we deduce that
	      \begin{align*}
	      	I_{2}\lesssim N_1 \int_{\mathbb{J}_0}|\partial_{s}\mathcal{A}(t,s)
            |^r\,dt\,ds+N_1^{1-r}\int_{\mathbb{J}_0}|\partial_{t}\partial_{s}\mathcal{A}(t,s)|^{r}\,dt\,ds.
	      \end{align*}

	  Combining the above estimates and taking the supremum over all product partitions completes the proof.
\end{proof}

\subsection{\texorpdfstring{Asymptotic expansion of $\widehat{\sigma_{t,s}}(\xi)$}{Asymptotic expansion of the Fourier transform of d sigma t s}}\label{subsect-asymptotic}
We record an asymptotic expansion of the Fourier transform of $\sigma_{t,s}$, uniformly for $(t,s)\in\mathbb J_0$. Note that
\begin{align}\label{ali-fourier-sigma-ts}
    \widehat{\sigma_{t,s}}(\xi):=\int_{0}^{2\pi}e^{-is \xi_3\sin \theta} \widehat{\mu}\big((t+s\cos \theta)\bar{\xi}\big)\,d\theta,
\end{align}
where $\mu$ denotes arc-length measure on the unit circle. The Fourier transform $\widehat{\mu}$ has the following asymptotic expansion; see, for example, \cite[pp.~338,~347]{Ste1993}:
\begin{align}\label{ali-symp-mu}
    \widehat{\mu}(\bar\xi)=b_{0}(|\bar{\xi}|)+\sum_{\pm}e^{\pm i|\bar{\xi}|}b^{\pm}(|\bar{\xi}|)
\end{align}
where $b_0\in C_c^{\infty}(\mr)$ is supported in $\{r\in \mr:|r|\leq 2\}$, and $b^{\pm}$ are supported in $\{r\in \mr:|r|\geq 1\}$. Moreover, for any integer $N\geq 0$,
\[
\left|\left(\frac{d}{dr}\right)^Nb^{\pm}(r)\right|\leq C_N(1+|r|)^{-N-\frac12}.
\]

Combining \eqref{ali-fourier-sigma-ts} and \eqref{ali-symp-mu} gives
\begin{align*} \widehat{\sigma_{t,s}}(\xi)&=\int_0^{2\pi}e^{-is\xi_3\sin \theta}b_0\big((t+s\cos \theta)|\bar{\xi}|\big)\,d\theta\notag \\
    &\quad +\sum_{\pm}\int_0^{2\pi}e^{-is\xi_3\sin \theta\pm i(t+s\cos \theta)|\bar{\xi}|}b^{\pm}\big((t+s\cos\theta)|\bar{\xi}|\big)\,d\theta.
    \end{align*}
Thus, by the stationary phase method
(see, for example, \cite[Theorem~7.7.5]{Hor}), we obtain
\begin{equation}\label{ali-expansion-sigmats}
\begin{aligned}
	     \widehat{\sigma_{t,s}}(\xi)=B_{0,0}(\xi,t,s)+\sum_{\pm}&e^{\pm is |\xi_3|}B_{0,1}^{\pm}(\xi,t,s)\\ &+\sum_{\kappa,\kappa'=\pm}e^{\kappa it|\bar{\xi}|}e^{\kappa' is|\xi|}B_{1,1}^{\kappa,\kappa'}(\xi,t,s).
\end{aligned}
\end{equation}
Here, the symbols $B_{0,0}$, $B_{0,1}^{\pm}$, and $B_{1,1}^{\kappa,\kappa'}$ are supported in
\[
\{|\bar{\xi}|\le 2^3,\ |\xi_3|\le 2^3\},\
\{|\bar{\xi}|\le 2^3,\ |\xi_3|>2^2\},\
\{|\bar{\xi}|>2^{-2}\},
\]
respectively. Moreover, for any multi-indices $\gamma$ and $\alpha$ and any $(t,s)\in\mathbb{J}_0$,
\begin{equation}
\label{asm-B-symbol}
\begin{aligned}
    &|\partial_{t,s}^{\gamma}\partial_{\xi}^{\alpha}B_{0,0}(\xi,t,s)|+|\partial_{t,s}^{\gamma}\partial_{\xi}^{\alpha}B_{0,1}^{\pm}(\xi,t,s)|+|\partial_{t,s}^{\gamma}\partial_{\xi}^{\alpha}B_{1,1}^{\kappa,\kappa'}(\xi,t,s)|\\[2pt]
    &\qquad\qquad \leq C_{\gamma,\alpha} (1+|\bar{\xi}|)^{-|(\alpha_1,\alpha_2)|-\frac12}(1+|\xi|)^{-\alpha_3-\frac12}.
\end{aligned}
\end{equation}
The constants in \eqref{asm-B-symbol} are uniform on $\mathbb{J}_0$. Indeed, since
$\tfrac34\leq t+s\cos\theta\leq\tfrac94$ and
$\tfrac18\leq s\leq\tfrac14$, the amplitudes stay in a fixed symbol
class and the critical points in the high-frequency pieces are uniformly
nondegenerate. The symbols in \eqref{ali-expansion-sigmats} include the smooth remainders in the parameter-dependent stationary phase decomposition, so the displayed identity is exact. Equivalently, a finite expansion may be taken to arbitrarily high order, with the remainder terms decaying faster than any prescribed power after any fixed finite number of $t$, $s$, and $\xi$ derivatives. In Sections~\ref{sec-square-func-est}--\ref{sec-one-para}, we use at most one derivative in each parameter; hence \eqref{asm-B-symbol} applies uniformly to every differentiated multiplier used below.

\subsection{Local smoothing estimates for propagators and averages}
We record below the local smoothing estimates used in the remainder of the paper.
We continue to use $\lambda$ and $h$ for the horizontal and vertical dyadic frequency scales introduced above.

We begin with the following estimates for the two-dimensional wave propagator $\mathcal{W}_{\pm}$, obtained by interpolating two-dimensional local smoothing estimates with standard fixed-time and endpoint estimates; see \cite{GWZ,BORSS,LL} and, for related results, \cite{SS,Lee03}.

\begin{lemma}\label{lem-one-qq-p}
    Let $1\leq p\leq q \leq \infty$, $1/p+1/q\leq 1$, and $\lambda\geq2$. Suppose $\supp \widehat{g}\subset \mathbb{A}_\lambda$. Then, for any $\epsilon>0$, we have
    \begin{align*}
		  	\|\mathcal{W}_{\pm}g\|_{L^{q}(\mathbb{R}^{2}\times [1,2])}\lesssim
                \begin{cases}
                    \lambda^{\frac{1}{2}+\frac1p-\frac3q+\epsilon}\|g\|_{\Lp(\mr^2)}, & \frac{1}{p}+\frac{3}{q}\leq 1,
                    \\[2pt]
                    \lambda^{\frac{3}{2p}-\frac{3}{2q}+\epsilon}\|g\|_{\Lp(\mr^2)}, & \frac{1}{p}+\frac{3}{q}>1.
                \end{cases}
	\end{align*}
\end{lemma}

We also record the corresponding mixed-norm estimate.
\begin{lem}\label{lem-one-qp-p}
    Let $2\leq p\leq q\leq\infty$ and $\lambda\geq2$. Suppose $\supp \widehat{g}\subset \mathbb{A}_\lambda$. Then, for any $\epsilon>0$, we have
        \begin{align*}
		  	\|\mathcal{W}_{\pm}g\|_{L^{q}(\mathbb{R}^{2};L^p([1,2]))}\lesssim
                \begin{cases}
                    \lambda^{-\frac{2}{q}+\frac{1}{2}+\epsilon}\|g\|_{\Lp(\mr^2)}, & \frac{1}{p}+\frac{3}{q}\leq 1,
                    \\[2pt]
                    \lambda^{\frac{1}{2p}-\frac{1}{2q}+\epsilon}\|g\|_{\Lp(\mr^2)}, & \frac{1}{p}+\frac{3}{q}>1.
                \end{cases}
		  \end{align*}
\end{lem}
\begin{proof}
    This follows by interpolation among the cases $(p,q)=(2,2),$ $(4,4),$ $(\infty,\infty),$ $(2,\infty)$, and $(2,6)$. The first three cases follow from Lemma~\ref{lem-one-qq-p}. The remaining cases $(p,q)=(2,6)$ and $(2,\infty)$ follow from the Stein--Tomas restriction theorem and Plancherel's theorem; see \cite[Lemma~2.4]{BORSS} and \cite[Section~2]{LL} for related arguments.
\end{proof}

Recall that our two-parameter average is defined by
\[
\mathcal{A}f(x,t,s)=f\ast \sigma_{t,s}(x).
\]
The following lemma treats the fully low-frequency and
low-horizontal-frequency regimes.

\begin{lem}\label{lem-A1-A2}
   Let $\alpha,\beta\in \{0,1\}$ and $1\leq p\leq q \leq \infty$. The following hold:
     \begingroup
     \addtolength{\leftmargini}{-2em}
     \begin{enumerate}
        \item[(a)] If $\supp \widehat{f}\subset \mathbb{A}_{2^3}^{\circ} \times \mathbb{B}_{2^3}^{\circ}$, then for $(t,s)\in \mathbb{J}_0$,
		     \begin{align*}
			    \|\partial_t^{\alpha} \partial_s^{\beta}\mathcal{A}f(\cdot,t,s)\|_{L^{q}(\mathbb{R}^{3})}
			    &\leq C_{\alpha,\beta} \|f\|_{\Lp(\mr^3)},\\[2pt]
		              \|\partial_t^{\alpha}\mathcal{A}_1f(\cdot,t)\|_{L^{q}(\mathbb{R}^{3})}
		              &\leq C_{\alpha} \|f\|_{\Lp(\mr^3)}.
	         \end{align*}
       \item[(b)] If $h\geq2^3$ and $\supp \widehat{f}\subset \mathbb{A}_{2^3}^{\circ}\times \mathbb{B}_h$, then for $(t,s)\in \mathbb{J}_0$,
	        \begin{align*}
			   \|\partial_t^{\alpha} \partial_s^{\beta} \mathcal{A}f(\cdot,t,s)\|_{L^{q}(\mathbb{R}^{3})}
			   &\leq C_{\alpha,\beta} h^{\beta-\frac{1}{2}+\frac{1}{p}-\frac{1}{q}}\|f\|_{\Lp(\mr^3)},\\[2pt]
		           \|\partial_t^{\alpha}\mathcal{A}_1f(\cdot,t)\|_{L^{q}(\mathbb{R}^{3})}
		           &\leq C_{\alpha} h^{\alpha-\frac{1}{2}+\frac{1}{p}-\frac{1}{q}}\|f\|_{\Lp(\mr^3)}.
	       \end{align*}
     \end{enumerate}
     \endgroup

\end{lem}
\begin{proof}
For $\alpha=\beta=0$, part (a) follows from the fixed-frequency
estimate in \cite[Lemma~2.5]{LL}, while part (b) follows from
\cite[Proposition~2.6]{LL} with $\tau=2^{-3}$. Since $\tau$ is
fixed, the corresponding enlargement of the low-frequency supports
only changes the constant.

The bounds arising from parameter differentiation follow directly from the asymptotic
formula \eqref{ali-expansion-sigmats}. For instance, consider the term
\begin{equation*}
    e^{i(\kappa t|\bar\xi|+\kappa' s|\xi|)}
    B_{1,1}^{\kappa,\kappa'}(\xi,t,s),
\end{equation*}
where $\kappa,\kappa'=\pm$. On
$\mathbb A_{2^3}^{\circ}\times\mathbb B_h$, differentiation with
respect to $t$ costs $O(1)$, whereas differentiation with respect
to $s$ costs $O(h)$. Along $s=c_0t$, the derivative
$\partial_t+c_0\partial_s$ therefore costs $O(h)$. On the fully
low-frequency support all these costs are $O(1)$. Derivatives
falling on the amplitudes preserve their symbol order by
\eqref{asm-B-symbol}, and the differentiated remainders are absorbed
by taking the expansion order sufficiently large, as explained
after that estimate. This gives all four cases
$\alpha,\beta\in\{0,1\}$.
\end{proof}

We next record estimates for the remaining regimes, in which the
horizontal frequency is localized to $\mathbb A_\lambda$ with
$\lambda\geq2^3$.

	\begin{lem}\label{lem-A4-A6}
		Let $\alpha,\beta\in \{0,1\}$, $\lambda\geq2^3$, $2\leq p\leq q \leq \infty$, and $\delta>0$. For any $\epsilon>0$, the following hold:
		\begingroup
		\addtolength{\leftmargini}{-2em}
		\begin{enumerate}
		\item[(a)] If $\lambda\leq h\leq\lambda^2$ and $\supp \widehat{f}\subset \mathbb{A}_\lambda\times \mathbb{B}_h$, then we have
		  	\begin{align*}
		  	    	\|\partial_t^{\alpha}\partial_s^{\beta}\mathcal{A}f\|_{L^{q}(\mathbb{R}^{3}\times \mathbb{J}_0)}\lesssim
                    \begin{cases}
                        \lambda^{\alpha+1-\frac{1}{p}-\frac{5}{q}+\epsilon}h^{\beta-1+\frac{2}{p}}\|f\|_{\Lp(\mr^3)}, & \frac{1}{p}+\frac{3}{q}\leq 1,\\[2pt]
                        \lambda^{\alpha-\frac{1}{2}+\frac{1}{2p}-\frac{1}{2q}+\epsilon}h^{\beta-\frac{1}{2}+\frac{3}{2p}-\frac{3}{2q}}\|f\|_{\Lp(\mr^3)}, & \frac{1}{p}+\frac{3}{q}>1.
                    \end{cases}
		  	\end{align*}

        \item[(b)] If $\supp \widehat{f}\subset \mathbb{A}_\lambda\times \mathbb{B}_\lambda^{\circ}$, then the estimate in (a) holds with $h=\lambda$.

		\item[(c)] If $h\geq\lambda^2$ and $\supp \widehat{f}\subset \mathbb{A}_\lambda\times \mathbb{B}_h$, then we have
			    \begin{align*}
			    	\|\partial_t^{\alpha}\partial_s^{\beta}\mathcal{A}f\|_{L^{q}(\mathbb{R}^{3}\times \mathbb{J}_0)}\lesssim
                    \begin{cases}
                        \lambda^{\alpha+\frac{1}{p}-\frac{3}{q}+\epsilon}h^{\beta-\frac{1}{2}+\frac{1}{p}-\frac{1}{q}}\|f\|_{\Lp(\mr^3)}, & \frac{1}{p}+\frac{3}{q}\leq 1, \\[2pt]
                        \lambda^{\alpha-\frac{1}{2}+\frac{3}{2p}-\frac{3}{2q}+\epsilon}h^{\beta-\frac{1}{2}+\frac{1}{p}-\frac{1}{q}}\|f\|_{\Lp(\mr^3)}, & \frac{1}{p}+\frac{3}{q}>1.
                    \end{cases}
			    \end{align*}
		\item[(d)] If $\supp \widehat{f}\subset \mathbb{A}_\lambda\times \mathbb{B}_{\lambda^{1/2+\delta}}^{\circ}$, then we have
			\begin{align*}
				\|\partial_t^{\alpha}\partial_s^{\beta}\mathcal{A}f\|_{L^{q}(\mathbb{R}^{3}\times \mathbb{J}_0)}\lesssim
				\begin{cases}
					\lambda^{\alpha+\beta-\frac{1}{2}+\frac{3}{2p}-\frac{7}{2q}+300\delta+\epsilon}\|f\|_{\Lp(\mr^3)}, & \frac{1}{p}+\frac{3}{q}\leq 1,\\[2pt]
					\lambda^{\alpha+\beta-1+\frac{2}{p}-\frac{2}{q}+300\delta+\epsilon}\|f\|_{\Lp(\mr^3)}, & \frac{1}{p}+\frac{3}{q}> 1.
				\end{cases}
			\end{align*}

			\item[(e)] If $\lambda^{1/2+\delta}\leq h\leq\lambda$ and $\supp \widehat{f}\subset \mathbb{A}_\lambda\times \mathbb{B}_h$, then we have
			\begin{align*}
				\|\partial_t^{\alpha}\partial_s^{\beta}\mathcal{A}f\|_{L^{q}(\mathbb{R}^{3}\times \mathbb{J}_0)}\lesssim
				\begin{cases}
					\lambda^{\alpha+\beta-1+\frac{2}{p}-\frac{2}{q}+\epsilon}h^{1-\frac{1}{p}-\frac{3}{q}}\|f\|_{\Lp(\mr^3)}, & \frac{1}{p}+\frac{3}{q}\leq 1,\\[2pt]
					\lambda^{\alpha+\beta-1+\frac{2}{p}-\frac{2}{q}+\epsilon}\|f\|_{\Lp(\mr^3)}, & \frac{1}{p}+\frac{3}{q}> 1.
				\end{cases}
			\end{align*}

		\end{enumerate}
		\endgroup
\end{lem}
\begin{proof}
    The estimates in parts (a)--(c) follow from
    \cite[Corollary~2.10 (a)--(c)]{LL} by taking $\tau=2^{-3}$
    and translating to the present dyadic scale notation. The cited
    result includes all choices of $\alpha,\beta\in\{0,1\}$.
    It remains to prove (d) and (e).

    We first establish (e). By interpolation, it is enough to consider the five cases $(p,q)=(2,2),\, (4,4),\, (2,6),\, (2,\infty),\, (\infty,\infty)$. The first three cases follow directly from (b), while the case $(p,q)=(2,\infty)$ follows from the Cauchy--Schwarz inequality and Plancherel's theorem. Thus, it remains to consider the case $(p,q)=(\infty,\infty)$. By \eqref{ali-expansion-sigmats}, the terms involving $B_{0,0}$ or
    $B_{0,1}^{\pm}$ occur only at bounded horizontal frequencies and are
    absorbed by Lemma~\ref{lem-A1-A2}. The four terms involving
    $B_{1,1}^{\kappa,\kappa'}$ are treated identically after changing
    the signs of the phases. It therefore suffices to prove that the
    inverse Fourier transform of the multiplier
    \[
    e^{i(t|\bar{\xi}|+s|\xi|)}\varphi_\lambda(|\bar{\xi}|)\varphi_h(|\xi_3|)B_{1,1}^{+,+}(\xi,t,s)
    \]
    has $L^1$ norm bounded by $Ch\lambda^{-1}$, uniformly for
    $(t,s)\in \mathbb{J}_0$.

     Let $\mathfrak{V}_\lambda \subset\mathbb{S}^1$ be a maximal $\lambda^{-1/2}$-separated set. Let $\{w_\nu\}_{\nu\in\mathfrak{V}_\lambda}$ be a partition of unity on $\mathbb{S}^1$, with each $w_\nu$ supported in an arc centered at $\nu$ of length comparable to $\lambda^{-1/2}$ and satisfying $\|\partial_\theta^m w_\nu\|_{L^\infty(\mathbb{S}^1)}\lesssim_m\lambda^{m/2}$ for every integer $m\geq0$. For each $\nu \in \mathfrak{V}_\lambda$, set
\[
\omega_\nu(\bar{\xi}):= w_\nu(\bar{\xi}/|\bar{\xi}|).
\]
Let $\rho \in C_c^{\infty}((-1,1))$ satisfy $\sum_{n\in \mathbb{Z}} \rho(\cdot - n) = 1$. Define
\[
\begin{aligned}
K^{\lambda,h,n,\nu}(x,t,s)
&=\int_{\mr^3}e^{i(x\cdot\xi+t|\bar\xi|+s|\xi|)}\\
&\qquad\times\varphi_\lambda(|\bar\xi|)\omega_\nu(\bar\xi)
\varphi_h(|\xi_3|)
\rho\Big(\frac{\lambda^{1/2}\xi_3}{|\bar\xi|}-n\Big)
B_{1,1}^{+,+}(\xi,t,s)\,d\xi.
\end{aligned}
\]
Thus, we have
\begin{align*}
	  & \left|
	   \int_{\mr^3}e^{i(x\cdot\xi +t|\bar{\xi}|+s|\xi|)}
	   \varphi_\lambda(|\bar{\xi}|)\varphi_h(|\xi_3|) B_{1,1}^{+,+}(\xi,t,s)\,d\xi
	  \right|
	  \\
		   &\qquad \quad \qquad \le \sum_{\nu \in \mathfrak{V}_\lambda}\sum_{n\in \mathbb{Z}}|K^{\lambda,h,n,\nu}(x,t,s)|.
\end{align*}
The multiplier defining $K^{\lambda,h,n,\nu}$ is supported in a cap
on $\mathbb{S}^2$ of diameter $O(\lambda^{-1/2})$ centered at
\[
\theta_{\nu,n}
:=
\frac{(\nu,\lambda^{-1/2}n)}
{\sqrt{1+\lambda^{-1}n^2}}.
\]
On this support, the phase differences
$t(|\bar{\xi}|-\nu\cdot\bar{\xi})$ and
$s(|\xi|-\theta_{\nu,n}\cdot\xi)$ have uniformly bounded derivatives
under the corresponding parabolic rescaling and can therefore be
absorbed into the amplitude. In coordinates adapted to
$\theta_{\nu,n}$, the resulting multiplier is supported in a box of
dimensions $\lambda\times\lambda^{1/2}\times\lambda^{1/2}$. 
Since $|\bar\xi|\sim|\xi|\sim\lambda$,
\eqref{asm-B-symbol} shows  that the amplitude has size
$O(\lambda^{-1})$ and that, after multiplying it by $\lambda$ and
rescaling the box to unit size, all its derivatives are uniformly
bounded.

Set
\[
y=x+t(\nu,0)+s\theta_{\nu,n}
\]
and choose an orthonormal basis
$\{\theta_{\nu,n},e_{\nu,n,1},e_{\nu,n,2}\}$ of $\mathbb R^3$ and write
$
y=y_\parallel\theta_{\nu,n}
+y_{\perp,1}e_{\nu,n,1}
+y_{\perp,2}e_{\nu,n,2}.
$
Integration by parts then gives, for $N$ sufficiently large,
\[
|K^{\lambda,h,n,\nu}(x,t,s)|
\lesssim_N \lambda
\bigl(1+\lambda|y_\parallel|
+\lambda^{1/2}|y_{\perp,1}|
+\lambda^{1/2}|y_{\perp,2}|\bigr)^{-N}.
\]
The change of variables in this majorant has Jacobian $\lambda^{-2}$;
hence
\[
\|K^{\lambda,h,n,\nu}(\cdot,t,s)\|_{L_x^1(\mr^3)}
\lesssim \lambda\lambda^{-2}=\lambda^{-1}.
\]
This estimate holds uniformly in $(t,s)\in\mathbb J_0$. Since $\#\mathfrak{V}_\lambda\approx\lambda^{1/2}$ and only $O(h\lambda^{-1/2})$ values of $n$ give nonzero multipliers, summing over $(\nu,n)$ yields
\[
\sum_{\nu,n}\|K^{\lambda,h,n,\nu}(\cdot,t,s)\|_{L_x^1}
\lesssim \lambda^{1/2}(h\lambda^{-1/2})\lambda^{-1}
=h\lambda^{-1}.
\]
Parameter derivatives falling on the phase contribute at most $\lambda^{\alpha+\beta}$, while derivatives falling on the amplitude preserve the symbol bounds. Hence the resulting bound is $\lambda^{\alpha+\beta-1}h$, which proves (e) for $(p,q)=(\infty,\infty)$.

We now turn to (d). The support considerations and symmetry reduction used above
for part (e) reduce the proof to estimating the propagator associated
with $B_{1,1}^{+,+}$. We define
\begin{align*}
    B(\xi,t,s):=(1+|\bar{\xi}|)^{1/2}(1+|\xi|)^{1/2}B_{1,1}^{+,+}(\xi,t,s).
\end{align*}
By \eqref{asm-B-symbol}, $B(\xi,t,s)$ satisfies the symbol estimates in \eqref{ali-condition-on-B}.
Let $\mathcal{U}_{t,s}$ be as in \eqref{def-wjk}.
Thus, it remains to show that
    \begin{align*}
	    	\|\mathcal{U}_{t,s}f_{\lambda,\leq \lambda^{1/2+\delta}}\|_{L^{q}_{x,t,s}(\mathbb{R}^{3}\times \mathbb{J}_0)}\lesssim
    		\begin{cases}
	    		\lambda^{\frac{1}{2}+\frac{3}{2p}-\frac{7}{2q}+300\delta+\epsilon}\|f\|_{\Lp(\mr^3)}, & \frac{1}{p}+\frac{3}{q}\leq 1,\\[2pt]
	    		\lambda^{\frac{2}{p}-\frac{2}{q}+300\delta+\epsilon}\|f\|_{\Lp(\mr^3)}, & \frac{1}{p}+\frac{3}{q}>1.
    		\end{cases}
    	\end{align*}
		        We first consider the case $1/p+3/q\leq1$. The desired estimate is
		        trivial when $\lambda^\delta<2$, so we may assume that
		        $\lambda^\delta\geq2$. Thus,
       \begin{align*}
	       	 \mathcal{U}_{t,s}f_{\lambda,\leq \lambda^{1/2+\delta}}(x)=\sum_{\lambda^{1/2}<h\leq\lambda^{1/2+\delta}}\mathcal{U}_{t,s}f_{\lambda,h}(x)+\mathcal{U}_{t,s}f_{\lambda,\leq \lambda^{1/2}}(x).
       \end{align*}
	       For the second term, the same argument applies after rescaling the
	       vertical frequency at $\lambda^{1/2}$ and replacing the annular
	       cutoff in $\xi_3$ by a smooth low-frequency cutoff. The resulting
	       amplitude satisfies the estimates below with $\delta=0$, and hence
	       gives the required bound for
	       $\mathcal{U}_{t,s}f_{\lambda,\leq\lambda^{1/2}}$.

	       It remains to estimate the first sum. It contains
	       $O(\log(\lambda^\delta))$ dyadic scales, and this factor is absorbed
	       by $\lambda^\delta$. It therefore suffices to prove that
	       \begin{align}\label{ali-sec2-qqp-k<j/2}
	       	 \|\mathcal{U}_{t,s}f_{\lambda,h}\|_{L^q_{x,t,s}(\mr^3\times\mathbb{J}_0)}\lesssim \lambda^{\frac{1}{2}+\frac{3}{2p}-\frac{7}{2q}+299\delta+\epsilon}\|f\|_{L^p(\mr^3)}
       \end{align}
	       holds uniformly for $\lambda^{1/2}<h\leq\lambda^{1/2+\delta}$.

        Let $\widetilde{\varphi}\in C_c^{\infty}((1/4,3))$ be such that $\widetilde{\varphi}\varphi=\varphi$. After a change of variables, we have
       \begin{align*}
       	\mathcal{U}_{t,s}f_{\lambda,h}(x)&=\lambda^2h\int_{\mr^3}e^{ix\cdot (\lambda\bar{\xi},h\xi_3)}e^{i(t+s)\lambda|\bar{\xi}|} B_{\lambda,h,1}(\xi,t,s)\widehat{f_{\lambda,h}}(\lambda\bar{\xi},h\xi_3)\,d\xi,
       \end{align*}
       where
       \begin{align*}
       	B_{\lambda,h,1}(\xi,t,s):=e^{is(|(\lambda\bar{\xi},h\xi_3)|-\lambda|\bar{\xi}|)}\widetilde{\varphi}(|\bar{\xi}|)\widetilde{\varphi}(|\xi_3|)B(\lambda\bar{\xi},h\xi_3,t,s).
       \end{align*}
	       The support of $B_{\lambda,h,1}(\cdot,t,s)$ is contained in
	       $(-\pi,\pi)^3$. Moreover, for every multi-index $\eta$,
	       \[
	       \sup_{\xi}|\partial_\xi^{\eta}B_{\lambda,h,1}(\xi,t,s)|
	       \leq C_{\eta}\lambda^{2\delta |\eta|}
	       \]
	       uniformly in $(t,s)\in \mathbb{J}_0$. Expanding
	       $B_{\lambda,h,1}(\cdot,t,s)$ in a Fourier series on
	       $[-\pi,\pi]^3$, we have
       \begin{align*}
       	B_{\lambda,h,1}(\xi,t,s)=\sum_{\ell\in \mathbb{Z}^3}C_{\ell,1}(t,s)e^{i\ell\cdot\xi}, 
       \end{align*}
	       where, uniformly in $(t,s)\in \mathbb{J}_0$,
	      $
	       |C_{\ell,1}(t,s)|\lesssim \lambda^{200\delta}(1+|\ell|)^{-100}.
	      $
	       Minkowski's inequality and the change of variables give
       \begin{align*}
	       	\|\mathcal{U}_{t,s}f_{\lambda,h}\|_{L^q_{x,t,s}(\mr^3\times \mathbb{J}_0)}\lesssim \lambda^{200\delta}\|\mathcal{L}f_{\lambda,h}\|_{L^q_{x,t,s}(\mr^3\times \mathbb{J}_0)},
       \end{align*}
       where
       \begin{align*}
           \mathcal{L}f(x,t,s):=\int_{\mr^3}e^{i(x\cdot\xi+(t+s)|\bar{\xi}|)}\widehat{f}(\xi)\,d\xi.
       \end{align*}

       Note that
     $
           \mathcal{L}f(x,t,s)=\mathcal{W}_+\left(f(\cdot, x_3)\right)(\bar{x},t+s).
     $
       Using a change of variables, Lemma~\ref{lem-one-qq-p}, and Minkowski's and Bernstein's inequalities, we obtain
       \begin{align*}
       	\|\mathcal{L}f_{\lambda,h}\|_{L^q_{x,t,s}(\mr^3\times \mathbb{J}_0)}
	       	&\lesssim \lambda^{\frac{1}{2}+\frac{1}{p}-\frac{3}{q}+\epsilon}
	       	\|f_{\lambda,h}\|_{L_{x_3}^{q}(\mr;L_{\bar{x}}^{p}(\mr^2))}\\
	       	&\leq \lambda^{\frac{1}{2}+\frac{1}{p}-\frac{3}{q}+\epsilon}
	       	\|f_{\lambda,h}\|_{L_{\bar{x}}^{p}(\mr^2;L_{x_3}^{q}(\mr))}\\
	       	&\lesssim \lambda^{\frac{1}{2}+\frac{3}{2p}-\frac{7}{2q}+\delta+\epsilon}\|f\|_{L^p(\mr^3)}.
	       \end{align*}
		       Together with the factor $\lambda^{200\delta}$ above, this proves
		       \eqref{ali-sec2-qqp-k<j/2}.

		       If $1/p+3/q>1$, we instead use the second estimate in Lemma
		       \ref{lem-one-qq-p} and repeat the mixed-norm argument with
		       Minkowski's and Bernstein's inequalities to obtain
	       \[
	       \|\mathcal{L}f_{\lambda,h}\|_
	       {L^q_{x,t,s}(\mr^3\times\mathbb{J}_0)}
	       \lesssim
	       \lambda^{\frac{2}{p}-\frac{2}{q}+\delta+\epsilon}
	       \|f\|_{L^p(\mr^3)}.
		       \]
		       Together with the factor $\lambda^{200\delta}$ above, this yields
		       the corresponding dyadic estimate when $1/p+3/q>1$.

		       Finally, on the frequency support in (d), we have
		       $|\bar{\xi}|\sim\lambda$ and $|\xi|\lesssim\lambda^{1+\delta}$.
		       By the definition of $B$, the multiplier obtained after restoring the original amplitude
		       $B_{1,1}^{+,+}$ and applying the parameter derivatives is bounded by
		       $\lambda^{\alpha-\frac12}|\xi|^{\beta-\frac12}
		       \lesssim\lambda^{\alpha+\beta-1+\delta/2}$.
		       Derivatives falling on the amplitude preserve its symbol bounds.
		       The additional factor $\lambda^{\delta/2}$ is absorbed by the
		       $300\delta$ loss in (d), which proves (d).
\end{proof}

The preceding estimates cover every frequency regime used below.
Lemma~\ref{lem-A1-A2} handles the fully low-frequency and
low-horizontal-frequency regimes. For $\lambda\geq2^3$, parts
(d), (e), (a)--(b), and (c) of Lemma~\ref{lem-A4-A6} treat,
respectively,
\[
 h<\lambda^{1/2+\delta},\qquad
 \lambda^{1/2+\delta}\leq h\leq\lambda,\qquad
 \lambda\leq h\leq\lambda^2,\qquad
 h\geq\lambda^2.
\]
Section~\ref{sec-square-func-est} provides the square-function estimates used for interpolation.
 \section{Square function estimates}\label{sec-square-func-est}
In this section, we establish the square function estimates for
$\mathcal{S}$ needed in Section~\ref{sec-Lp-Lq-bounds}; see
\eqref{def-square-function} for its definition. To handle the phase
$t|\bar{\xi}|+s|\xi|$ in $\mathcal{U}_{t,s}$, we consider four
frequency ranges:
\begin{align*}
    h<\lambda^{1/2+\delta},\quad
    \lambda^{1/2+\delta}\leq h\leq\lambda,\quad
    \lambda\leq h\leq\lambda^2,\quad
    h\geq\lambda^2.
\end{align*}
Here $\delta>0$ is a small parameter that determines the frequency
threshold, while $\epsilon>0$ below is reserved for losses in the
analytic estimates. In the first range,
$|\xi|=|\bar{\xi}|+O(\lambda^{2\delta})$; the resulting phase
correction is absorbed into the amplitude with an admissible
$\lambda^{C\delta}$ loss in the proof of Lemma~\ref{lem-A4-A6} (d).
In the last range, $|\xi|=|\xi_3|+O(\lambda^2/h)$, and
Proposition~\ref{prop-L6L2-L2-k>2j} applies the analogous reduction with
$|\xi_3|$ in place of $|\bar{\xi}|$. The two intermediate ranges
require more detailed support analysis and are treated in
Propositions~\ref{prop-main-jk} and \ref{prop-main-j<k<2j}.

\subsection{An auxiliary annular estimate}
The following lemma will play a key role in the square function estimates below.
\begin{lem}\label{lem-annulus-L2-L6}
	Let $\tau\in [2^{-2},2^2+2^{-2}]$ and $0< \eta\leq 1$. Suppose that
 	$f\in L^2(\mathbb{R}^2)$ satisfies
 	\[
 \supp \widehat{f}\subset \bigl\{\bar{\xi}\in \mathbb{R}^2:\bigl||\bar{\xi}|-\tau\bigr|\leq \eta\bigr\}.
 	\]
 	Then
 	\[
	\|f\|_{L^6(\mathbb{R}^2)}\lesssim \eta^{1/2}\|f\|_{L^2(\mathbb{R}^2)},
 	\]
	where the implicit constant is independent of $\tau$, $\eta$, and $f$.
 \end{lem}
\begin{proof}
	We first consider the case $\eta\geq \tau/2$. Note that
 	\[
	\supp \widehat{f}\subset \{|\bar{\xi}|\leq \tau+\eta\}\subset B(0,6).
 	\]
	If $\eta\geq \tau/2$, then $\eta\geq 2^{-3}$, and hence Bernstein's inequality yields
 	\[
	\|f\|_{L^6(\mathbb{R}^2)}\lesssim \|f\|_{L^2(\mathbb{R}^2)}
	\lesssim \eta^{1/2}\|f\|_{L^2(\mathbb{R}^2)}.
	\]
	It remains to consider the case $\eta<\tau/2$. Set
 	\[
 	Eg(\bar{x}):=\int_{\mathbb{S}^1} e^{i\bar{x}\cdot \omega}g(\omega)\,d\sigma(\omega),
 	\]
 	and
 	\[
	g_{\rho}(\omega):=\widehat{f}(\rho\omega), \quad (\omega,\rho)\in \mathbb{S}^1\times [\tau-\eta,\tau+\eta].
 	\]
 	By polar coordinates, we have
 	\[
 	f(\bar{x})
	=\int_{\tau-\eta}^{\tau+\eta}\int_{\mathbb{S}^1}
	e^{i\bar{x}\cdot \rho\omega}\widehat{f}(\rho\omega)\,\rho\,d\sigma(\omega)\,d\rho=\int_{\tau-\eta}^{\tau+\eta}\rho\, Eg_{\rho}(\rho\bar{x})\,d\rho.
 	\]
	By the Stein--Tomas restriction theorem for $\mathbb{S}^1$ and scaling,
	\[
	\|Eg(\rho\cdot)\|_{L^6(\mathbb{R}^2)}
	=\rho^{-1/3}\|Eg\|_{L^6(\mathbb{R}^2)}
	\lesssim \rho^{-1/3}\|g\|_{L^2(\mathbb{S}^1)}
	\lesssim \|g\|_{L^2(\mathbb{S}^1)}.
	\]
	The last inequality is uniform for
	$\rho\in[\tau-\eta,\tau+\eta]\subset[2^{-3},6]$.
	Therefore, by
 	Minkowski's inequality,
 	\[
 	\|f\|_{L^6(\mathbb{R}^2)}
	\leq \int_{\tau-\eta}^{\tau+\eta}\rho\, \|E g_{\rho}(\rho\cdot)\|_{L^6(\mathbb{R}^2)}\,d\rho
	\lesssim \int_{\tau-\eta}^{\tau+\eta}\|g_{\rho}\|_{L^2(\mathbb{S}^1)}\,d\rho.
 	\]
 	Applying the Cauchy--Schwarz inequality and Plancherel's theorem, we obtain
 	\begin{align*}
		\|f\|^2_{L^6(\mathbb{R}^2)}
		&\lesssim \eta\int_{\tau-\eta}^{\tau+\eta}\|g_{\rho}\|_{L^2(\mathbb{S}^1)}^2\,d\rho\\
		&\sim \eta\int_{\tau-\eta}^{\tau+\eta}\int_{\mathbb{S}^1}|\widehat{f}(\rho\omega)|^2\,\rho\,d\sigma(\omega)\,d\rho
		\sim\eta \|f\|_{L^2(\mathbb{R}^2)}^2,
 	\end{align*}
 	which completes the proof.
 \end{proof}

\subsection{Intermediate-frequency estimates}

\subsubsection*{The regime $\lambda^{1/2+\delta}\leq h\leq\lambda$}

\begin{prop}\label{prop-main-jk}
	Let $\delta,\epsilon>0$, $\lambda\geq2^3$, and $\lambda^{1/2+\delta}\leq h\leq\lambda$. Then,
    \begin{align}
       \|\mathcal{S}f_{\lambda,h}\|_{L^{\infty}(\mr^3)} & \lesssim_{\delta,\epsilon} \lambda^{\frac{1}{2}+\epsilon}\|f\|_{L^\infty(\mr^3)} \label{ali-main-jk},\\
        \|\mathcal{S}f_{\lambda,h}\|_{L^6(\mathbb{R}^3)}&\lesssim \lambda^{\frac12}h^{-\frac13}\|f\|_{L^2(\mr^3)},\label{ali-L6L2-L2-j/2-j} \\
        \|\mathcal{S}f_{\lambda,h}\|_{L^{\infty}(\mr^3)}&\lesssim \lambda h^{-\frac12}\|f\|_{L^2(\mr^3)}\label{ali-LinftyL2-L2-j/2-j}.
    \end{align}
\end{prop}

Fix $\psi\in C_c^{\infty}((-1,1))$ such that $\widehat{\psi}(s)>0$
for $s\in [2^{-3},2^3]$. We use the following localization lemma in
the proof of Proposition~\ref{prop-main-jk}.

\begin{lem}\label{lem-intersection}
	Let $\lambda\geq2^3$ and $2^7\lambda^{1/2}\leq h\leq\lambda$. Suppose that
	$(\tau,\sigma)\in [2^{-2},2^2+2^{-2}]^2$ and
	\[
		\psi\bigl(\lambda(|(\bar{\xi},(h/\lambda)\xi_3)|-\sigma)\bigr)
	\psi(\lambda(|\bar{\xi}|-\tau))
		\varphi(|\bar{\xi}|)\varphi(|\xi_{3}|)\neq0.
	\]
	Then
	\begin{align}\label{ali-range-barxi}
		\tau-\lambda^{-1}\leq |\bar{\xi}|\leq \tau+\lambda^{-1},
	\end{align}
	and
	\begin{align}\label{ali-range-xi3}
		\sqrt{\sigma^2-\tau^2}-2^8h^{-1}
		\leq (h/\lambda)|\xi_3|
		\leq \sqrt{\sigma^2-\tau^2}+2^8h^{-1}.
	\end{align}
\end{lem}

\begin{proof}
		Put $R=|(\bar{\xi},(h/\lambda)\xi_3)|$. The support condition on
		$\psi$ gives
	    \begin{align}\label{ali-barxi-R}
	        \bigl||\bar{\xi}|-\tau\bigr|\leq \lambda^{-1},\qquad|R-\sigma|\leq \lambda^{-1},
	    \end{align}
	and hence \eqref{ali-range-barxi}. Since $\varphi(|\xi_3|)\neq 0$, $h/(2\lambda)\leq (h/\lambda)|\xi_3|\leq 2h/\lambda$.
	Note that
    \begin{align}\label{ali-2(k-j)xi3}
        R-|\bar{\xi}|=\frac{(h/\lambda)^2|\xi_3|^2}{R+|\bar{\xi}|}.
    \end{align}
	Using $\varphi(|\bar{\xi}|)\varphi(|\xi_3|)\neq 0$, we have
	$1\leq R+|\bar{\xi}|\leq 2^3$, and thus
	\[2^{-5}(h/\lambda)^2\leq R-|\bar{\xi}|\leq 2^2(h/\lambda)^2.\]
	Since $|(\sigma-\tau)-(R-|\bar{\xi}|)|\leq2\lambda^{-1}$,
	and since $2\lambda^{-1}\leq2^{-13}(h/\lambda)^2$ by the assumption $h\geq2^7\lambda^{1/2}$, we obtain
	\[
	\begin{aligned}
	\sigma-\tau
\geq R-|\bar{\xi}|-2\lambda^{-1} \geq2^{-8}(h/\lambda)^2
	>2\lambda^{-1}.
	\end{aligned}
	\]
	     By \eqref{ali-barxi-R} and \eqref{ali-2(k-j)xi3}, we have
		\[(\sigma-\lambda^{-1})^2-(\tau+\lambda^{-1})^2\leq (h/\lambda)^2|\xi_3|^2\leq(\sigma+\lambda^{-1})^2-(\tau-\lambda^{-1})^2,\]
		 equivalently,
		\[\sigma^2-\tau^2-2\lambda^{-1}(\sigma+\tau)\leq (h/\lambda)^2|\xi_3|^2\leq\sigma^2-\tau^2+2\lambda^{-1}(\sigma+\tau).\]
		  Hence, using $\sigma+\tau<2^4$ and
		$\sigma-\tau\geq2^{-8}(h/\lambda)^2>2\lambda^{-1}$,
		and rationalizing the difference, we obtain
		\[
		\begin{aligned}
		\left|(h/\lambda)|\xi_3|-\sqrt{\sigma^2-\tau^2}\right|
		&=
		\frac{\left|(h/\lambda)^2|\xi_3|^2-(\sigma^2-\tau^2)\right|}
		{(h/\lambda)|\xi_3|+\sqrt{\sigma^2-\tau^2}}\leq
		\frac{2\lambda^{-1}(\sigma+\tau)}
		{\sqrt{\sigma^2-\tau^2}}\\
		&=
		2\lambda^{-1}
		\sqrt{\frac{\sigma+\tau}{\sigma-\tau}}
		\leq2^8h^{-1}.
		\end{aligned}
		\]
	This proves \eqref{ali-range-xi3}.
\end{proof}

We now prove Proposition~\ref{prop-main-jk}.
\begin{proof}[Proof of Proposition~\ref{prop-main-jk}]
We first dispose of the range
$\lambda^{1/2+\delta}\leq h\leq2^7\lambda^{1/2}$.
As in the proof of Lemma~\ref{lem-A4-A6} (d), we write
$e^{is|\xi|}=e^{is|\bar\xi|}e^{is(|\xi|-|\bar\xi|)}$ and absorb the
second factor into the amplitude. Since $h^2/\lambda\leq2^{14}$,
the resulting symbol bounds are uniform.
Applying the Stein--Tomas and Plancherel estimates in the
horizontal variables, together with the one-dimensional Bernstein
inequality in $x_3$, gives
\[
\begin{aligned}
\|\mathcal{S}f_{\lambda,h}\|_{L^\infty}
&\lesssim \lambda^{1/2}\|f\|_{L^\infty},\\
\|\mathcal{S}f_{\lambda,h}\|_{L^6}
&\lesssim \lambda^{1/6}h^{1/3}\|f\|_{L^2},\\
\|\mathcal{S}f_{\lambda,h}\|_{L^\infty}
&\lesssim \lambda^{1/2}h^{1/2}\|f\|_{L^2}.
\end{aligned}
\]
Here the powers $h^{1/3}$ and $h^{1/2}$ result from
applying Bernstein's inequality in the $x_3$-variable with $q=6$
and $q=\infty$, respectively.
Because $h\leq2^7\lambda^{1/2}$, these imply
\eqref{ali-main-jk}--\eqref{ali-LinftyL2-L2-j/2-j} in this range.
It therefore remains to consider $2^7\lambda^{1/2}\leq h\leq\lambda$.

Let $\widetilde{\varphi}\in C_c^{\infty}((1/4,3))$ be such that $\widetilde{\varphi}\varphi=\varphi$. After a change of variables, we write
\begin{align*}
			\mathcal{U}_{t,s}f_{\lambda,h}(x)
			&=\int_{\mr^3}\int_{\mr^3}
			e^{i(\lambda\bar{x}-\bar{y},\,hx_3-y_3)\cdot(\bar{\xi},\xi_3)}
			e^{i(t\lambda|\bar{\xi}|+s\lambda|(\bar{\xi},(h/\lambda)\xi_3)|)} \\
			&\qquad \times \varphi(|\bar{\xi}|)\varphi(|\xi_3|)
				B_{\lambda,h,2}(\xi,t,s)\,f(\lambda^{-1}\bar{y},h^{-1}y_3)\,d\xi\,dy,
		\end{align*}
		where
		\[
		B_{\lambda,h,2}(\xi,t,s)
		=\widetilde{\varphi}(|\bar{\xi}|)\widetilde{\varphi}(|\xi_3|)
		\,B(\lambda\bar{\xi},h\xi_3,t,s).
		\]
        Note that for any multi-index $\alpha$,
		$
			\sup_{\xi}|\partial_\xi^{\alpha}B_{\lambda,h,2}(\xi,t,s)|\leq C_{\alpha}
	        $
		holds uniformly in $(t,s)\in \mathbb{J}_0$. Expanding $B_{\lambda,h,2}(\cdot,t,s)$ in a Fourier series on $[-\pi,\pi]^3$, we have
			\begin{align*}
				B_{\lambda,h,2}(\xi,t,s)=\sum_{\ell\in \mathbb{Z}^3}C_{\ell,2}(t,s)e^{i\ell\cdot\xi},
			\end{align*}
			where $|C_{\ell,2}(t,s)|\lesssim (1+|\ell|)^{-100}$. Note that
			$\widehat{\psi}$ is bounded below on $[2^{-3},2^3]$. Using the bound
			on $C_{\ell,2}$, Minkowski's inequality, and then Plancherel's theorem
			in $t$ and $s$, we obtain
    \begin{align}
        \mathcal{S}f_{\lambda,h}(x) & \lesssim \left(\int_{\mr^2}|\mathcal{U}_{t,s}f_{\lambda,h}(x)\widehat{\psi}(t)\widehat{\psi}(s)|^{2}\,dt\,ds\right)^{1/2}\notag\\
	        &\lesssim \sum_{\ell\in \mathbb{Z}^3}\frac{\lambda}
        {(1+|\ell|)^{100}}\left(\int_{\mr^2}|\mathcal{T}^{\lambda,h}_{\tau,\sigma}(\mathcal{D}_{\lambda,h}^\ell f)(\lambda\bar{x},hx_3)|^2\,d\tau\,d\sigma \right)^{1/2},\label{ali-Wjk-j/2<k<j}
    \end{align}
        where
        \[\mathcal{D}_{\lambda,h}^\ell f:=f(\lambda^{-1}(x_1+\ell_1), \lambda^{-1}(x_2+\ell_2),h^{-1}(x_3+\ell_3)),\quad \ell:=(\ell_1,\ell_2,\ell_3),\]
        and $\mathcal{T}^{\lambda,h}_{\tau,\sigma}f:=K^{\lambda,h}_{\tau,\sigma}*f$ with
		\begin{align*}
			K_{\tau,\sigma}^{\lambda,h}(x):=\int_{\mr^3}e^{i x\cdot \xi}\psi\bigl(\lambda(|(\bar{\xi},(h/\lambda)\xi_{3})|-\sigma)\bigr)\psi(\lambda(|\bar{\xi}|-\tau)) \varphi(|\bar{\xi}|)\varphi(|\xi_{3}|)\,d\xi.
		\end{align*}

		Let
		\begin{align*}
\mathcal{T}^{\lambda,h}f(x):=\left(\int_{\mr^2}|\mathcal{T}^{\lambda,h}_{\tau,\sigma} f(x)|^2\,d\tau\,d\sigma\right)^{1/2}.
		\end{align*}
	      By the support conditions on $\psi$ and $\varphi$, the
		$(\tau,\sigma)$-integration defining $\mathcal T^{\lambda,h}$ is restricted to
	        $[2^{-2},2^2+2^{-2}]^2$.
	In this normalization, the spatial change of variables has Jacobian
	$\lambda^2h$, while Plancherel's theorem applied in the $(t,s)$-variables
	gives $\lambda$ in \eqref{ali-Wjk-j/2<k<j}. Taking the $L^q_x$-norm there, using the rapid decay of
        the Fourier coefficients, and scaling back, we obtain, for any $2\leq p\leq q\leq \infty$,
        \begin{equation}\label{Wjk-Sjk-j/2<k<j}
    \Vert \mathcal{S}f_{\lambda,h}\Vert_{ L^q(\mr^3)}
    \lesssim
	    \lambda(\lambda^2h)^{\frac1p-\frac1q}
    \Vert \mathcal{T}^{\lambda,h}\Vert_{L^p(\mr^3)\to L^q(\mr^3)}\Vert f\Vert_{L^p(\mr^3)}.
    \end{equation}
     Thus, \eqref{ali-L6L2-L2-j/2-j} and \eqref{ali-LinftyL2-L2-j/2-j} are reduced to
		\begin{align}\label{ali-L6L2-L2-improve-k}
			\|\mathcal{T}^{\lambda,h}f\|_{L^6(\mr^3)}\lesssim \lambda^{-\frac{7}{6}}h^{-\frac{2}{3}}\|f\|_{L^2(\mr^3)},
		\end{align}
		and
		\begin{align}\label{ali-Sjk-Linfty-L2}
			\|\mathcal{T}^{\lambda,h}f\|_{L^{\infty}(\mr^3)}\lesssim (\lambda h)^{-1}\|f\|_{L^2(\mr^3)}.
		\end{align}

	Since $\mathcal F[\mathcal{T}_{\tau,\sigma}^{\lambda,h}f]=\widehat{f}\cdot\mathcal F[K_{\tau,\sigma}^{\lambda,h}]$, it suffices to consider the parameters for which the multiplier is nonzero. Lemma~\ref{lem-intersection} then gives $\sigma>\tau$, and for every such $(\tau,\sigma)\in [2^{-2},2^2+2^{-2}]^2$,
		 \begin{align}\label{fourier-support--Sjk}
		 	\supp \mathcal F[ \mathcal{T}_{\tau,\sigma}^{\lambda,h}f ]\subset \{\xi:\bigl||\bar{\xi}|-\tau\bigr|\leq \lambda^{-1},\  \bigl||\xi_3|-\lambda/h\sqrt{\sigma^2-\tau^2}\bigr|\leq 2^8\lambda h^{-2}\}.
		 \end{align}
		 Thus, in the rescaled frequency variables, the horizontal radial and vertical thicknesses are $O(\lambda^{-1})$ and $O(\lambda h^{-2})$, respectively, and the support has measure $O(h^{-2})$.

		 For each fixed $\xi$, the set of admissible $(\tau,\sigma)$ has measure $O(\lambda^{-2})$.
		 Combining \eqref{fourier-support--Sjk} with Minkowski's inequality,
		 Lemma~\ref{lem-annulus-L2-L6}, and Bernstein's inequality gives
	\[ \|\mathcal{T}^{\lambda,h}f\|_{L^{6}(\mr^3)}^2
		 	\lesssim
		 	\lambda^{-\frac13}h^{-\frac43}
		 	\int_{\mr^2}\int_{\mr^3}
		 	|\mathcal{T}^{\lambda,h}_{\tau,\sigma}f(x)|^2
		 	\,dx\,d\tau\,d\sigma.\]
Thus, by Plancherel's theorem we obtain
		 \begin{align*}
	\|\mathcal{T}^{\lambda,h}f\|_{L^{6}(\mr^3)}^2	 &\lesssim \lambda^{-\frac13}h^{-\frac43}
		 	\int_{\mr^3}
		 	\int_{\left|\tau-|\bar{\xi}|\right|\leq \lambda^{-1}}
		 	\int_{\left|\sigma-|(\bar{\xi},(h/\lambda)\xi_3)|\right|\leq \lambda^{-1}}
		 	d\tau\,d\sigma\,|\widehat{f}(\xi)|^2\,d\xi \\
		 	&\lesssim
		 	\lambda^{-\frac73}h^{-\frac43}\|f\|_{L^2(\mr^3)}^2,
		 \end{align*}
		 which proves \eqref{ali-L6L2-L2-improve-k}.

	     To show \eqref{ali-Sjk-Linfty-L2}, we note that, by \eqref{fourier-support--Sjk}, the support of $\mathcal F[\mathcal{T}_{\tau,\sigma}^{\lambda,h}f]$ has measure $\lesssim h^{-2}$. Consequently, Minkowski's inequality, the Fourier inversion formula, H\"older's inequality, and Plancherel's theorem yield
		 \begin{align*}
		 	\|\mathcal{T}^{\lambda,h}f\|_{L^{\infty}(\mr^3)}^2&\lesssim h^{-2}\int_{\mr^2}\int_{\mr^3}|\mathcal{T}^{\lambda,h}_{\tau,\sigma}f(x)|^2\,dx\,d\tau\,d\sigma\\
		 	&\lesssim (\lambda h)^{-2}\|f\|_{L^2(\mr^3)}^2.
		 \end{align*}

		Finally, to prove \eqref{ali-main-jk}, we cover $\mr^3$ by finitely overlapping boxes $\{Q\}$ of dimensions $\lambda\times \lambda\times h^2/\lambda$, with their shortest side parallel to the $x_3$-axis. Let $\{\eta_Q\}$ be a partition of unity subordinate to this covering.

		Set $f_Q:=f\eta_Q$ and, for each pair $Q,Q'$, define
	\begin{align*}
		d_{1}(Q,Q')&:=\inf \{|\bar{x}-\bar{x}'|:x\in Q,\ x'\in Q' \},\\
		d_{2}(Q,Q')&:=\inf \{|x_3-x'_3|:x\in Q,\ x'\in Q' \}.
	\end{align*}
	    Fix a small constant $0<c<\min\{\delta,\epsilon\}/10$, and set
	\begin{align*}
	G_{Q'}:=\sum_{\substack{d_1(Q,Q')< \lambda^{1+c}\\\text{and } d_2(Q,Q')< (h^2/\lambda)^{1+c}}}f_Q,
	\qquad
	F_{Q'}:=\sum_{\substack{d_1(Q,Q')\geq \lambda^{1+c}\\\text{or } d_2(Q,Q')\geq (h^2/\lambda)^{1+c}}}f_Q.
	\end{align*}
	Then
		\begin{align}
			\|\mathcal{T}^{\lambda,h}f\|_{L^{\infty}(\mr^3)}
				& \leq \sup_{Q'}\|\mathcal{T}^{\lambda,h}\big( G_{Q'}\big)\|_{L^{\infty}(Q')}+\sup_{Q'}\|\mathcal{T}^{\lambda,h}\big(F_{Q'}\big)\|_{L^{\infty}(Q')}\notag \\
				&\leq (\lambda h^2)^c\sup_{Q} \|\mathcal{T}^{\lambda,h}f_Q\|_{L^{\infty}(\mr^3)}+\sup_{Q'}\|\mathcal{T}^{\lambda,h}\big(F_{Q'}\big)\|_{L^{\infty}(Q')}.\label{far-near-decom}
		\end{align}

			  We show that the second term in \eqref{far-near-decom} is negligible. The covering and its fixed enlargements have bounded overlap, and the number of boxes entering $G_{Q'}$ is $O((\lambda h^2)^c)$.
Note that
		\begin{align*}
			\left|\partial_{\xi}^{\alpha}\left[\psi\bigl(|(\bar{\xi},(\lambda/h)\xi_3)|-\lambda\sigma\bigr)\psi(|\bar{\xi}|-\lambda\tau) \varphi(\lambda^{-1}|\bar{\xi}|) \varphi(\lambda h^{-2}|\xi_{3}|) \right]\right|\leq C_{\alpha}
		\end{align*}
		holds for any $\alpha\in \mathbb{N}^3$. Under the frequency scaling $(\bar\xi,\xi_3)\mapsto(\lambda\bar\xi,(h^2/\lambda)\xi_3)$ appearing in the displayed symbol, the Jacobian factor in the original integral is $(\lambda h^2)^{-1}$, while Lemma~\ref{lem-intersection} shows that the rescaled support has measure $O(\lambda)$. Thus, for $(\tau,\sigma) \in [2^{-2},2^2+2^{-2}]^2$, integration by parts gives
		\[
|K_{\tau,\sigma}^{\lambda,h}(x)|\lesssim_N h^{-2}
\bigl(1+|(\lambda^{-1}\bar{x},\lambda h^{-2}x_3)|\bigr)^{-N},
\]
which holds for every $N>0$. This inequality implies
		\begin{align*}
				|\mathcal{T}_{\tau,\sigma}^{\lambda,h}F_{Q'}(x)|\lesssim h^{-2}\int_{\mr^3}\big(1+|(\lambda^{-1}(\bar{x}-\bar{y}),\lambda h^{-2}(x_3-y_3))|\big)^{-N} |F_{Q'}(y)|\,dy.
		\end{align*}

For $x\in Q'$ and $y$ in the support of $F_{Q'}$, the anisotropic distance in the last display is at least a constant times
$
    \min\{\lambda^c,(h^2/\lambda)^c\}\geq \lambda^{2c\delta}.
$
Moreover, bounded overlap gives $|F_{Q'}(y)|\lesssim\|f\|_\infty$. In the anisotropic change of variables, the volume factor $\lambda h^2$ combines with the kernel factor $h^{-2}$ to give $\lambda$. Decomposing the complement into anisotropic dyadic annuli, its contribution is therefore bounded by a constant times
\[
    \lambda\sum_{2^m\geq\lambda^{2c\delta}}2^{m(3-N)}\|f\|_\infty.
\]
Choosing $N$ sufficiently large in terms of $M,c$, and $\delta$, and then integrating over the fixed compact $(\tau,\sigma)$ set, we obtain, for any $M\geq 1$,
		\begin{equation}\label{error}
				\sup_{Q'}\|\mathcal{T}^{\lambda,h}F_{Q'}\|_{L^{\infty}(Q')}\lesssim_{M} \lambda^{-M}\|f\|_{L^{\infty}(\mr^3)}.
		\end{equation}
		Then, by \eqref{Wjk-Sjk-j/2<k<j}, \eqref{far-near-decom}, and \eqref{error}, the desired estimate \eqref{ali-main-jk} is reduced to
		\begin{align*}
					(\lambda h^2)^c\sup_{Q} \|\mathcal{T}^{\lambda,h}f_Q\|_{L^{\infty}(\mr^3)}\lesssim \lambda^{-\frac{1}{2}+\epsilon}\|f\|_{L^{\infty}(\mr^3)}.
		\end{align*}
		Since $h\leq\lambda$ and $c<\epsilon/10$, we have $(\lambda h^2)^c\leq\lambda^{3c}\leq\lambda^\epsilon$. It therefore remains to show
			\begin{align*}
					\|\mathcal{T}^{\lambda,h}f_Q\|_{L^{\infty}(\mr^3)}\lesssim \lambda^{-\frac12}\|f\|_{L^{\infty}(\mr^3)}.
			\end{align*}
	The bound $\Vert f_Q\Vert_{L^2(\mr^3)}
	\lesssim \lambda^{\frac12}h\Vert f\Vert_{L^{\infty}(\mr^3)}$,
	together with \eqref{ali-Sjk-Linfty-L2} and H\"older's inequality,
	proves the desired estimate.

\end{proof}

\subsubsection*{The regime $\lambda\leq h\leq\lambda^2$}

For $\lambda\leq h\leq\lambda^2$, we obtain the following analogue.
\begin{prop}\label{prop-main-j<k<2j}
	Let $2^3\leq\lambda\leq h\leq\lambda^2$. Then, for any $\epsilon>0$,
    \begin{align}
        \|\mathcal{S}f_{\lambda,h}\|_{L^{\infty}(\mr^3)}&\lesssim_\epsilon \lambda^{\frac{1}{2}+\epsilon}\|f\|_{L^\infty(\mr^3)}, \label{ali-main-j<k<2j} \\
        \|\mathcal{S}f_{\lambda,h}\|_{L^6(\mathbb{R}^3)}&\lesssim \lambda^{\frac16}\|f\|_{L^2(\mr^3)}, \label{ali-L6L2-L2-j-2j} \\
        \|\mathcal{S}f_{\lambda,h}\|_{L^{\infty}(\mathbb{R}^3)}&\lesssim \lambda^{\frac12}\|f\|_{L^2(\mr^3)}. \label{ali-LinftyL2-L2-j-2j}
    \end{align}
\end{prop}

To prove this proposition, we first establish the following analogue of Lemma~\ref{lem-intersection}.
\begin{lemma}\label{lem-intersection-j<k<2j}
	Let $2^{12}\leq\lambda\leq h\leq\lambda^2$. Suppose that $(\tau,\sigma)\in [2^{-2},2^2+2^{-2}]^2$ and
	\begin{align*}
		\psi\bigl(h(|((\lambda/h)\bar{\xi},\xi_{3})|-\sigma)\bigr)\psi(\lambda(|\bar{\xi}|-\tau)) \varphi(|\bar{\xi}|)\varphi(|\xi_{3}|)\neq 0.
	\end{align*}
	Then
	\begin{align}\label{ali-range-barxi-j<k<2j}
		\tau-\lambda^{-1}\leq |\bar{\xi}|\leq \tau+\lambda^{-1},
	\end{align}
	and
	\begin{align}\label{ali-range-xi3-j<k<2j}
		\sqrt{\sigma^2-(\lambda/h)^2\tau^2}-2^8h^{-1} \leq |\xi_3| \leq \sqrt{\sigma^2-(\lambda/h)^2\tau^2}+2^8h^{-1}.
	\end{align}

\end{lemma}

	\begin{proof}
			Put $\tilde{R}=|((\lambda/h)\bar{\xi},\xi_3)|$.
			The support condition on $\psi$ gives
			\[
			\bigl||\bar{\xi}|-\tau\bigr|\leq \lambda^{-1},
		\qquad
			|\tilde{R}-\sigma|\leq h^{-1},
			\]
		which immediately implies \eqref{ali-range-barxi-j<k<2j}. Moreover,
		\[|\xi_3|^2=\tilde{R}^2-(\lambda/h)^2|\bar{\xi}|^2.\]
		Consequently,
		\begin{align*}
			\left||\xi_3|^2-\bigl(\sigma^2-(\lambda/h)^2\tau^2\bigr)\right|
			\leq
			|\tilde{R}-\sigma|(\tilde{R}+\sigma)+(\lambda/h)^2\bigl||\bar{\xi}|-\tau\bigr|(|\bar{\xi}|+\tau).
		\end{align*}
		Using the fact that $(\tau,\sigma)\in [2^{-2},2^2+2^{-2}]^2$, we have
		$\tilde R+\sigma<2^4$ and $|\bar{\xi}|+\tau<2^4$. Since $\lambda\leq h$,
		we obtain
		\[
		\left||\xi_3|^2-\bigl(\sigma^2-(\lambda/h)^2\tau^2\bigr)\right|\leq 2^4h^{-1}+2^4(\lambda/h)^2\lambda^{-1}\leq 2^5h^{-1}.
		\]
			Note that $\varphi(|\xi_3|)\neq 0$ yields $|\xi_3|\geq 2^{-1}$.

			Since $h\geq2^{12}$, the preceding estimate gives
			$\sigma^2-(\lambda/h)^2\tau^2\geq |\xi_3|^2-2^5h^{-1}\geq2^{-2}-2^{-7}>2^{-3}$.
			Hence
		\begin{align*}
			\left||\xi_3|-\sqrt{\sigma^2-(\lambda/h)^2\tau^2}\right|
			\leq\frac{\left||\xi_3|^2-\bigl(\sigma^2-(\lambda/h)^2\tau^2\bigr)\right|}{|\xi_3|+\sqrt{\sigma^2-(\lambda/h)^2\tau^2}}
			\leq 2^8h^{-1}.
		\end{align*}
		This proves \eqref{ali-range-xi3-j<k<2j}.
	\end{proof}

Since the proof of Proposition~\ref{prop-main-j<k<2j} closely follows that of Proposition~\ref{prop-main-jk}, we only indicate the necessary modifications, using the notation from the proof of Proposition~\ref{prop-main-jk}.
\begin{proof}[Proof of Proposition~\ref{prop-main-j<k<2j}]
It suffices to treat the case $\lambda\geq2^{12}$, since the remaining case follows from H\"older's inequality and Young's convolution inequality. Set
		\begin{align*}
				\widetilde{K}_{\tau,\sigma}^{\lambda,h}(x):=\int_{\mr^3}e^{i x\cdot \xi}\psi\bigl(h(|((\lambda/h)\bar{\xi},\xi_{3})|-\sigma)\bigr)\psi(\lambda(|\bar{\xi}|-\tau)) \varphi(|\bar{\xi}|)\varphi(|\xi_{3}|)\,d\xi.
		\end{align*}
		Define the local operators and the associated square function by
		\begin{align*}
			\widetilde{\mathcal{T}}^{\lambda,h}_{\tau,\sigma}f
			&:=\widetilde{K}_{\tau,\sigma}^{\lambda,h}*f,\\
			\widetilde{\mathcal{T}}^{\lambda,h}f(x)
			&:=\left(\int_{\mr^2}|\widetilde{\mathcal{T}}^{\lambda,h}_{\tau,\sigma}f(x)|^2\,d\tau\,d\sigma\right)^{1/2}.
		\end{align*}
    As in the proof of Proposition~\ref{prop-main-jk}, one obtains
    \begin{align}
       \mathcal{S}f_{\lambda,h}(x)
	       &\lesssim \sum_{\ell\in \mathbb{Z}^3}\frac{(\lambda h)^{\frac12}}{(1+|\ell|)^{100}}
       \left(\int_{\mr^2}|\widetilde{\mathcal{T}}^{\lambda,h}_{\tau,\sigma}(\mathcal{D}_{\lambda,h}^\ell f)(\lambda\bar{x},hx_3)|^2\,d\tau\,d\sigma \right)^{1/2}.
       \label{ali-Wjk-tilde-j<k<2j}
    \end{align}
		    Here the spatial Jacobian is again $\lambda^2h$, while Plancherel's theorem in $(t,s)$, with parameter frequencies $\lambda$ and $h$, gives the factor $(\lambda h)^{1/2}$.
		    Thus, the estimates \eqref{ali-main-j<k<2j}, \eqref{ali-L6L2-L2-j-2j}, and \eqref{ali-LinftyL2-L2-j-2j} are reduced to
   \begin{align}\label{ali-square-improve-j<k<2j}
		\|\widetilde{\mathcal{T}}^{\lambda,h}f\|_{L^{\infty}(\mr^3)}\lesssim \lambda^\epsilon h^{-\frac12}\|f\|_{L^{\infty}(\mr^3)},
	\end{align}
   \begin{align}\label{ali-square-L6L2-L2-j-2j}
		\|\widetilde{\mathcal{T}}^{\lambda,h}f\|_{L^6(\mr^3)}\lesssim \lambda^{-1}h^{-\frac56}\|f\|_{L^2(\mr^3)},
   \end{align}
   and
    \begin{align}\label{ali-square-LinftyL2-L2-j-2j}
		\|\widetilde{\mathcal{T}}^{\lambda,h}f\|_{L^{\infty}(\mr^3)}\lesssim (\lambda h)^{-1}\|f\|_{L^{2}(\mr^3)}.
	\end{align}

	    It suffices to consider the parameters for which the multiplier defining $\widetilde{\mathcal T}_{\tau,\sigma}^{\lambda,h}$ is nonzero. The proof of Lemma~\ref{lem-intersection-j<k<2j} then gives $\sigma^2-(\lambda/h)^2\tau^2>2^{-3}$, and for every such $(\tau,\sigma)\in [2^{-2},2^2+2^{-2}]^2$, we have
			 \begin{align}\label{fourier-support--tildeSjk}
			 	\supp \mathcal F[\widetilde{\mathcal{T}}_{\tau,\sigma}^{\lambda,h}f]\subset \{\xi:\bigl||\bar{\xi}|-\tau\bigr|\leq \lambda^{-1}, \bigl||\xi_3|-\sqrt{\sigma^2-(\lambda/h)^2\tau^2}\bigr|\leq 2^8h^{-1}\}.
			 \end{align}
			 In the rescaled frequency variables, the horizontal radial thickness is $O(\lambda^{-1})$, and the vertical thickness is $O(h^{-1})$. Consequently, the support has measure $O((\lambda h)^{-1})$. The admissible parameter set for fixed $\xi$ has the same measure bound.
		 Combining \eqref{fourier-support--tildeSjk} with Minkowski's inequality,
		 Lemma~\ref{lem-annulus-L2-L6}, and Bernstein's inequality, we obtain
		 \begin{align*}
		 	\|\widetilde{\mathcal{T}}^{\lambda,h}f\|_{L^{6}(\mr^3)}^2
		 	&\lesssim
		 	\lambda^{-1}h^{-\frac23}
		 	\int_{\mr^2}\int_{\mr^3}
		 	|\widetilde{\mathcal{T}}^{\lambda,h}_{\tau,\sigma}f(x)|^2
		 	\,dx\,d\tau\,d\sigma.
		 \end{align*}
		 By Plancherel's theorem, this is bounded by
		 \begin{align*}
		 	\lambda^{-1}h^{-\frac23}
		 	\int_{\mr^3}
		 	\int_{\left|\tau-|\bar{\xi}|\right|\leq \lambda^{-1}}
		 	\int_{\left|\sigma-|((\lambda/h)\bar{\xi},\xi_3)|\right|\leq h^{-1}}
		 	d\tau\,d\sigma\,|\widehat{f}(\xi)|^2\,d\xi \lesssim
		 	\lambda^{-2}h^{-\frac53}\|f\|_{L^2(\mr^3)}^2,
		 \end{align*}
		 which proves \eqref{ali-square-L6L2-L2-j-2j}.

	    Note that, by \eqref{fourier-support--tildeSjk}, the measure of $\supp \mathcal F[ \widetilde{\mathcal{T}}_{\tau,\sigma}^{\lambda,h}f ]$ is $\lesssim (\lambda h)^{-1}$. Consequently, Minkowski's inequality, the Fourier inversion formula, H\"older's inequality, and Plancherel's theorem yield
		 \begin{align*}
		 	\|\widetilde{\mathcal{T}}^{\lambda,h}f\|_{L^{\infty}(\mr^3)}^2&\lesssim (\lambda h)^{-1}\int_{\mr^2}\int_{\mr^3}|\widetilde{\mathcal{T}}^{\lambda,h}_{\tau,\sigma}f(x)|^2\,dx\,d\tau\,d\sigma\\
		 	&\lesssim (\lambda h)^{-2}\|f\|_{L^2(\mr^3)}^2.
		 \end{align*}
    This proves \eqref{ali-square-LinftyL2-L2-j-2j}.

		   Finally, to prove \eqref{ali-square-improve-j<k<2j}, we cover $\mr^3$ by finitely overlapping boxes $\{\widetilde Q\}$ of dimensions $\lambda\times \lambda\times h$, with their longest side parallel to the $x_3$-axis. Let $\{\eta_{\widetilde Q}\}$ be a partition of unity subordinate to this covering and set $f_{\widetilde Q}:=f\eta_{\widetilde Q}$.

		   We reuse the preceding localization argument at the anisotropic scale $(\lambda,\lambda,h)$; the only change is that the vertical scale is now $h$. The rescaled kernel satisfies
\[
    |\widetilde K_{\tau,\sigma}^{\lambda,h}(x)|\lesssim_N
	    (\lambda h)^{-1}\bigl(1+|(\lambda^{-1}\bar x,h^{-1}x_3)|\bigr)^{-N}.
\]
		Splitting at horizontal and vertical distances $\lambda^{1+c}$ and $h^{1+c}$, respectively, leaves $O((\lambda^2h)^c)$ near boxes. The far boxes give $O(\lambda^{-M})\|f\|_\infty$ for every $M$ by the same anisotropic-annulus sum as above.

		Since $h\leq\lambda^2$ and $c<\epsilon/10$, we have $(\lambda^2h)^c\leq\lambda^{4c}\leq\lambda^\epsilon$. Hence \eqref{ali-square-improve-j<k<2j} reduces to
   \begin{align*}
			\|\widetilde{\mathcal{T}}^{\lambda,h}f_{\widetilde Q}\|_{L^{\infty}(\mr^3)}\lesssim h^{-\frac12}\|f\|_{L^{\infty}(\mr^3)}.
		\end{align*}
	        Indeed, H\"older's inequality gives
		        $\Vert f_{\widetilde Q}\Vert_{L^2(\mr^3)}\lesssim \lambda h^{\frac12}\Vert f\Vert_{L^{\infty}(\mr^3)}$, and \eqref{ali-square-LinftyL2-L2-j-2j} yields the desired bound.
\end{proof}

\subsection{Sharpness of the intermediate-frequency estimates}
\label{subsec-sharp-intermediate}

We show that the four estimates with $L^2$ input in Propositions
\ref{prop-main-jk} and \ref{prop-main-j<k<2j} are sharp throughout
their stated frequency ranges. For the two estimates with $L^\infty$
input, the common boundary case $h=\lambda$ rules out any uniform
improvement by a fixed negative power of $\lambda$.

\subsubsection*{Sharpness of \eqref{ali-main-jk} and
\eqref{ali-main-j<k<2j}}

Since $\varphi(1)=1$, choose a sufficiently small fixed $\delta_0>0$
and an arc $\Theta\Subset\mathbb S^1$ such that, with
$I_0:=[1-\delta_0,1+\delta_0]$, the conic patch
\[
\Gamma
:=
\{(\rho\omega,\rho u):
\rho,u\in I_0,\ \omega\in\Theta\}
\]
lies in the interior of
$\{\xi:\varphi(|\bar\xi|)\varphi(|\xi_3|)\neq0\}$. Choose
$\chi\in C_c^\infty(\mr^3)$ supported in $\Gamma$ and equal to $1$
on a smaller product patch. Then the amplitudes
$
B_\lambda(\xi,t,s)=\chi(\xi/\lambda)
$
satisfy \eqref{ali-condition-on-B} uniformly in $\lambda$. Let
\[
K_{t,s}^{\lambda}
:=
\mathcal F^{-1}\!\left[
e^{i(t|\bar\xi|+s|\xi|)}
\varphi_\lambda(|\bar\xi|)
\varphi_\lambda(|\xi_3|)
\chi(\xi/\lambda)
\right].
\]

Fix $s_0\in(2^{-3},2^{-2})$ and choose $N\sim\lambda$ points $t_m$
in a compact subinterval of $(1,2)$, separated by a sufficiently
large multiple of $\lambda^{-1}$. After decreasing $\delta_0$ and shrinking $\Theta$ so that $\chi=1$
on the resulting product patch, define
\[
\begin{aligned}
\Omega_m:=\bigg\{&
\left(
\big(t_m+s_0(1+u^2)^{-1/2}+v\big)\omega, \,
s_0u(1+u^2)^{-1/2}
\right):
\\[-2pt]
&\qquad \qquad \qquad \omega\in\Theta,\ u\in I_0,\ |v|\leq c\lambda^{-1}
\bigg\}.
\end{aligned}
\]
These are transverse $c\lambda^{-1}$-neighborhoods of fixed patches
of the stationary surfaces for $K_{t_m,s_0}^{\lambda}(-y)$. For
$y\in\Omega_m$, let
$(\omega_y,u_y,v_y)\in\Theta\times I_0\times[-c\lambda^{-1},c\lambda^{-1}]$
be the unique triple satisfying
\[
R_y:=t_m+s_0(1+u_y^2)^{-1/2}+v_y,
\qquad
\bar y=R_y\omega_y,
\qquad
y_3=s_0u_y(1+u_y^2)^{-1/2}.
\]
To estimate the kernel,
rescale $\xi=\lambda(\rho\omega,\rho u)$. For
$y\in\Omega_m$, write $\omega=\omega(\vartheta)$ in a unit-speed local angular
coordinate, with $\omega(\vartheta_y)=\omega_y$.
Then the oscillatory factor in
$K_{t_m,s_0}^{\lambda}(-y)$ becomes $e^{i\lambda\rho\Psi_y(\vartheta,u)}$, where
\[
\Psi_y(\vartheta,u)
:=t_m+s_0(1+u^2)^{1/2}-\bar y\cdot\omega(\vartheta)-u y_3.
\]
Since $\bar y=R_y\omega_y$ and
$\omega'(\vartheta_y)\perp\omega_y$, we have
$\partial_\vartheta\Psi_y(\vartheta_y,u_y)=0$. The identity
$y_3=s_0u_y(1+u_y^2)^{-1/2}$ also gives
$\partial_u\Psi_y(\vartheta_y,u_y)=0$. Moreover,
$\omega''(\vartheta)=-\omega(\vartheta)$, and the mixed derivative
vanishes. Consequently,
\[
\begin{aligned}
D^2_{\vartheta,u}\Psi_y(\vartheta_y,u_y)
&=
\begin{pmatrix}
R_y&0\\
0&s_0(1+u_y^2)^{-3/2}
\end{pmatrix}.
\end{aligned}
\]
Since
$\det D^2_{\vartheta,u}\Psi_y(\vartheta_y,u_y)\sim1$, the critical point
is uniformly nondegenerate. At this point, the definitions of $R_y$ and
$y_3$ give $\Psi_y(\vartheta_y,u_y)=-v_y$. Thus, stationary phase in
$\vartheta$ and $u$, uniformly for $\rho\in I_0$, yields an expansion of
order $\lambda^2$ whose remaining radial phase is
$e^{-i\lambda\rho v_y}$. Since $|\lambda v_y|\leq c$, choosing the
product patch and $c$ sufficiently small keeps the leading radial
coefficient uniformly away from zero. Therefore,
\[
|K_{t_m,s_0}^{\lambda}(-y)|\gtrsim\lambda^2
\]
for $y\in\Omega_m$.
The rescaling contributes $\lambda^3$, while stationary
phase in two variables contributes $\lambda^{-1}$, giving
$\lambda^2$.
Moreover, $|\Omega_m|\sim\lambda^{-1}$, and the
separation constant may be chosen so that the sets $\Omega_m$ are
pairwise disjoint. Define
\[
F_\lambda(y)
:=
\begin{cases}
\displaystyle
\frac{\overline{K_{t_m,s_0}^{\lambda}(-y)}}
{|K_{t_m,s_0}^{\lambda}(-y)|},&y\in\Omega_m,\\[8pt]
\hspace{20pt} 0,&y\notin\bigcup_m\Omega_m.
\end{cases}
\]
Then $\|F_\lambda\|_\infty\leq1$.
By the definition of $K_{t,s}^{\lambda}$, we have
\[
\mathcal U_{t,s}^{B_\lambda}
\bigl((F_\lambda)_{\lambda,\lambda}\bigr)(0)
=
\int_{\mr^3}K_{t,s}^{\lambda}(-y)F_\lambda(y)\,dy.
\]

For $|t-t_m|\leq c\lambda^{-1}$ and $|s-s_0|\leq c\lambda^{-1}$,
the contribution from $\Omega_m$ remains comparable to $\lambda$.
Let $L$ denote the separation constant in the choice of
the points $t_m$.
For $m'\neq m$ and $y\in\Omega_{m'}$, stationary phase in
$(\vartheta,u)$ leaves a radial phase whose $\rho$-derivative is
$t-t_{m'}+O(c\lambda^{-1})$. Since its absolute value is comparable to
$|t-t_{m'}|$ when $L$ is sufficiently large, repeated integration by
parts in $\rho$ gives, for every $M>0$,
\[
\sup_{y\in\Omega_{m'}}
|K_{t,s}^{\lambda}(-y)|
\lesssim_M
\lambda^2(1+\lambda|t-t_{m'}|)^{-M}.
\]
Since $|\Omega_{m'}|\sim\lambda^{-1}$ and
$\lambda|t-t_{m'}|\gtrsim L|m-m'|$, it follows that
\[
\sum_{m'\neq m}
\left|\int_{\Omega_{m'}}
K_{t,s}^{\lambda}(-y)F_\lambda(y)\,dy\right|
\lesssim_M
\lambda\sum_{\ell\neq0}(1+L|\ell|)^{-M}.
\]

Choosing $L$ sufficiently large makes this smaller than the
contribution from $\Omega_m$.
Hence
$
\left|
\mathcal U_{t,s}^{B_\lambda}
\bigl((F_\lambda)_{\lambda,\lambda}\bigr)(0)
\right|
\gtrsim\lambda
$
on $N\sim\lambda$ disjoint parameter rectangles of measure
comparable to $\lambda^{-2}$. Consequently,
\[
\left\|
\mathcal U_{t,s}^{B_\lambda}
\bigl((F_\lambda)_{\lambda,\lambda}\bigr)(0)
\right\|_{L^2_{t,s}(\mathbb J_0)}
\gtrsim
\lambda\bigl(\lambda\lambda^{-2}\bigr)^{1/2}
=\lambda^{1/2}.
\]
Since $h=\lambda$ belongs to the ranges of both propositions, this
rules out a uniform improvement by any fixed negative power of
$\lambda$ in \eqref{ali-main-jk} or \eqref{ali-main-j<k<2j}.

\subsubsection*{Sharpness of \eqref{ali-L6L2-L2-j/2-j} and
\eqref{ali-L6L2-L2-j-2j}}

Take $B\equiv1$. Fix positive constants $\tau_0,\zeta_0$ in the
interior of the support of $\varphi$, an arc
$\Theta\Subset\mathbb S^1$, and $(t_0,s_0)$ in the interior of
$\mathbb J_0$. In either of the ranges
\[
\lambda^{1/2+\delta}\leq h\leq\lambda
\qquad\text{or}\qquad
\lambda\leq h\leq\lambda^2,
\]
set
\[
R_{\lambda,h}
:=
\sqrt{(\lambda\tau_0)^2+(h\zeta_0)^2},
\qquad
d_{\lambda,h}:=\max\{1,\lambda/h\},
\]
and, for a sufficiently small fixed $c>0$, let
\[
E_{\lambda,h}
:=
\left\{\xi:
\bigl||\bar\xi|-\lambda\tau_0\bigr|\leq c,\quad
\bigl||\xi|-R_{\lambda,h}\bigr|\leq c,\quad
\frac{\bar\xi}{|\bar\xi|}\in\Theta,\quad
\xi_3>0
\right\}.
\]

For $c$ sufficiently small, $E_{\lambda,h}$ lies in
$\mathbb A_\lambda\times\mathbb B_h$. Indeed, in the coordinates
$r=|\bar\xi|$ and $R=|\xi|$,
$d\xi=r\tfrac{R}{\sqrt{R^2-r^2}}\,dr\,dR\,d\omega$.
On $E_{\lambda,h}$, $r\sim\lambda$ and
$R(R^2-r^2)^{-1/2}\sim d_{\lambda,h}$, and hence
\[
|E_{\lambda,h}|\sim\lambda d_{\lambda,h}.
\]

Define
\[
\widehat f_{\lambda,h}(\xi)
:=
e^{-i(t_0|\bar\xi|+s_0|\xi|)}
\chi_{E_{\lambda,h}}(\xi).
\]
Plancherel's theorem gives
\[
\|f_{\lambda,h}\|_2
\sim(\lambda d_{\lambda,h})^{1/2}.
\]
The definition of $E_{\lambda,h}$ also gives
$|\xi_3-h\zeta_0|\lesssim d_{\lambda,h}$. After removing a constant
phase, the nonconstant part is
\[
\bar x\cdot\bar\xi+x_3(\xi_3-h\zeta_0)
+(t-t_0)(|\bar\xi|-\lambda\tau_0)
+(s-s_0)(|\xi|-R_{\lambda,h}).
\]
If $(t,s)$ lies in a sufficiently small fixed neighborhood of
$(t_0,s_0)$, $|\bar x|\leq c\lambda^{-1}$, and
$|x_3|\leq c d_{\lambda,h}^{-1}$, then
$
|\mathcal U_{t,s}f_{\lambda,h}(x)|
\gtrsim |E_{\lambda,h}|
\sim\lambda d_{\lambda,h}.
$
Consequently,
\[
\mathcal Sf_{\lambda,h}(x)
\gtrsim\lambda d_{\lambda,h}
\]
throughout a spatial region of measure comparable to
$\lambda^{-2}d_{\lambda,h}^{-1}$.
It follows that
\[
\frac{\|\mathcal Sf_{\lambda,h}\|_6}
{\|f_{\lambda,h}\|_2}
\gtrsim
(\lambda d_{\lambda,h})^{1/2}
(\lambda^{-2}d_{\lambda,h}^{-1})^{1/6}
=
\lambda^{1/6}d_{\lambda,h}^{1/3}.
\]
For $h\leq\lambda$, this becomes
$\lambda^{1/2}h^{-1/3}$, matching
\eqref{ali-L6L2-L2-j/2-j}; for $h\geq\lambda$, it becomes
$\lambda^{1/6}$, matching \eqref{ali-L6L2-L2-j-2j}.

\subsubsection*{Sharpness of \eqref{ali-LinftyL2-L2-j/2-j} and
\eqref{ali-LinftyL2-L2-j-2j}}

We use the same $f_{\lambda,h}$. Since the preceding pointwise lower
bound holds on a spatial set of positive measure, it gives
\[
\frac{\|\mathcal Sf_{\lambda,h}\|_\infty}
{\|f_{\lambda,h}\|_2}
\gtrsim
(\lambda d_{\lambda,h})^{1/2}.
\]
For $h\leq\lambda$, this is $\lambda h^{-1/2}$ and matches
\eqref{ali-LinftyL2-L2-j/2-j}. For $h\geq\lambda$, it is
$\lambda^{1/2}$ and matches \eqref{ali-LinftyL2-L2-j-2j}. Thus all
four estimates with $L^2$ input are sharp throughout their stated
frequency ranges. The characteristic functions used above may be
replaced by harmless smooth cutoffs.

\subsection{The high-vertical-frequency regime}

 It remains to handle the regime $h\geq\lambda^2$.
 \begin{prop}\label{prop-L6L2-L2-k>2j}
       Let $\lambda\geq2^3$ and $h\geq\lambda^2$. Then
       \begin{align*}
	       	\|\mathcal{S}f_{\lambda,h}\|_{L^{6}(\mathbb{R}^{3})}\lesssim \lambda^{\frac16}\|f\|_{L^2(\mr^3)}.
       \end{align*}
       In addition, for any $h\geq2^3$,
       \begin{align*}
           \|\mathcal{S}f_{\leq 2^3,h}\|_{L^6(\mr^3)}\lesssim \|f\|_{L^2(\mr^3)}.
       \end{align*}
   \end{prop}
   \begin{proof}
       We first prove the estimate for $\mathcal{S}f_{\lambda,h}$ with
       $\lambda\geq2^3$. Let $\widetilde{\varphi}\in
       C_c^{\infty}((1/4,3))$ satisfy
       $\widetilde{\varphi}\varphi=\varphi$. After a change of variables,
       \begin{align*}
       \mathcal{U}_{t,s}f_{\lambda,h}(x)
       &=\lambda^2h\int_{\mr^3}e^{ix\cdot(\lambda\bar\xi,h\xi_3)}
       e^{i(t\lambda|\bar\xi|+sh|\xi_3|)}B_{\lambda,h,3}(\xi,t,s)
       \widehat{f_{\lambda,h}}(\lambda\bar\xi,h\xi_3)\,d\xi,
       \end{align*}
       where
       \begin{align*}
       B_{\lambda,h,3}(\xi,t,s)
       &:={e^{is\left(|(\lambda\bar\xi,h\xi_3)|-h|\xi_3|\right)}}
       \widetilde{\varphi}(|\bar\xi|)\widetilde{\varphi}(|\xi_3|)
       B(\lambda\bar\xi,h\xi_3,t,s).
       \end{align*}
       On the support of this amplitude,
       $|(\lambda\bar\xi,h\xi_3)|-h|\xi_3|$ equals
       $\lambda^2|\bar\xi|^2/(|(\lambda\bar\xi,h\xi_3)|+h|\xi_3|)$.
       Hence its derivatives are $O(\lambda^2/h)=O(1)$, and
       \eqref{ali-condition-on-B} gives
       $\sup_\xi|\partial_\xi^\alpha
       B_{\lambda,h,3}(\xi,t,s)|\leq C_\alpha$ uniformly on
       $\mathbb J_0$.

       Expanding $B_{\lambda,h,3}$ in a Fourier series gives coefficients
       satisfying $|C_{\ell,3}(t,s)|\lesssim(1+|\ell|)^{-100}$. Define
       \begin{align*}
       \widetilde{\mathcal L}f(x,t,s)
       :=\int_{\mr^3}e^{i(x\cdot\xi+t|\bar\xi|+s|\xi_3|)}
       \widehat f(\xi)\,d\xi.
       \end{align*}
       Reversing the change of variables, each Fourier mode is a spatial
       translate of $\widetilde{\mathcal L}f_{\lambda,h}$ by
       $(\lambda^{-1}\ell_1,\lambda^{-1}\ell_2,h^{-1}\ell_3)$; thus no
       scale factor is lost. Minkowski's inequality and translation
       invariance give
       \begin{align*}
       \|\mathcal Sf_{\lambda,h}\|_{L^6(\mr^3)}
       \lesssim
       \|\widetilde{\mathcal L}f_{\lambda,h}\|_
       {L_x^6(\mr^3;L_{t,s}^2(\mathbb J_0))}.
       \end{align*}

       Define $f^\pm$ by
       $\widehat{f^\pm}(\xi)=\chi_{(0,\infty)}(\pm\xi_3)\widehat f(\xi)$.
       Then
       \begin{align*}
       \widetilde{\mathcal L}f(x,t,s)
       =\sum_\pm\mathcal W_+
       \bigl(f^\pm(\cdot,x_3\pm s)\bigr)(\bar x,t).
       \end{align*}
       Applying the lossless $(2,6)$ Stein--Tomas endpoint used in the proof
       of Lemma~\ref{lem-one-qp-p}, together with Minkowski's inequality,
       yields
       \begin{align*}
       \|\widetilde{\mathcal L}f_{\lambda,h}\|_
       {L_x^6(\mr^3;L_{t,s}^2(\mathbb J_0))}
       \lesssim \lambda^{1/6}\sum_\pm
       \left\|\|f_{\lambda,h}^\pm(\bar x,x_3\pm s)\|_
       {L_{x_3}^6(\mr;L_s^2(\mathbb I_{-3}))}
       \right\|_{L_{\bar x}^2(\mr^2)}.
       \end{align*}
	       For each fixed $\bar x$, Fubini's theorem gives the corresponding
	       bound in $L_{x_3}^2(L_s^2)$. Restricting the $s$-integral to
	       $\mathbb I_{-3}$ gives the bound in $L_{x_3}^\infty(L_s^2)$.
	       Interpolating these bounds and applying Plancherel's theorem to
	       the projections onto $\{\pm\xi_3>0\}$ gives
       \begin{align*}
       \|f_{\lambda,h}^\pm(\bar x,x_3\pm s)\|_
       {L_{x_3}^6(\mr;L_s^2(\mathbb I_{-3}))}
       \lesssim\|f_{\lambda,h}(\bar x,\cdot)\|_{L^2(\mr)}.
       \end{align*}
       This proves the first estimate.

       The same argument applies to $f_{\leq2^3,h}$. Indeed, the horizontal
       frequency remains in a fixed ball, so the phase correction has
       uniformly bounded derivatives, and the bounded-frequency analogue of
       the preceding $(2,6)$ estimate follows from Bernstein's inequality and
       Plancherel's theorem. This proves the second estimate.
\end{proof}

\subsection{Square function estimates for averages over tori}

Propositions~\ref{prop-main-jk}, \ref{prop-main-j<k<2j}, and
\ref{prop-L6L2-L2-k>2j} yield the following square-function estimates
for averages over tori.
Recall that $P_1=(0,0)$, $P_6=(1/2,1/6)$, and $Q_1=(1/2,0)$.
\begin{cor}\label{cor-L6L2-L2}
	Let $\alpha,\beta\in\{0,1\}$, $\lambda\geq2^3$, $\delta>0$, and
	$(1/p,1/q)\in\{P_1,P_6,Q_1\}$. For any $\epsilon>0$, the following
	hold.
	\begingroup
	\addtolength{\leftmargini}{-2em}
	\begin{enumerate}
		\item[(a)] If $\lambda^{1/2+\delta}\leq h\leq\lambda$ and
		$\supp\widehat f\subset\mathbb A_\lambda\times\mathbb B_h$, then
		\begin{align*}
		\|\partial_t^\alpha\partial_s^\beta\mathcal Af\|_
		{L^q(\mr^3;L^2(\mathbb J_0))}
		&\lesssim_{\delta,\epsilon,\alpha,\beta}
		\lambda^{\alpha+\beta+\frac1p-\frac3q-\frac12
		+\epsilon(1-\frac2p)}
		h^{-\frac1p+\frac1q}\|f\|_{L^p(\mr^3)}.
		\end{align*}

		\item[(b)] If $\lambda\leq h\leq\lambda^2$ and
		$\supp\widehat f\subset\mathbb A_\lambda\times\mathbb B_h$, then
		\begin{align*}
		\|\partial_t^\alpha\partial_s^\beta\mathcal Af\|_
		{L^q(\mr^3;L^2(\mathbb J_0))}
		&\lesssim_{\epsilon,\alpha,\beta}
		\lambda^{\alpha-\frac2q+\epsilon(1-\frac2p)}
		h^{\beta-\frac12}\|f\|_{L^p(\mr^3)}.
		\end{align*}

		\item[(c)] If $(1/p,1/q)=P_6$, the estimate in part (b) remains
		valid for $h\geq\lambda^2$ under the same Fourier-support assumption.
		Moreover, if $h\geq2^3$ and
		$\supp\widehat f\subset\mathbb A_{2^3}^{\circ}\times\mathbb B_h$,
		then
		\begin{align*}
		\|\partial_t^\alpha\partial_s^\beta\mathcal Af\|_
		{L^6(\mr^3;L^2(\mathbb J_0))}
		&\lesssim_{\alpha,\beta}
		h^{\beta-\frac12}\|f\|_{L^2(\mr^3)}.
		\end{align*}
	\end{enumerate}
	\endgroup
\end{cor}
\begin{proof}
	By \eqref{ali-expansion-sigmats} and the support properties stated
	after that identity, we consider separately the terms associated with
	$B_{0,0}$, $B_{0,1}^{\varepsilon}$, and
	$B_{1,1}^{\kappa,\kappa'}$, where
	$\varepsilon,\kappa,\kappa'\in\{+,-\}$.

	The term associated with $B_{0,0}$ occurs only in the fixed
	low-frequency region and is covered by the fixed-frequency estimates
		in Lemma~\ref{lem-A1-A2}. The terms associated with $B_{0,1}^{\varepsilon}$ are supported where
	$|\bar\xi|\leq2^3$ and $|\xi_3|>2^2$, and hence vanish on
	$\mathbb A_\lambda\times\mathbb B_h$ when $\lambda\geq2^4$. On
	$\mathbb A_{2^3}^{\circ}\times\mathbb B_h$ with $h\geq2^3$,
	\eqref{asm-B-symbol} gives
	\[|B_{0,1}^{\varepsilon}(\xi,t,s)|\lesssim h^{-1/2}.\] A $t$-derivative
	preserves this bound, while an $s$-derivative costs at most a factor of
	order $h$.
	Let $\mathcal{A}_{0,1}^{\varepsilon}$ denote the operator with
	multiplier
	$e^{\varepsilon is|\xi_3|}B_{0,1}^{\varepsilon}(\xi,t,s)$.
	After separating the two signs of $\xi_3$, the one-dimensional
	translation estimate established in the proof of Proposition
	\ref{prop-L6L2-L2-k>2j}, together with Bernstein's inequality in the
	compact horizontal-frequency variables, gives
	\[
	\bigl\|
	\partial_t^\alpha\partial_s^\beta
	\mathcal{A}_{0,1}^{\varepsilon}f
	\bigr\|_{L^6(\mathbb{R}^3;L^2(\mathbb{J}_0))}
	\lesssim
	h^{\beta-\frac12}
	\|f\|_{L^2(\mathbb{R}^3)}.
	\]
	Thus the $B_{0,1}^{\varepsilon}$ terms satisfy the second statement of
	part (c). Since they vanish when $\lambda\geq2^4$, the same estimate
	also covers their possible $\lambda=2^3$ contribution to the first
	statement of part (c). In parts (a) and (b), the remaining
	$B_{0,1}^{\varepsilon}$ terms occur for only finitely many dyadic pairs
	$(\lambda,h)$ and are therefore covered by Lemma~\ref{lem-A1-A2},
	after adjusting the constants.

	After these low-horizontal-frequency terms have been handled, the
	main contributions in the horizontal high-frequency regimes are the
	doubly oscillatory terms
	\[
	e^{\kappa it|\bar{\xi}|}
	e^{\kappa' i s|\xi|}
	B_{1,1}^{\kappa,\kappa'}(\xi,t,s).
	\]
	The proofs of Propositions~\ref{prop-main-jk},
	\ref{prop-main-j<k<2j}, and \ref{prop-L6L2-L2-k>2j} depend only on
	the support geometry and symbol bounds, so their estimates apply
	uniformly to all four choices of $(\kappa,\kappa')$.
	On $\mathbb{A}_\lambda\times\mathbb{B}_h$, \eqref{asm-B-symbol} gives
	\[
	\bigl|
	B_{1,1}^{\kappa,\kappa'}(\xi,t,s)
	\bigr|
	\lesssim
	\begin{cases}
		\lambda^{-1}, & h\leq\lambda,\\[2pt]
		(\lambda h)^{-1/2}, & h\geq\lambda.
	\end{cases}
	\]
	Moreover, differentiation with respect to $t$ introduces at most a
	factor of order $\lambda$. Differentiation with respect to $s$ introduces
	at most a factor of order $\lambda$ when $h\leq\lambda$ and at most a factor
	of order $h$ when $h\geq\lambda$.

	Combining these amplitude and derivative bounds with the three
	propositions gives the asserted estimates. Multiplying the three estimates
	in Proposition~\ref{prop-main-jk} by $\lambda^{\alpha+\beta-1}$ gives the
	formula in part (a) at $P_1$, $P_6$, and $Q_1$, respectively. Likewise,
	multiplying the three estimates in Proposition
	\ref{prop-main-j<k<2j} by
	$\lambda^{\alpha-1/2}h^{\beta-1/2}$ gives the formula in part (b) at
	the same three points. The extension asserted in the first statement
	of part (c) follows from Proposition
	\ref{prop-L6L2-L2-k>2j} with the latter factor.
	For the low-horizontal estimate in part (c), \eqref{asm-B-symbol} gives
	$|B_{1,1}^{\kappa,\kappa'}(\xi,t,s)|\lesssim h^{-1/2}$; differentiation
	in $t$ costs $O(1)$ and differentiation in $s$ costs $O(h)$. The second
	estimate in Proposition~\ref{prop-L6L2-L2-k>2j} therefore applies with
	the factor $h^{\beta-1/2}$.

	Derivatives falling on any of the symbols, including the smooth
	remainders in \eqref{ali-expansion-sigmats}, are controlled by
	\eqref{asm-B-symbol}.
\end{proof}

Lemmas~\ref{lem-A1-A2} and \ref{lem-A4-A6}, together with Corollary~\ref{cor-L6L2-L2}, provide all the estimates used as interpolation inputs in Section~\ref{sec-Lp-Lq-bounds}.

 \section{\texorpdfstring{Mixed-norm interpolation and dyadic summation}{Mixed-norm interpolation and dyadic summation}}\label{sec-Lp-Lq-bounds}
In this section, we prove the $L^p$--$L^q$ estimates for the interior variation $\widetilde{V}_r(\mathcal{A})$ needed for the sufficient part of Theorem~\ref{thm-Lp-Lq}.

\subsection{Mixed-norm interpolation estimates}
The estimates in Lemmas~\ref{lem-A1-A2} and \ref{lem-A4-A6}
at $P_1,P_4,$ and $P_5$, together with the square-function
estimates in Corollary~\ref{cor-L6L2-L2} at $P_6$ and $Q_1$,
give the interpolation data at the five points used below. At each
of these points, the parameter exponent $r$ equals the input exponent
$p$, so the interpolation takes place on the plane $1/r=1/p$.

For points $R_1,\ldots,R_N$ in the $(1/p,1/q)$-plane, $N\geq2$,
we write
\[[R_1,\ldots,R_N]:=\operatorname{conv}\{R_1,\ldots,R_N\}.\]
When $N=2$, $[R_1,R_2]$ is the closed line segment joining $R_1$
and $R_2$. With this notation, the projected interpolation region
is divided into the three triangles $[P_5,P_6,P_4]$,
$[P_5,Q_1,P_6]$, and $[P_5,P_1,Q_1]$.

The exponents in the estimates below are the affine functions
determined by these interpolation data. We use the standard complex
interpolation theorem for mixed-norm Lebesgue spaces. Since the input
estimates hold with arbitrarily small losses, the resulting losses are
absorbed into the $\epsilon$-loss appearing below.
The relevant region is illustrated in Figure~\ref{fig:mixed-norm-region}.

	\begin{prop}\label{prop-LqLp-Lp-one}
	    Let  $\alpha,\beta\in \{0,1\}$. 	   The following hold.
	    \begingroup
	    \addtolength{\leftmargini}{-2em}
        \begin{enumerate}
	        \item[(a)]  Let $(1/p,1/q)\in[P_1,Q_1,P_4]$ and
        $\delta,\epsilon>0$. If $\lambda\geq2^3$ and $\supp \widehat{f}\subset \mathbb{A}_\lambda\times \mathbb{B}_{\lambda^{1/2+\delta}}^{\circ}$, then we have
        \begin{align*}
			\|\partial_t^{\alpha}\partial_s^{\beta}\mathcal{A}f\|_{L^{q}(\mathbb{R}^{3};L^p(\mathbb{J}_0))}\lesssim
			\begin{cases}
					\lambda^{\alpha+\beta-\frac{1}{2}+\frac{1}{2p}-\frac{5}{2q}+300\delta+\epsilon} \|f\|_{L^p(\mathbb{R}^3)},
                & \frac{1}{p}+\frac{3}{q}\leq 1,\\[2pt]
	                \lambda^{\alpha+\beta+\frac{1}{p}-\frac{1}{q}-1+300\delta+\epsilon}\|f\|_{L^p(\mathbb{R}^3)},
                & \frac{1}{p}+\frac{3}{q}>1.
			\end{cases}
		\end{align*}
         \item[(b)] Let $(1/p,1/q)\in[P_1,P_6,P_4]$ and
        $\epsilon>0$. If $\lambda\geq2^3$, $h\geq\lambda^2$, and $\supp \widehat{f}\subset \mathbb{A}_\lambda\times \mathbb{B}_h$, then we have
			\begin{align*}
				\|\partial_t^{\alpha}\partial_s^{\beta}\mathcal{A}f\|_{L^{q}(\mathbb{R}^{3};L^p(\mathbb{J}_0))}\lesssim
				\begin{cases}
					\lambda^{\alpha-\frac{2}{q}+\epsilon}h^{\beta-\frac{1}{2}}\|f\|_{L^p(\mathbb{R}^3)},
                    & \frac{1}{p}+\frac{3}{q}\leq 1,\\[2pt]
                    \lambda^{\alpha-\frac{1}{2}+\frac{1}{2p}-\frac{1}{2q}+\epsilon}h^{\beta-\frac{1}{2}}\|f\|_{L^p(\mathbb{R}^3)},
                    & \frac{1}{p}+\frac{3}{q}>1.
				\end{cases}
			\end{align*}

            \item[(c)] Let $(1/p,1/q)\in[P_1,P_6,P_4]$. If $h\geq2^3$ and $\supp \widehat{f}\subset \mathbb{A}_{2^3}^{\circ}\times \mathbb{B}_h$, then we have
			    \begin{align*}
				   \|\partial_t^{\alpha}\partial_s^{\beta}\mathcal{A}f\|_{L^{q}(\mathbb{R}^{3};L^p(\mathbb{J}_0))}\lesssim h^{\beta-\frac{1}{2}}\|f\|_{L^p(\mr^3)}.
			    \end{align*}
		\end{enumerate}
		\endgroup
\end{prop}

\begin{proof}
Part (a) follows from Lemma~\ref{lem-one-qp-p} by the same multiplier
decomposition and symbol estimates used in the proof of
Lemma~\ref{lem-A4-A6} (d). Since the range
$2\leq p\leq q\leq\infty$ corresponds to
$(1/p,1/q)\in[P_1,Q_1,P_4]$, this proves part (a) throughout its
stated region.

    For part (b), the estimates at $(p,q)=(2,2),(4,4),$ and $(\infty,\infty)$ follow from Lemma~\ref{lem-A4-A6} (c), whereas the estimates at the same three points for part (c) follow from Lemma~\ref{lem-A1-A2} (b). In both parts, the estimate at $(p,q)=(2,6)$ follows from Corollary~\ref{cor-L6L2-L2} (c).
    Interpolation gives the stated estimates throughout $[P_1,P_6,P_4]$.
\end{proof}
  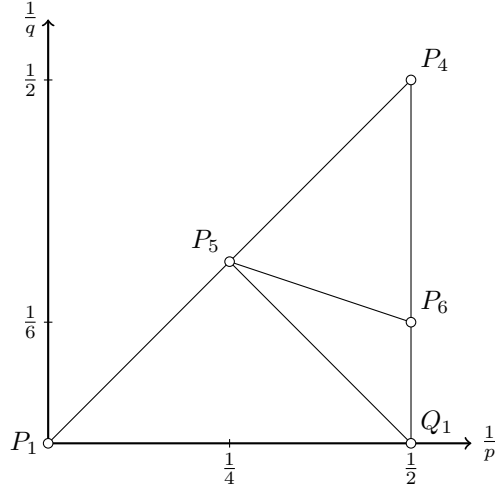
\begin{figure}[t]
	\centering
	\begin{tikzpicture}[scale=0.8]
		\draw[thick,->]  (0,0) -- (7,0);
		\draw[thick,->] (0,0) -- (0,7);
		\node[left] at (0,7) {$\frac{1}{q}$};
        \node[right] at (7,0) {$\frac{1}{p}$};
		\node[left] at (0,6) {$\frac{1}{2}$};
		\node[left] at (0,2) {$\frac{1}{6}$};
		\node[below] at (3,0) {$\frac{1}{4}$};
		\node[below] at (6,0) {$\frac{1}{2}$};
		\draw (6,-2pt) -- (6,2pt);
		\draw (-2pt,2) -- (2pt,2);
		\draw (3,-2pt) -- (3,2pt);
		\draw (-2pt,6) -- (2pt,6);

		\draw (0,0) -- (6,6) -- (6,0);
		\draw (3,3) -- (6,2) ;
		\draw (3,3) -- (6,0) ;

		\node[left] at (0,0) {$P_{1}$};
		\node[above right] at (6,0) {$Q_1$};
		\node[above right] at (6,2) {$P_6$};
		\node[above right] at (6,6) {$P_4$};
		\node[above left] at (3,3) {$P_5$};

		\node[circle, draw=black, fill=white, inner sep=0pt,minimum size=3.6pt] at (0,0) {};
		\node[circle,draw=black, fill=white, inner sep=0pt,minimum size=3.6pt] at (6,0) {};
		\node[circle,draw=black, fill=white, inner sep=0pt,minimum size=3.6pt] at (6,2) {};

		\node[circle,draw=black, fill=white, inner sep=0pt,minimum size=3.6pt] at (3,3) {};

		\node[circle,draw=black, fill=white, inner sep=0pt,minimum size=3.6pt] at (6,6) {};
	\end{tikzpicture}
	\caption{The exponent region used for the mixed-norm estimates.}
	\label{fig:mixed-norm-region}
\end{figure}

\begin{prop}\label{prop-LqLp-Lp}
	Let $\lambda\geq2^3$ and $\alpha,\beta\in \{0,1\}$. In each estimate below, the first, second, and third lines apply, respectively, when $(1/p,1/q)\in[P_5,P_6,P_4]$, $(1/p,1/q)\in[P_5,Q_1,P_6]$, and $(1/p,1/q)\in[P_5,P_1,Q_1]$. For any $\delta,\epsilon>0$, the following hold:
	\begingroup
	\addtolength{\leftmargini}{-2em}
	\begin{enumerate}
		\item[(a)] If $\lambda^{1/2+\delta}\leq h\leq\lambda$ and $\supp \widehat{f}\subset \mathbb{A}_\lambda\times \mathbb{B}_h$, then we have
			\begin{align*}
				\|\partial_t^{\alpha}\partial_s^{\beta}\mathcal{A}f\|_{L^{q}(\mathbb{R}^{3};L^p(\mathbb{J}_0))}\lesssim
				\begin{cases}
					\lambda^{\alpha+\beta+\frac{3}{2p}-\frac{3}{2q}-1+\epsilon}h^{-\frac{1}{p}+\frac{1}{q}}\|f\|_{\Lp(\mr^3)},\\[2pt]
					\lambda^{\alpha+\beta-\frac{1}{2}+\frac{1}{p}-\frac{3}{q}+\epsilon}h^{-\frac{1}{p}+\frac{1}{q}} \|f\|_{\Lp(\mr^3)},\\[2pt]
					\lambda^{\alpha+\beta+\frac{2}{p}-\frac{2}{q}-1+\epsilon}h^{1-\frac{3}{p}-\frac{1}{q}} \|f\|_{\Lp(\mr^3)}.
				\end{cases}
			\end{align*}
			\item[(b)] If $\lambda\leq h\leq\lambda^2$ and $\supp \widehat{f}\subset \mathbb{A}_\lambda\times \mathbb{B}_h$, then we have
			\begin{align*}
				\|\partial_t^{\alpha}\partial_s^{\beta}\mathcal{A}f\|_{L^{q}(\mathbb{R}^{3};L^p(\mathbb{J}_0))}\lesssim
				\begin{cases}
					\lambda^{\alpha+\frac{1}{2p}-\frac{1}{2q}-\frac{1}{2}+\epsilon}h^{\beta-\frac{1}{2}}\|f\|_{\Lp(\mr^3)},\\[2pt]
					\lambda^{\alpha-\frac{2}{q}+\epsilon}h^{\beta-\frac{1}{2}}\|f\|_{\Lp(\mr^3)},\\[2pt]
					\lambda^{\alpha-\frac{2}{p}-\frac{4}{q}+1+\epsilon}h^{\beta+\frac{1}{p}+\frac{1}{q}-1}\|f\|_{\Lp(\mr^3)}.
				\end{cases}
			\end{align*}
	\end{enumerate}
	\endgroup
\end{prop}
\begin{proof}
    For part (a), interpolate the estimates in Lemma~\ref{lem-A4-A6} (e) at $P_4,P_5,$ and $P_1$ with those in Corollary~\ref{cor-L6L2-L2} (a) at $P_6$ and $Q_1$. Part (b) follows in the same way from Lemma~\ref{lem-A4-A6} (a) and Corollary~\ref{cor-L6L2-L2} (b).
\end{proof}

\subsection{Mixed-norm bounds in the three exponent regions}

We next interpolate in the parameter exponent to obtain the following mixed-norm bounds.

\begin{cor}\label{cor-pqr-top}
    Let $\lambda\geq2^3$, $\alpha,\beta\in \{0,1\}$, $(1/p,1/q)\in [P_5,P_6,P_4]$, and $p\leq r \leq q$.
    For any $\delta,\epsilon>0$, the following hold:
	    \begingroup
	    \addtolength{\leftmargini}{-2em}
	    \begin{enumerate}
	        \item[(a)] If $\supp \widehat{f}\subset \mathbb{A}_\lambda\times \mathbb{B}_{\lambda^{1/2+\delta}}^{\circ}$, then we have
			\begin{align*}
						\|\partial_t^{\alpha}\partial_s^\beta\mathcal{A}f\|_{L^q(\mr^3;L^r(\mathbb{J}_0))}\lesssim \lambda^{\alpha+\beta+\frac{2}{p}-\frac{1}{q}-\frac{1}{r}-1+300\delta+\epsilon}\|f\|_{\Lp(\mr^3)}.
			\end{align*}

			\item[(b)] If $\lambda^{1/2+\delta}\leq h\leq\lambda$ and $\supp \widehat{f}\subset \mathbb{A}_\lambda\times \mathbb{B}_h$, then we have
			\begin{align*}
				\|\partial_t^{\alpha}\partial_s^\beta\mathcal{A}f\|_{L^q(\mr^3;L^r(\mathbb{J}_0))}\lesssim \lambda^{\alpha+\beta+\frac{2}{p}-\frac{3}{2q}-\frac{1}{2r}-1+\epsilon}h^{-\frac{1}{r}+\frac{1}{q}}\|f\|_{\Lp(\mr^3)}.
			\end{align*}
			\item[(c)] If $\lambda\leq h\leq\lambda^2$ and $\supp \widehat{f}\subset \mathbb{A}_\lambda\times \mathbb{B}_h$, then we have
			\begin{align*}
					\|\partial_t^\alpha\partial_s^\beta \mathcal{A}f\|_{L^q(\mr^3;L^r(\mathbb{J}_0))}\lesssim \lambda^{\alpha+\frac{1}{2p}-\frac{1}{2q}-\frac{1}{2}+\epsilon}h^{\beta-\frac{1}{2}+\frac{3}{2p}-\frac{3}{2r}}\|f\|_{\Lp(\mr^3)}.
			\end{align*}
	    \end{enumerate}
	    \endgroup
\end{cor}
\begin{proof}
    Part (a) follows by interpolating Lemma~\ref{lem-A4-A6} (d) with Proposition~\ref{prop-LqLp-Lp-one} (a). Parts (b) and (c) follow by interpolating Lemma~\ref{lem-A4-A6} (e) and (a), respectively, with the upper-triangle estimates in Proposition~\ref{prop-LqLp-Lp} (a) and (b).
\end{proof}
\begin{cor}\label{cor-pqr-mid}
    Let $\lambda\geq2^3$, $\alpha,\beta\in\{0,1\}$, $(1/p,1/q)\in[P_5,Q_1,P_6]$, and $p\leq r\leq q$.
    Define $r_{p,q}$\footnote{For $(1/p,1/q)\in[P_5,Q_1,P_6]$, one has $p\leq r_{p,q}\leq q$.} by
        \[
            \frac{1}{r_{p,q}}:=1-\frac{1}{p}-\frac{2}{q}.
        \]

	For any $\delta,\epsilon>0$, the following hold:
	    \begingroup
	    \addtolength{\leftmargini}{-2em}
	    \begin{enumerate}
            \item[(a)] If $\supp \widehat{f}\subset \mathbb{A}_\lambda\times \mathbb{B}_{\lambda^{1/2+\delta}}^{\circ}$, then
			\begin{align*}
					 \|\partial_t^\alpha\partial_s^\beta \mathcal{A}f\|_{L^q(\mr^3;L^r(\mathbb{J}_0))}\lesssim \lambda^{\alpha+\beta-\frac{1}{2}+\frac{3}{2p}-\frac{5}{2q}-\frac{1}{r}+300\delta+\epsilon}\|f\|_{\Lp(\mr^3)}.
			\end{align*}

			\item[(b)] Suppose that $\lambda^{1/2+\delta}\leq h\leq\lambda$ and $\supp \widehat{f}\subset \mathbb{A}_\lambda\times \mathbb{B}_h$. Then
			\begin{align*}
				\|\partial_t^\alpha\partial_s^\beta \mathcal{A}f\|_{L^q(\mr^3;L^r(\mathbb{J}_0))}
				\lesssim
				\begin{cases}
					\lambda^{\alpha+\beta-\frac{1}{2}-\frac{1}{2r}+\frac{3}{2p}-\frac{3}{q}+\epsilon}h^{-\frac{1}{r}+\frac{1}{q}}\|f\|_{\Lp(\mr^3)},
					& p\leq r\leq r_{p,q},\\[2pt]
					\lambda^{\alpha+\beta+\frac{2}{p}-\frac{2}{q}-1+\epsilon}h^{1-\frac{1}{p}-\frac{1}{q}-\frac{2}{r}}\|f\|_{\Lp(\mr^3)},
					& r_{p,q}\leq r\leq q.
				\end{cases}
			\end{align*}
			\item[(c)] If $\lambda\leq h\leq\lambda^2$, $\supp \widehat{f}\subset \mathbb{A}_\lambda\times \mathbb{B}_h$, and $p\leq r\leq r_{p,q}$, then
			\begin{align*}
					\|\partial_t^\alpha\partial_s^\beta \mathcal{A}f\|_{L^q(\mr^3;L^r(\mathbb{J}_0))}\lesssim \lambda^{\alpha-\frac{2}{q}+\epsilon}h^{\beta-\frac{3}{2r}-\frac{1}{2}+\frac{3}{2p}}\|f\|_{\Lp(\mr^3)}.
			\end{align*}
	        \end{enumerate}
	        \endgroup
\end{cor}
\begin{proof}
    Part (a) follows by interpolating Lemma~\ref{lem-A4-A6} (d) with Proposition~\ref{prop-LqLp-Lp-one} (a). For parts (b) and (c), we first establish the following two estimates at $r=r_{p,q}$. If $\lambda^{1/2+\delta}\leq h\leq\lambda$ and $\supp \widehat{f}\subset \mathbb{A}_\lambda\times\mathbb{B}_h$, then
        \begin{align}\label{ali-q-r_pq-p-j/2}
            \|\partial_t^{\alpha}\partial_s^{\beta}\mathcal{A}f\|_{L^{q}(\mathbb{R}^{3}; L^{r_{p,q}}(\mathbb{J}_0))}\lesssim \lambda^{\alpha+\beta+\frac{2}{p}-\frac{2}{q}-1+\epsilon}h^{\frac{1}{p}+\frac{3}{q}-1}\|f\|_{\Lp(\mr^3)}.
        \end{align}
    If $\lambda\leq h\leq\lambda^2$ under the same Fourier-support assumption, then
        \begin{align}\label{ali-q-r_pq-p-j}
            \|\partial_t^{\alpha}\partial_s^{\beta}\mathcal{A}f\|_{L^{q}(\mathbb{R}^{3}; L^{r_{p,q}}(\mathbb{J}_0))}\lesssim \lambda^{\alpha-\frac{2}{q}+\epsilon}h^{\beta+\frac{3}{p}+\frac{3}{q}-2}\|f\|_{\Lp(\mr^3)}.
        \end{align}
    For \eqref{ali-q-r_pq-p-j/2}, the cases $(p,q,r_{p,q})=(4,4,4)$ and $(2,6,6)$ follow from Lemma~\ref{lem-A4-A6} (e), while $(2,\infty,2)$ follows from Corollary~\ref{cor-L6L2-L2} (a). For \eqref{ali-q-r_pq-p-j}, the first two cases follow from Lemma~\ref{lem-A4-A6} (a), and the third follows from Corollary~\ref{cor-L6L2-L2} (b). Interpolation in $(1/p,1/q)$ proves the two displayed estimates. Interpolating \eqref{ali-q-r_pq-p-j/2} with Proposition~\ref{prop-LqLp-Lp} (a) for $p\leq r\leq r_{p,q}$ and with Lemma~\ref{lem-A4-A6} (e) for $r_{p,q}\leq r\leq q$ gives part (b). Finally, interpolating \eqref{ali-q-r_pq-p-j} with Proposition~\ref{prop-LqLp-Lp} (b) gives part (c).
\end{proof}

\begin{cor}\label{cor-pqr-below}
    Let $\lambda\geq2^3$, $\alpha,\beta\in\{0,1\}$, and $(1/p,1/q)\in[P_5,P_1,Q_1]$.     Define $r_q$\footnote{For $(1/p,1/q)\in[P_5,P_1,Q_1]$, one has $r_q\leq p\leq q$.} by
        \[
            \frac{1}{r_q}:=\frac{1}{2}-\frac{1}{q}.
        \]
	For any $\delta,\epsilon>0$, the following hold:
	    \begingroup
	    \addtolength{\leftmargini}{-2em}
	    \begin{enumerate}
            \item[(a)] If $\supp \widehat{f}\subset \mathbb{A}_\lambda\times \mathbb{B}_{\lambda^{1/2+\delta}}^{\circ}$ and $p\leq r\leq q$, then
			\begin{align*}
					 \|\partial_t^\alpha\partial_s^\beta \mathcal{A}f\|_{L^q(\mr^3;L^r(\mathbb{J}_0))}\lesssim \lambda^{\alpha+\beta-\frac{1}{2}+\frac{3}{2p}-\frac{5}{2q}-\frac{1}{r}+300\delta+\epsilon}\|f\|_{\Lp(\mr^3)}.
			\end{align*}

			\item[(b)] Suppose that $\lambda^{1/2+\delta}\leq h\leq\lambda$ and $\supp \widehat{f}\subset \mathbb{A}_\lambda\times \mathbb{B}_h$. Then
			\begin{align*}
				\|\partial_t^\alpha\partial_s^\beta \mathcal{A}f\|_{L^q(\mr^3;L^r(\mathbb{J}_0))}
				\lesssim
				\begin{cases}
					\lambda^{\alpha+\beta+\frac{1}{r}+\frac{1}{p}-\frac{2}{q}-1+\epsilon}h^{-\frac{2}{r}-\frac{1}{p}-\frac{1}{q}+1}\|f\|_{\Lp(\mr^3)},
					& r_q\leq r\leq p,\\[2pt]
					\lambda^{\alpha+\beta+\frac{2}{p}-\frac{2}{q}-1+\epsilon}h^{1-\frac{1}{p}-\frac{1}{q}-\frac{2}{r}}\|f\|_{\Lp(\mr^3)},
					& p\leq r\leq q.
				\end{cases}
			\end{align*}

			\item[(c)] If $\lambda\leq h\leq\lambda^2$, $\supp \widehat{f}\subset \mathbb{A}_\lambda\times \mathbb{B}_h$, and $r_q\leq r\leq p$, then
			\begin{align*}
					\|\partial_t^\alpha\partial_s^\beta \mathcal{A}f\|_{L^q(\mr^3;L^r(\mathbb{J}_0))}\lesssim \lambda^{\alpha+1-\frac{2}{r}-\frac{4}{q}+\epsilon}h^{\beta-1+\frac{1}{r}+\frac{1}{q}}\|f\|_{\Lp(\mr^3)}.
			\end{align*}
	        \end{enumerate}
	        \endgroup
\end{cor}
\begin{proof}
    Part (a) follows by interpolating Lemma~\ref{lem-A4-A6} (d) with Proposition~\ref{prop-LqLp-Lp-one} (a). The estimate in part (b) for $p\leq r\leq q$ follows by interpolating Lemma~\ref{lem-A4-A6} (e) with the lower-triangle estimate in Proposition~\ref{prop-LqLp-Lp} (a). For the remaining range in part (b) and for part (c), we first establish estimates at $r=r_q$. If $\lambda^{1/2+\delta}\leq h\leq\lambda$ and $\supp \widehat{f}\subset \mathbb{A}_\lambda\times \mathbb{B}_h$, then
        \begin{align}\label{ali-q-r_q-p-j/2}
            \|\partial_t^{\alpha}\partial_s^{\beta}\mathcal{A}f\|_{L^{q}(\mathbb{R}^{3};L^{r_q}(\mathbb{J}_0))}\lesssim \lambda^{\alpha+\beta+\frac{1}{p}-\frac{3}{q}-\frac{1}{2}+\epsilon}h^{-\frac{1}{p}+\frac{1}{q}}\|f\|_{\Lp(\mr^3)}.
        \end{align}
    If $\lambda\leq h\leq\lambda^2$ under the same Fourier-support assumption, then
        \begin{align}\label{ali-q-r_q-p-j}
            \|\partial_t^{\alpha}\partial_s^{\beta}\mathcal{A}f\|_{L^{q}(\mathbb{R}^{3};L^{r_q}(\mathbb{J}_0))}\lesssim \lambda^{\alpha-\frac{2}{q}+\epsilon}h^{\beta-\frac{1}{2}}\|f\|_{\Lp(\mr^3)}.
        \end{align}
    For \eqref{ali-q-r_q-p-j/2}, the case $(p,q,r_q)=(4,4,4)$ follows from Lemma~\ref{lem-A4-A6} (e), while the cases $(\infty,\infty,2)$ and $(2,\infty,2)$ follow from Corollary~\ref{cor-L6L2-L2} (a). For \eqref{ali-q-r_q-p-j}, the first case follows from Lemma~\ref{lem-A4-A6} (a), and the latter two follow from Corollary~\ref{cor-L6L2-L2} (b). Interpolation in $(1/p,1/q)$ proves the two displayed estimates. Interpolating \eqref{ali-q-r_q-p-j/2} with the lower-triangle estimate in Proposition~\ref{prop-LqLp-Lp} (a) gives the estimate in part (b) for $r_q\leq r\leq p$. Interpolating \eqref{ali-q-r_q-p-j} with the corresponding estimate in Proposition~\ref{prop-LqLp-Lp} (b) gives part (c).
\end{proof}

\subsection{Frequency decomposition of the interior variation}\label{subsect-frequency-decomposition}
We decompose the interior variation according to the horizontal and
vertical frequency scales.

    We first separate the terms with bounded horizontal frequency. Define
    \begin{align*}
        I_r^{\mathrm{low}}(x)
        &:=\widetilde{V}_r(\mathcal A f_{\leq2^3,\leq2^3})(x)
        +\sum_{h>2^3}\widetilde{V}_r(\mathcal A f_{\leq2^3,h})(x),\\
        I_r^{\mathrm{high}}(x)
        &:=\sum_{\lambda>2^3}\bigl[\widetilde{V}_r(\mathcal A f_{\lambda,<\lambda^2})(x)
        +\widetilde{V}_r(\mathcal A f_{\lambda,\geq\lambda^2})(x)\bigr].
    \end{align*}
    By subadditivity,
    \[
        \widetilde{V}_r(\mathcal A f)(x)\leq I_r^{\mathrm{low}}(x)+I_r^{\mathrm{high}}(x).
    \]
  For $r>2$, $2<p\leq q$, and
$(1/p,1/q)\in[P_1,P_6,P_4]$,  the contribution from bounded horizontal frequencies satisfies the required $L^p$--$L^q$ bound.
    Indeed, if $2<r\leq p\leq q$, Lemma~\ref{lem-embedding}, Lemma~\ref{lem-A1-A2} (a), Proposition~\ref{prop-LqLp-Lp-one} (c), and H\"older's inequality give the following estimate.
    \[
        \|I_r^{\mathrm{low}}\|_{L^q(\mr^3)}
        \lesssim\Bigl(1+\sum_{h>2^3}h^{\frac1r-\frac12}\Bigr)\|f\|_{L^p(\mr^3)}
        \lesssim\|f\|_{L^p(\mr^3)}.
    \]
    When $r>p$, applying the monotonicity $\widetilde{V}_r\leq\widetilde{V}_p$ to each frequency piece reduces the estimate to the preceding case with variation exponent $p$.
    The case $p=q=r=\infty$ follows by choosing any finite $\rho>2$ and using $\widetilde{V}_\infty\leq\widetilde{V}_\rho$ termwise.

     Thus it remains to estimate the high-horizontal contribution $I_r^{\mathrm{high}}$, which we decompose further.
    For $0<\delta<1/2$, define
	\[I_{r,\delta}^{\mathrm{hd}}(x)
	:=\sum_{\lambda>2^3}\widetilde{V}_{r}
	\bigl(\mathcal{A}f_{\lambda,\leq \lambda^{1/2+\delta}}\bigr)(x),\qquad
	I_{r,\delta}^{\mathrm{tr}}(x)
	:=\sum_{\lambda>2^3}\sum_{\lambda^{1/2+\delta}<h\leq\lambda}
	\widetilde{V}_{r}(\mathcal{A}f_{\lambda,h})(x),\]
	and
	\[I_{r}^{\mathrm{mid}}(x):=\sum_{\lambda>2^3}\sum_{\lambda<h<\lambda^2}\widetilde{V}_{r}(\mathcal{A}f_{\lambda,h})(x), \qquad I_{r}^{\mathrm{vert}}(x):=\sum_{\lambda>2^3}\sum_{h\geq\lambda^2}\widetilde{V}_{r}(\mathcal{A}f_{\lambda,h})(x).\]
    These four dyadic index ranges are disjoint and cover all pieces with $\lambda>2^3$. Hence, by subadditivity,
    \[
        I_r^{\mathrm{high}}(x)\leq I_{r,\delta}^{\mathrm{hd}}(x)+I_{r,\delta}^{\mathrm{tr}}(x)+I_r^{\mathrm{mid}}(x)+I_r^{\mathrm{vert}}(x).
    \]
		The superscripts ``hd'', ``tr'', ``mid'', and ``vert'' stand for the horizontally dominant, transition, intermediate, and vertical regimes, respectively.

\begin{figure}[t]
		\centering
		\begin{tikzpicture}[scale=.8]
			\draw[thick,->]  (0,0) -- (7,0);
			\draw[thick,->] (0,0) -- (0,7);
			\node[left] at (0,7) {$\frac{1}{q}$};
			\node[left] at (0,6) {$\frac{1}{2}$};
			\node[right] at (7,0) {$\frac{1}{p}$};
			\draw (-2pt,6) -- (2pt,6);
			\draw (6,-2pt) -- (6,2pt);
			\node[below] at (6,-0.1) {$\frac{1}{2}$};

			\draw[thin] (0,0) -- (6,6) -- (60/11,24/11) -- (36/7,12/7) -- (0,0);
			\draw[thin] (3,3) -- (60/11,24/11) ;
            \draw[thin] (3,3) -- (9/2,3/2) ;
            \draw[densely dotted] (9/2,3/2) -- (6,0) ;
			\node[left] at (0,0) {$P_{1}$};
			\node[above right] at (6,6) {$P_4$};
            \node[above left] at (3,3) {$P_5$};
            \node[above right] at (60/11,24/11) {$P_3$};
            \node[below] at (36/7+1/10,12/7-1/10) {$P_2$};
            \node[below] at (9/2,3/2-1/10) {$P_7$};
            \node[above] at (6+1/10,0) {$Q_1$};

			\node[circle, draw=black, fill=white, inner sep=0pt,minimum size=3.6pt] at (0,0) {};
			\node[circle,draw=black, fill=white, inner sep=0pt,minimum size=3.6pt] at (60/11,24/11) {};
			\node[circle,draw=black, fill=white, inner sep=0pt,minimum size=3.6pt] at (3,3) {};
			\node[circle,draw=black, fill=white, inner sep=0pt,minimum size=3.6pt] at (36/7,12/7) {};
			\node[circle,draw=black, fill=white, inner sep=0pt,minimum size=3.6pt] at (6,6) {};
			\node[circle,draw=black, fill=white, inner sep=0pt,minimum size=3.6pt] at (9/2,3/2) {};
			\node[circle,draw=black, fill=white, inner sep=0pt,minimum size=3.6pt] at (6,0) {};
		\end{tikzpicture}
		\caption{The $(1/p,1/q)$ region used in the proof.}
		\label{fig:proof-region}
	\end{figure}
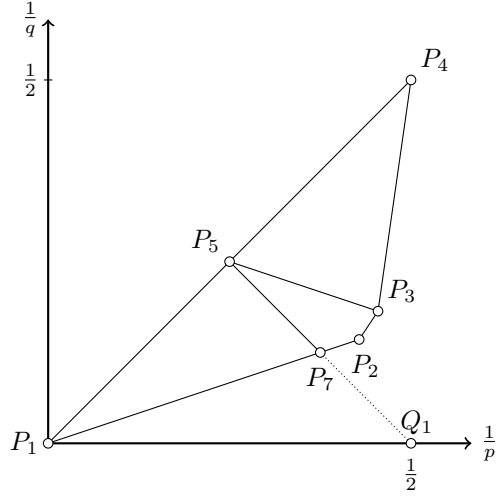

\subsection{Frequency summation in the three exponent regions}\label{subsect-Lpq-variation}
For the subdivision of the exponent region used below, set
\[ P_7:=(3/8,1/8). \]  (See Figure~\ref{fig:proof-region}.)
Recall the affine functions $L_1,\ldots,L_5$ defined in
\eqref{ali-affine-functions}. We also introduce the auxiliary affine
functions
\begin{equation}\label{ali-auxiliary-affine-functions}
\begin{aligned}
    M_1(p,q)&:=-\frac{7}{p}+\frac{1}{q}+3,\\
    M_2(p,q)&:=-\frac{3}{p}+\frac{2}{q}+1,\\
    M_3(p,q)&:=-\frac{1}{p}+\frac{3}{q}.
\end{aligned}
\end{equation}
Their zero sets are the lines containing $[P_3,P_4]$,
$[P_2,P_3]$, and $[P_1,P_2]$, respectively. The next three
propositions treat the top, middle, and lower exponent regions shown
in Figure~\ref{fig:proof-region}, in that order.

In the proofs of these propositions, $0<\delta<1/2$ is the
frequency-threshold parameter introduced above, while $\epsilon>0$
denotes the arbitrarily small loss in the frequency-localized estimates.
For each fixed
$(p,q,r)$ in the relevant range, we first
choose $\delta$ and then $\epsilon$ sufficiently small so that the
resulting dyadic sums converge.

\begin{prop}\label{prop-top}
    Let $r>2$, and let $(1/p,1/q)\in [P_5,P_3,P_4]$ but not on the closed segment $[P_3,P_4]$.
    Then $\widetilde{V}_r(\mathcal{A}):L^p(\mr^3)\to L^q(\mr^3)$ is bounded provided one of the following conditions holds.
    \begin{enumerate}
       \item[(i)] $r>q$;
        \item[(ii)] $p\leq r\leq q$ and $L_3(p,q,r)>0$;
        \item[(iii)] $r<p$ and $L_5(p,q,r)>0$.
    \end{enumerate}
\end{prop}
\begin{proof}
	 We first treat case (ii). We use the elementary dyadic summation estimate
	 \begin{align}\label{est-dyadic-sum}
	 	\sum_{\lambda^a<h<\lambda^b}h^c\lesssim
        \begin{cases}
            \lambda^{bc},&c>0,\\[2pt]
            \lambda^{ac},&c<0,\\[2pt]
            1+\log\lambda,&c=0,
        \end{cases}
        \qquad 0<a<b<\infty.
	 \end{align}
    The logarithmic factor when $c=0$ is bounded by an arbitrarily small power of $\lambda$. For case (ii), Lemma~\ref{lem-embedding}, Corollary~\ref{cor-pqr-top}, Proposition~\ref{prop-LqLp-Lp-one} (b), and \eqref{est-dyadic-sum} give
    \begin{align*}
        \|I_{r,\delta}^{\mathrm{hd}}+I_{r,\delta}^{\mathrm{tr}}\|_{L^q}
        &\lesssim \sum_{\lambda>2^3}\lambda^{-L_3(p,q,r)+C(\delta+\epsilon)}\|f\|_{L^p},\\
        \|I_r^{\mathrm{mid}}+I_r^{\mathrm{vert}}\|_{L^q}
        &\lesssim \sum_{\lambda>2^3}\left(
        \lambda^{-\frac12L_3(p,q,r)+\frac1p-\frac12+C\epsilon}
        +\lambda^{-\frac12M_1(p,q)+C\epsilon}\right)\|f\|_{L^p}.
    \end{align*}
    The affine function $M_1$ defined in
    \eqref{ali-auxiliary-affine-functions} is positive on
    $[P_5,P_3,P_4]\setminus[P_3,P_4]$, and this region also satisfies
    $p>2$. Together with $L_3(p,q,r)>0$, the positivity of $M_1(p,q)$ and the inequality $p>2$ make all three exponents negative after choosing $\delta$ and then $\epsilon$ sufficiently small. This proves case (ii).

    Case (i) follows from case (ii) with variation exponent $q$, since $L_3(p,q,q)=1-2/p>0$, and from the monotonicity $\widetilde{V}_r\leq\widetilde{V}_q$ for $r>q$.

    For case (iii), Lemma~\ref{lem-embedding}, Propositions~\ref{prop-LqLp-Lp-one} and \ref{prop-LqLp-Lp}, and H\"older's inequality give
    \begin{align*}
        \|I_{r,\delta}^{\mathrm{hd}}+I_{r,\delta}^{\mathrm{tr}}+I_r^{\mathrm{mid}}+I_r^{\mathrm{vert}}\|_{L^q}
        &\lesssim
        \sum_{\lambda>2^3}\lambda^{-L_5(p,q,r)+C(\delta+\epsilon)}\|f\|_{L^p}\\
        &\qquad+\sum_{\lambda>2^3}
        \lambda^{-\frac12L_5(p,q,r)+\frac1r-\frac12+C(\delta+\epsilon)}
        \|f\|_{L^p}.
    \end{align*}
    Both exponents are negative after $\delta$ and $\epsilon$ are chosen sufficiently small, because $L_5(p,q,r)>0$ and $r>2$. This proves case (iii).
\end{proof}

\begin{prop}\label{prop-mid}
    Let $r>2$, and let $(1/p,1/q)\in [P_5,P_7,P_2,P_3]$ but not on the closed segments $[P_2,P_3]$ or $[P_7,P_2]$.
    Then $\widetilde{V}_r(\mathcal{A}):L^p(\mr^3)\to L^q(\mr^3)$ is bounded provided one of the following conditions holds.
    \begin{enumerate}
       \item[(i)] $r>q$;
        \item[(ii)] $p\leq r\leq q$ and $L_2(p,q,r)>0$;
        \item[(iii)] $r<p$ and $L_4(p,q,r)>0$.
    \end{enumerate}
\end{prop}
\begin{proof}
    The affine functions $M_2$ and $M_3$ defined in
    \eqref{ali-auxiliary-affine-functions} are positive on the stated
    quadrilateral away from the excluded segments. For case (ii), Lemma~\ref{lem-embedding}, Corollary~\ref{cor-pqr-mid}, and \eqref{est-dyadic-sum} give
    \[
        \|I_{r,\delta}^{\mathrm{hd}}+I_{r,\delta}^{\mathrm{tr}}\|_{L^q}
        \lesssim\sum_{\lambda>2^3}\left(
        \lambda^{-L_2(p,q,r)+C(\delta+\epsilon)}
        +\lambda^{-M_3(p,q)+C(\delta+\epsilon)}
        \right)\|f\|_{L^p}.
    \]
    For the intermediate regime, the two ranges in Corollary~\ref{cor-pqr-mid} yield the following estimates. If $r_{p,q}\leq r\leq q$, then
    \[
        \|I_r^{\mathrm{mid}}\|_{L^q}\lesssim
        \sum_{\lambda>2^3}\left(
        \lambda^{-M_2(p,q)+C\epsilon}+\lambda^{-M_3(p,q)+C\epsilon}
        \right)\|f\|_{L^p}.
    \]
    If $p\leq r\leq r_{p,q}$, then
    \[
        \|I_r^{\mathrm{mid}}\|_{L^q}\lesssim
        \sum_{\lambda>2^3}\left(
        \lambda^{-\frac12L_2(p,q,r)-\frac12L_2(p,q,q)+C\epsilon}
        +\lambda^{-M_2(p,q)+C\epsilon}
        \right)\|f\|_{L^p}.
    \]
    Finally, Proposition~\ref{prop-LqLp-Lp-one} (b) and $\widetilde{V}_r\leq\widetilde{V}_p$ give
    \[
        \|I_r^{\mathrm{vert}}\|_{L^q}\lesssim
        \sum_{\lambda>2^3}\lambda^{-M_2(p,q)+C\epsilon}\|f\|_{L^p}.
    \]
    By the hypothesis in case (ii), $L_2(p,q,r)>0$, while the geometric restrictions on $(1/p,1/q)$ give $L_2(p,q,q)>0$, $M_2(p,q)>0$, and $M_3(p,q)>0$. Hence, choosing $\delta$ and then $\epsilon$ sufficiently small makes all the preceding series converge. This proves case (ii). Case (i) follows by taking the variation exponent to be $q$, since $L_2(p,q,q)>0$ throughout the stated quadrilateral, and using $\widetilde{V}_r\leq\widetilde{V}_q$ for $r>q$.

    For case (iii), Lemma~\ref{lem-embedding}, Propositions~\ref{prop-LqLp-Lp-one} and \ref{prop-LqLp-Lp}, and H\"older's inequality give
    \begin{align*}
        \|I_{r,\delta}^{\mathrm{hd}}+I_{r,\delta}^{\mathrm{tr}}+I_r^{\mathrm{mid}}+I_r^{\mathrm{vert}}\|_{L^q}
        &\lesssim
        \sum_{\lambda>2^3}\lambda^{-L_4(p,q,r)+C(\delta+\epsilon)}\|f\|_{L^p}\\
        &\qquad+\sum_{\lambda>2^3}
        \lambda^{-L_4(p,q,r)+\frac12(\frac1q-\frac1p)+C(\delta+\epsilon)}
        \|f\|_{L^p}.
    \end{align*}
    Since $p\leq q$, both exponents are negative when $L_4(p,q,r)>0$ and $\delta,\epsilon$ are sufficiently small. This proves case (iii).
\end{proof}

\begin{prop}\label{prop-below}
    Let $r>2$, and let $(1/p,1/q)\in [P_5,P_1,P_7]$ but not on the closed segment $[P_1,P_7]$.
    Then $\widetilde{V}_r(\mathcal{A}):L^p(\mr^3)\to L^q(\mr^3)$ is bounded provided one of the following conditions holds.
    \begin{enumerate}
        \item[(i)] $r>q$;
        \item[(ii)] $p\leq r\leq q$ and $L_2(p,q,r)>0$;
        \item[(iii)] $r_q\leq r\leq p$, $L_1(p,q,r)>0$, and $L_4(p,q,r)>0$;
        \item[(iv)] $r<r_q$ and $L_4(p,q,r)>0$.
    \end{enumerate}
\end{prop}
\begin{proof}
    Recall the affine functions $M_2$ and $M_3$ defined in
    \eqref{ali-auxiliary-affine-functions}. On the stated triangle, $M_2>0$, while $M_3>0$ away from the excluded segment $[P_1,P_7]$. The estimate for $I_r^{\mathrm{vert}}$ follows by the same argument as in the proof of Proposition~\ref{prop-mid}, using $M_2$ in cases (i)--(ii) and $L_4(p,q,r)$ in cases (iii)--(iv).

    We first treat case (ii). Lemma~\ref{lem-embedding}, Corollary~\ref{cor-pqr-below}, and \eqref{est-dyadic-sum} give
    \[
        \|I_{r,\delta}^{\mathrm{hd}}+I_{r,\delta}^{\mathrm{tr}}\|_{L^q}
        \lesssim\sum_{\lambda>2^3}\left(
        \lambda^{-L_2(p,q,r)+C(\delta+\epsilon)}
        +\lambda^{-M_3(p,q)+C(\delta+\epsilon)}
        \right)\|f\|_{L^p},
    \]
    while $\widetilde{V}_r\leq\widetilde{V}_p$ and the $r=p$ estimate in Corollary~\ref{cor-pqr-below} give
    \[
        \|I_r^{\mathrm{mid}}\|_{L^q}
        \lesssim\sum_{\lambda>2^3}\left(
        \lambda^{-M_2(p,q)+C\epsilon}+\lambda^{-M_3(p,q)+C\epsilon}
        \right)\|f\|_{L^p}.
    \]
    Thus, case (ii) follows from $L_2(p,q,r)>0$, $M_2(p,q)>0$, and $M_3(p,q)>0$. Case (i) follows by taking the variation exponent to be $q$, since $L_2(p,q,q)>0$ on the stated triangle, and using $\widetilde{V}_r\leq\widetilde{V}_q$ for $r>q$.

    For case (iii), the same embedding and Corollary~\ref{cor-pqr-below} give
    \[
        \|I_{r,\delta}^{\mathrm{hd}}+I_{r,\delta}^{\mathrm{tr}}\|_{L^q}
        \lesssim\sum_{\lambda>2^3}\left(
        \lambda^{-L_4(p,q,r)+C(\delta+\epsilon)}
        +\lambda^{-L_1(p,q,r)+C(\delta+\epsilon)}
        \right)\|f\|_{L^p}.
    \]
    Moreover, $r\geq r_q$ implies $2/r+1/q-1\leq0$. Hence the sum in $h$ for the intermediate regime is controlled by its lower endpoint and
    \[
        \|I_r^{\mathrm{mid}}\|_{L^q}\lesssim
        \sum_{\lambda>2^3}\lambda^{-L_1(p,q,r)+C\epsilon}\|f\|_{L^p}.
    \]
    These estimates prove case (iii) under the stated conditions $L_1(p,q,r)>0$ and $L_4(p,q,r)>0$.

    Finally, in case (iv), Lemma~\ref{lem-embedding}, Corollary~\ref{cor-pqr-below}, and H\"older's inequality yield
    \begin{align*}
        \|I_{r,\delta}^{\mathrm{hd}}+I_{r,\delta}^{\mathrm{tr}}+I_r^{\mathrm{mid}}\|_{L^q}
        &\lesssim\sum_{\lambda>2^3}\lambda^{-L_4(p,q,r)+C(\delta+\epsilon)}\|f\|_{L^p}\\
        &\qquad+\sum_{\lambda>2^3}
        \lambda^{-L_4(p,q,r)+\frac12(\frac1q-\frac1p)+C(\delta+\epsilon)}
        \|f\|_{L^p}.
    \end{align*}
    Since $p\leq q$, both sums converge when $L_4(p,q,r)>0$ and $\delta,\epsilon$ are sufficiently small. This proves case (iv) and completes the proof.
\end{proof}

As shown in Figure~\ref{fig:proof-region}, the three closed regions
$[P_5,P_3,P_4]$,
$[P_5,P_7,P_2,P_3]$, and $[P_5,P_1,P_7]$ cover
$\overline{\mathcal Q}$. Their common edges $[P_5,P_3]$ and
$[P_5,P_7]$ are included, whereas the excluded portions are precisely
the three outer edges
\[
    [P_1,P_2]\cup[P_2,P_3]\cup[P_3,P_4].
\]
Except for $P_1$, these edges lie outside $\mathcal Q$, and $P_1$
does not belong to $\mathfrak P(r)$ because $L_1(p,q,r)>0$ fails
there. Consequently, every point of $\mathfrak P(r)$ lies in the
geometric domain of at least one of Propositions
\ref{prop-top}--\ref{prop-below}.

For every $(1/p,1/q)\in\mathfrak P(r)$, the applicable case is determined as follows. In Proposition~\ref{prop-top}, use case (i) when $r>q$, case (ii) when $p\leq r\leq q$, and case (iii) when $r<p$. The same correspondence applies to cases (i)--(iii) in Proposition~\ref{prop-mid}. In Proposition~\ref{prop-below}, use case (i) when $r>q$, case (ii) when $p\leq r\leq q$, case (iii) when $r_q\leq r<p$, and case (iv) when $r<r_q$. The geometry of the corresponding region gives the required positivity of the auxiliary affine functions $M_1,M_2,$ and $M_3$ defined in \eqref{ali-auxiliary-affine-functions}, while the defining inequalities of $\mathfrak P(r)$ give $L_1,\ldots,L_5>0$. Consequently, the interior variation is bounded throughout $\mathfrak P(r)$.

To pass from the interior variation to $V_r(\mathcal A)$, it remains to control the two one-parameter boundary variations. These are treated in the next section.

 \section{One-parameter variational inequalities}\label{sec-one-para}

In this section, we establish the one-parameter variational
inequalities required for the two boundary terms of
$V_r(\mathcal A)$ and for $V_r(\mathcal A_1)$. We first treat the
boundary variations $[\mathcal A f(x,1,\cdot)]_{r,1}$ and
$[\mathcal A f(x,\cdot,2^{-3})]_{r,1}$, and then turn to the
one-parameter operator $\mathcal A_1$.

Throughout this section, we use the following boundary quantities:
\[
\begin{aligned}
\mathfrak B_{q,r}^{t,\alpha}(f)
&:=
\|\partial_t^\alpha\mathcal Af(\cdot,\cdot,2^{-3})\|_{L^q(\mr^3;L^r(\mathbb I))},\\
\mathfrak B_{q,r}^{s,\alpha}(f)
&:=
\|\partial_s^\alpha\mathcal Af(\cdot,1,\cdot)\|_{L^q(\mr^3;L^r(\mathbb I_{-3}))}.
\end{aligned}
\]

\subsection{Mixed-norm estimates for the boundary terms}

\begin{lem}\label{lem-boundary}
    Let $(1/p,1/q)\in [P_1,P_6,P_4]$ and $\alpha\in \{0,1\}$.
    For any $\epsilon>0$, the following hold:
    \begingroup
    \addtolength{\leftmargini}{-2em}
    \begin{enumerate}
        \item[(a)] If $\supp \widehat{f}\subset \mathbb{A}_{2^3}^{\circ}\times \mathbb{B}_{2^3}^{\circ}$, then
        \begin{align*}
            \mathfrak B_{q,p}^{t,\alpha}(f)
            +\mathfrak B_{q,p}^{s,\alpha}(f)
            \lesssim \|f\|_{L^p(\mr^3)}.
        \end{align*}
        \item[(b)] If $\lambda\geq2^3$, $h\geq\lambda^2$, and $\supp \widehat{f}\subset \mathbb{A}_\lambda\times \mathbb{B}_h$, then
            \begin{align*}
                \mathfrak B_{q,p}^{t,\alpha}(f)
                &\lesssim
                \begin{cases}
                    \lambda^{\alpha-\frac{2}{q}+\epsilon}h^{\frac{1}{p}-\frac{1}{2}}\|f\|_{L^p(\mr^3)},
                    & \frac{1}{p}+\frac{3}{q}\leq 1,\\[2pt]
                    \lambda^{\alpha-\frac{1}{2}+\frac{1}{2p}-\frac{1}{2q}+\epsilon}h^{\frac{1}{p}-\frac{1}{2}}\|f\|_{L^p(\mr^3)},
                    & \frac{1}{p}+\frac{3}{q}>1,
                \end{cases}\\[2pt]
                \mathfrak B_{q,p}^{s,\alpha}(f)
                &\lesssim
                \begin{cases}
                    \lambda^{\frac{1}{p}-\frac{2}{q}+\epsilon}h^{\alpha-\frac{1}{2}}\|f\|_{L^p(\mr^3)},
                    & \quad \,\frac{1}{p}+\frac{3}{q}\leq 1,\\[2pt]
                    \lambda^{\frac{3}{2p}-\frac{1}{2q}-\frac{1}{2}+\epsilon}h^{\alpha-\frac{1}{2}}\|f\|_{L^p(\mr^3)},
                    & \quad \,\frac{1}{p}+\frac{3}{q}>1.
                \end{cases}
            \end{align*}
        \item[(c)] If $h\geq2^3$ and $\supp \widehat{f}\subset \mathbb{A}_{2^3}^{\circ}\times\mathbb{B}_h$, then
            \begin{align*}
                \mathfrak B_{q,p}^{t,\alpha}(f)
                &\lesssim h^{\frac{1}{p}-\frac{1}{2}}\|f\|_{L^p(\mr^3)},\\[2pt]
                \mathfrak B_{q,p}^{s,\alpha}(f)
                &\lesssim h^{\alpha-\frac{1}{2}}\|f\|_{L^p(\mr^3)}.
            \end{align*}
    \end{enumerate}
    \endgroup
\end{lem}
	\begin{proof}
		Part (a) follows from Lemma~\ref{lem-A1-A2}. We prove part (b). By \eqref{ali-expansion-sigmats}, it suffices to consider a high-horizontal-frequency term with amplitude $A=B_{1,1}^{\kappa,\kappa'}$; the signs $\kappa,\kappa'$ in the phases do not affect the argument. Set
		\[
		B(\xi,t,s):=(1+|\bar\xi|)^{1/2}(1+|\xi|)^{1/2}A(\xi,t,s).
		\]
		Then \eqref{asm-B-symbol} shows that $B$ satisfies the order-zero symbol estimates in \eqref{ali-condition-on-B}. On $\mathbb A_\lambda\times\mathbb B_h$, the original amplitude $A$ is bounded by $C(\lambda h)^{-1/2}$, while differentiation of the phases in $t$ and $s$ contributes at most $\lambda^\alpha$ and $h^\alpha$, respectively.
		We first estimate the restrictions of the normalized operator $\mathcal U_{t,s}^B$ to the parameter boundaries $s=2^{-3}$ and $t=1$.
		Since $2^3\leq\lambda\leq h$, we have $|\bar\xi|\sim\lambda$ and $|\xi|\sim h$ on the Fourier support. The Fourier-series argument below therefore gives local-constancy estimates at scales $\lambda^{-1}$ in $t$ and $h^{-1}$ in $s$. More precisely,
		\begin{align}\label{ali-local-cons-s}
			|\mathcal{U}_{t,2^{-3}}^Bf(x)|
			\lesssim \sum_{\ell\in \mathbb{Z}^3}(1+|\ell|)^{-100}
			|\mathcal{U}_{t,s}^1f(x+\tau_{\lambda,h}^{\ell})|,
		\end{align}
		for all $(t,s)\in \mathbb I\times[2^{-3},2^{-3}+h^{-1}]$, and
		\begin{align}\label{ali-local-cons-t}
			|\mathcal{U}_{1,s}^Bf(x)|
			\lesssim \sum_{\ell\in \mathbb{Z}^3}(1+|\ell|)^{-100}
			|\mathcal{U}_{t,s}^1f(x+\tau_{\lambda,h}^{\ell})|,
		\end{align}
		for all $(t,s)\in [1,1+\lambda^{-1}]\times \mathbb I_{-3}$,
		where \(\mathcal{U}_{t,s}^{1}\) denotes the operator
in \eqref{def-wjk} with \(B\equiv1\) and $\tau_{\lambda,h}^{\ell}:=(\lambda^{-1}\ell_1,\lambda^{-1}\ell_2,h^{-1}\ell_3)$ for $\ell=(\ell_1,\ell_2,\ell_3)\in \mathbb{Z}^3$.

	Indeed, let $\widetilde{\varphi}\in C_c^{\infty}((1/4,3))$ be equal to $1$ on $[2^{-1},2]$.
	For $s\in J_h:=[2^{-3},2^{-3}+h^{-1}]$, write
	$
	e^{i2^{-3}|\xi|}
	=e^{is|\xi|}e^{i(2^{-3}-s)|\xi|},
	$
	and then rescale the frequency variables. This gives
	\begin{align*}
		\mathcal{U}_{t,2^{-3}}^Bf(x)
		=\lambda^2h\int_{\mr^3}&e^{ix\cdot (\lambda\bar{\xi},h\xi_3)}e^{it\lambda|\bar{\xi}|}e^{is|(\lambda\bar{\xi},h\xi_3)|} \\
		&\quad\times B_{\lambda,h,4}(\xi,t,s)\widehat f(\lambda\bar{\xi},h\xi_3)\,d\xi,
	\end{align*}
	where
	\[
	B_{\lambda,h,4}(\xi,t,s):=\widetilde{\varphi}(|\bar{\xi}|)\widetilde{\varphi}(|\xi_3|)e^{i(2^{-3}-s)|(\lambda\bar{\xi},h\xi_3)|}B(\lambda\bar{\xi},h\xi_3,t,2^{-3}).
	\]
	Since $h\geq\lambda$ and $|s-2^{-3}|\leq h^{-1}$, for every multi-index $\gamma$,
	$
	\sup_{\xi}|\partial_\xi^{\gamma}B_{\lambda,h,4}(\xi,t,s)|\leq C_{\gamma}
	$
	holds uniformly in $(t,s)\in\mathbb{I}\times J_h$. Expanding
	$B_{\lambda,h,4}(\xi,t,s)=\sum_{\ell\in\mathbb Z^3}C_{\ell,4}(t,s)e^{i\ell\cdot\xi}$
	and arguing as in the proof of Lemma~\ref{lem-A4-A6} (d), we have
	$|C_{\ell,4}(t,s)|\lesssim(1+|\ell|)^{-100}$.
	Reversing the rescaling shows that each Fourier mode produces a
	spatial translation by $\tau_{\lambda,h}^{\ell}$.
	This proves \eqref{ali-local-cons-s}. The same argument, using $|\bar\xi|\sim\lambda$ and $|t-1|\leq\lambda^{-1}$, proves \eqref{ali-local-cons-t}.

	Consequently,
	\begin{align*}
		\|\mathcal{U}_{t,2^{-3}}^Bf\|_{L_x^q(\mr^3;L_t^p(\mathbb I))}&\lesssim h^{\frac{1}{p}}\sum_{\ell\in \mathbb{Z}^3}(1+|\ell|)^{-100}\|\mathcal{U}_{t,s}^1f\|_{L_x^q(\mr^3;L_{t,s}^p(\mathbb{I}\times J_h))}\\
		&\lesssim h^{\frac{1}{p}}\|\mathcal{U}_{t,s}^1f\|_{L_x^q(\mr^3;L_{t,s}^p(\mathbb{J}_0))},
	\end{align*}
	and
	\begin{align*}
		\|\mathcal{U}_{1,s}^Bf\|_{L_x^q(\mr^3;L_s^p(\mathbb I_{-3}))}\lesssim \lambda^{\frac{1}{p}}\|\mathcal{U}_{t,s}^1f\|_{L_x^q(\mr^3;L_{t,s}^p(\mathbb{J}_0))}.
	\end{align*}
		Combining the preceding mixed-norm estimates with the amplitude decay and the bounds for the parameter derivatives, as in the proof of Proposition~\ref{prop-LqLp-Lp-one} (b), gives part (b).
	For part (c), the horizontal frequency is bounded and the amplitudes $B_{0,1}^{\pm}$ and $B_{1,1}^{\kappa,\kappa'}$ are bounded by $Ch^{-1/2}$; the same argument gives the stated estimates. Derivatives falling on the amplitudes are controlled by \eqref{asm-B-symbol}.
\end{proof}

For every $1\leq q\leq\infty$ and every inner exponent
$1\leq u\leq\infty$, the same Fourier-series argument yields the
corresponding local-constancy estimates, with the convention
$N^{1/\infty}=1$,  and remains valid after one differentiation in either parameter. On $\mathbb A_\lambda\times\mathbb B_h$, the local-constancy scales in $t$ and $s$ are $\lambda^{-1}$ and $\max\{\lambda,h\}^{-1}$, respectively.
Thus, for $\alpha,\beta\in\{0,1\}$, the following estimates hold:
\begin{equation}\label{ali-boundary-transfer}
\begin{aligned}
 \|\partial_t^\alpha\mathcal A f_{\lambda,h}(\cdot,\cdot,2^{-3})\|_{L_x^qL_t^u}
 &\lesssim \max\{\lambda,h\}^{1/u}
 \|\partial_t^\alpha\mathcal A f_{\lambda,h}\|_{L_x^qL_{t,s}^u},\\
	 \|\partial_s^\beta\mathcal A f_{\lambda,h}(\cdot,1,\cdot)\|_{L_x^qL_s^u}
	 &\lesssim \lambda^{1/u}
	 \|\partial_s^\beta\mathcal A f_{\lambda,h}\|_{L_x^qL_{t,s}^u}.
\end{aligned}
\end{equation}
Here $L_x^qL_t^u$, $L_x^qL_s^u$, and $L_x^qL_{t,s}^u$ denote the mixed norms over $\mr^3\times\mathbb I$, $\mr^3\times\mathbb I_{-3}$, and $\mr^3\times\mathbb J_0$, respectively.

For $h\leq\lambda$, the multipliers produced by both parameter derivatives have size $\lambda$; for $h\geq\lambda$, those produced by $\partial_t$ and $\partial_s$ have sizes $\lambda$ and $h$, respectively.
Taking $u=p$ recovers the mixed-norm estimates used in the proof of Lemma~\ref{lem-boundary}.
Taking $u=r$ and combining the two estimates in \eqref{ali-boundary-transfer} with the mixed-norm estimates of Section~\ref{sec-Lp-Lq-bounds} gives the powers displayed in the corollaries below. These powers agree at $h=\lambda$ and $h=\lambda^2$ with those in the adjacent frequency regimes.

\begin{cor}\label{cor-boundary-above}
    Let $\lambda\geq2^3$, $\alpha\in \{0,1\}$, $(1/p,1/q)\in [P_5,P_6,P_4]$, and $p\leq r \leq q$.
    For any $\delta,\epsilon>0$, the following hold:
	    \begingroup
	    \addtolength{\leftmargini}{-2em}
	    \begin{enumerate}
	        \item[(a)] If $\supp \widehat{f}\subset \mathbb{A}_\lambda\times \mathbb{B}_{\lambda^{1/2+\delta}}^{\circ}$, then
	            \begin{align*}
						\mathfrak B_{q,r}^{t,\alpha}(f)+\mathfrak B_{q,r}^{s,\alpha}(f)
						&\lesssim \lambda^{\alpha+\frac{2}{p}-\frac{1}{q}-1+300\delta+\epsilon}
						\|f\|_{\Lp(\mr^3)}.
			\end{align*}

			\item[(b)] If $\lambda^{1/2+\delta}\leq h\leq\lambda$ and $\supp \widehat{f}\subset \mathbb{A}_\lambda\times \mathbb{B}_h$, then
			\begin{align*}
				\mathfrak B_{q,r}^{t,\alpha}(f)+\mathfrak B_{q,r}^{s,\alpha}(f)
				&\lesssim \lambda^{\alpha+\frac{2}{p}-\frac{3}{2q}+\frac{1}{2r}-1+\epsilon}h^{-\frac{1}{r}+\frac{1}{q}}\|f\|_{\Lp(\mr^3)}.
			\end{align*}

			\item[(c)] If $\lambda\leq h\leq\lambda^2$ and $\supp \widehat{f}\subset \mathbb{A}_\lambda\times \mathbb{B}_h$, then
			\begin{align*}
				\mathfrak B_{q,r}^{t,\alpha}(f)
				&\lesssim \lambda^{\alpha+\frac{1}{2p}-\frac{1}{2q}-\frac{1}{2}+\epsilon}h^{-\frac{1}{2}+\frac{3}{2p}-\frac{1}{2r}}\|f\|_{\Lp(\mr^3)},\\[2pt]
				\mathfrak B_{q,r}^{s,\alpha}(f)
				&\lesssim \lambda^{\frac{1}{2p}-\frac{1}{2q}-\frac{1}{2}+\frac{1}{r}+\epsilon}h^{\alpha-\frac{1}{2}+\frac{3}{2p}-\frac{3}{2r}}\|f\|_{\Lp(\mr^3)}.
			\end{align*}
	    \end{enumerate}
	    \endgroup
\end{cor}

The next corollary treats the middle interpolation region
$[P_5,Q_1,P_6]$, where the estimate in part (b) changes at
$r=r_{p,q}$.

\begin{cor}\label{cor-boundary-middle}
    Let $\lambda\geq2^3$, $\alpha\in\{0,1\}$, $(1/p,1/q)\in[P_5,Q_1,P_6]$, and $p\leq r\leq q$.
    Let $r_{p,q}$ be as in Corollary~\ref{cor-pqr-mid}. For any $\delta,\epsilon>0$, the following hold:
	    \begingroup
	    \addtolength{\leftmargini}{-2em}
        \begin{enumerate}
	            \item[(a)] If $\supp \widehat{f}\subset \mathbb{A}_\lambda\times \mathbb{B}_{\lambda^{1/2+\delta}}^{\circ}$, then
	            \begin{align*}
						 \mathfrak B_{q,r}^{t,\alpha}(f)+\mathfrak B_{q,r}^{s,\alpha}(f)
						 &\lesssim \lambda^{\alpha-\frac{1}{2}+\frac{3}{2p}-\frac{5}{2q}+300\delta+\epsilon}
						 \|f\|_{\Lp(\mr^3)}.
			\end{align*}

			\item[(b)] Suppose that $\lambda^{1/2+\delta}\leq h\leq\lambda$ and $\supp \widehat{f}\subset \mathbb{A}_\lambda\times \mathbb{B}_h$. Then
			\begin{align*}
				\mathfrak B_{q,r}^{t,\alpha}(f)+\mathfrak B_{q,r}^{s,\alpha}(f)
				&\lesssim
				\begin{cases}
					\lambda^{\alpha-\frac{1}{2}+\frac{1}{2r}+\frac{3}{2p}-\frac{3}{q}+\epsilon}
					h^{-\frac{1}{r}+\frac{1}{q}}\|f\|_{\Lp(\mr^3)},
					& p\leq r\leq r_{p,q},\\[2pt]
					\lambda^{\alpha+\frac{2}{p}-\frac{2}{q}-1+\frac{1}{r}+\epsilon}
					h^{1-\frac{1}{p}-\frac{1}{q}-\frac{2}{r}}\|f\|_{\Lp(\mr^3)},
					& r_{p,q}\leq r\leq q.
				\end{cases}
			\end{align*}
			\item[(c)] If $\lambda\leq h\leq\lambda^2$, $\supp \widehat{f}\subset \mathbb{A}_\lambda\times \mathbb{B}_h$, and $p\leq r\leq r_{p,q}$, then
			\begin{align*}
				\mathfrak B_{q,r}^{t,\alpha}(f)
				&\lesssim \lambda^{\alpha-\frac{2}{q}+\epsilon}h^{-\frac{1}{2r}-\frac{1}{2}+\frac{3}{2p}}\|f\|_{\Lp(\mr^3)},\\[2pt]
				\mathfrak B_{q,r}^{s,\alpha}(f)
				&\lesssim \lambda^{\frac{1}{r}-\frac{2}{q}+\epsilon}h^{\alpha-\frac{3}{2r}-\frac{1}{2}+\frac{3}{2p}}\|f\|_{\Lp(\mr^3)}.
			\end{align*}
        \end{enumerate}
	    \endgroup
\end{cor}

Finally, we consider the lower interpolation region
$[P_5,P_1,Q_1]$; here $r_q$ is the lower endpoint of the
$r$-range appearing in parts (b) and (c).

\begin{cor}\label{cor-boundary-below}
    Let $\lambda\geq2^3$, $\alpha\in\{0,1\}$, and $(1/p,1/q)\in[P_5,P_1,Q_1]$.
    Let $r_q$ be as in Corollary~\ref{cor-pqr-below}. For any $\delta,\epsilon>0$, the following hold:
	    \begingroup
	    \addtolength{\leftmargini}{-2em}
        \begin{enumerate}
	            \item[(a)] If $\supp \widehat{f}\subset \mathbb{A}_\lambda\times \mathbb{B}_{\lambda^{1/2+\delta}}^{\circ}$ and $p\leq r\leq q$, then
				\begin{align*}
						 \mathfrak B_{q,r}^{t,\alpha}(f)+\mathfrak B_{q,r}^{s,\alpha}(f)
						 &\lesssim \lambda^{\alpha-\frac{1}{2}+\frac{3}{2p}-\frac{5}{2q}+300\delta+\epsilon}
						 \|f\|_{\Lp(\mr^3)}.
				\end{align*}

			\item[(b)] Suppose that $\lambda^{1/2+\delta}\leq h\leq\lambda$ and $\supp \widehat{f}\subset \mathbb{A}_\lambda\times \mathbb{B}_h$. Then
				\begin{align*}
					\mathfrak B_{q,r}^{t,\alpha}(f)+\mathfrak B_{q,r}^{s,\alpha}(f)
					&\lesssim
					\begin{cases}
						\lambda^{\alpha+\frac{1}{p}-\frac{2}{q}-1+\frac{2}{r}+\epsilon}
						h^{1-\frac{1}{p}-\frac{1}{q}-\frac{2}{r}}\|f\|_{\Lp(\mr^3)},
						& r_q\leq r\leq p,\\[2pt]
						\lambda^{\alpha+\frac{2}{p}-\frac{2}{q}-1+\frac{1}{r}+\epsilon}
						h^{1-\frac{1}{p}-\frac{1}{q}-\frac{2}{r}}\|f\|_{\Lp(\mr^3)},
						& p\leq r\leq q.
					\end{cases}
				\end{align*}

			\item[(c)] If $\lambda\leq h\leq\lambda^2$, $\supp \widehat{f}\subset \mathbb{A}_\lambda\times \mathbb{B}_h$, and $r_q\leq r\leq p$, then
			\begin{align*}
				\mathfrak B_{q,r}^{t,\alpha}(f)
				&\lesssim \lambda^{\alpha+1-\frac{2}{r}-\frac{4}{q}+\epsilon}h^{-1+\frac{2}{r}+\frac{1}{q}}\|f\|_{\Lp(\mr^3)},\\[2pt]
				\mathfrak B_{q,r}^{s,\alpha}(f)
				&\lesssim \lambda^{1-\frac{1}{r}-\frac{4}{q}+\epsilon}h^{\alpha-1+\frac{1}{r}+\frac{1}{q}}\|f\|_{\Lp(\mr^3)}.
			\end{align*}
        \end{enumerate}
	    \endgroup
\end{cor}

\subsection{Boundary variation estimates and two-parameter bounds}

\begin{prop}\label{prop-boundary-terms}
    Let $\mathfrak{P}(r)$ be as in Theorem~\ref{thm-Lp-Lq}. Then the boundary variation operators $[\mathcal{A}f(\cdot,1,\cdot)]_{r,1}$ and $[\mathcal{A}f(\cdot,\cdot,2^{-3})]_{r,1}$ are bounded from $L^p(\mr^3)$ to $L^q(\mr^3)$ if $(1/p,1/q)\in\mathfrak{P}(r)$.
\end{prop}
\begin{proof} 
The pieces with bounded horizontal frequency are controlled by
parts (a) and (c) of Lemma~\ref{lem-boundary}. For $\lambda\geq2^3$,
Lemma~\ref{lem-embed-1-var} gives
   \begin{align*}
   \|[\mathcal A f_{\lambda,h}(\cdot,\cdot,2^{-3})]_{r,1}\|_{L^q}
   &\lesssim \sum_{\alpha=0}^1
   \lambda^{1/r-\alpha}\mathfrak B_{q,r}^{t,\alpha}(f_{\lambda,h}),\\
   \|[\mathcal A f_{\lambda,h}(\cdot,1,\cdot)]_{r,1}\|_{L^q}
   &\lesssim \sum_{\beta=0}^1
   \max\{\lambda,h\}^{1/r-\beta}\mathfrak B_{q,r}^{s,\beta}(f_{\lambda,h}).
   \end{align*}
   Lemma~\ref{lem-boundary} (b) handles the high-vertical range,
   while Corollaries~\ref{cor-boundary-above}--\ref{cor-boundary-below}
   handle the remaining ranges. By \eqref{ali-boundary-transfer}, the
   first sum is bounded by the terms in Lemma~\ref{lem-embedding} with
   $\beta=0$, and the second by those with $\alpha=0$, upon taking
   $N_1=\lambda$ and $N_2=\max\{\lambda,h\}$.

   To make this termwise comparison explicit, suppose that
   $(1/p,1/q)\in[P_5,P_3,P_4]$ and $p\leq r\leq q$, and consider
   the transition range $\lambda^{1/2+\delta}\leq h\leq\lambda$.
   Since $[P_5,P_3,P_4]\subset[P_5,P_6,P_4]$, Corollary
   \ref{cor-boundary-above} (b) gives, for
   $\alpha,\beta\in\{0,1\}$,
   \begin{align*}
   \lambda^{1/r-\alpha}\mathfrak B_{q,r}^{t,\alpha}(f_{\lambda,h})
   &\lesssim
   \lambda^{\frac{2}{p}-\frac{3}{2q}+\frac{3}{2r}-1+\epsilon}
   h^{-\frac1r+\frac1q}\|f_{\lambda,h}\|_{L^p},\\
   \lambda^{1/r-\beta}\mathfrak B_{q,r}^{s,\beta}(f_{\lambda,h})
   &\lesssim
   \lambda^{\frac{2}{p}-\frac{3}{2q}+\frac{3}{2r}-1+\epsilon}
   h^{-\frac1r+\frac1q}\|f_{\lambda,h}\|_{L^p}.
   \end{align*}
   The same powers are obtained by combining Corollary
   \ref{cor-pqr-top} (b) with the $(\alpha,0)$ and $(0,\beta)$ terms,
   respectively, in Lemma~\ref{lem-embedding}. Since
   $-1/r+1/q\leq0$, \eqref{est-dyadic-sum} gives
   \[
   \lambda^{-L_3(p,q,r)+\delta(\frac1q-\frac1r)+C\epsilon}.
   \]
   The term involving $\delta$ is nonpositive; when $r=q$, the
   resulting logarithm is absorbed into the $\epsilon$-loss. Thus
   this is bounded by $\lambda^{-L_3(p,q,r)+C\epsilon}$, as in the
   proof of Proposition~\ref{prop-top}.

   In the remaining frequency and exponent regions, the same
   termwise comparison, combined with the H\"older and monotonicity
   reductions used in Section~\ref{subsect-Lpq-variation}, shows that
   neither boundary variation produces a larger dyadic exponent.

   The controlling margins are inherited from Propositions
   \ref{prop-top}--\ref{prop-below}. Here $L_1,\ldots,L_5$ and
   $M_1,M_2,M_3$ are defined in \eqref{ali-affine-functions} and
   \eqref{ali-auxiliary-affine-functions}, respectively. On the top region they are $L_3$
   together with $M_1$ when $r\geq p$, and $L_5$ when $r<p$. On the
   middle region they are $L_2,M_2,M_3$ when $r\geq p$, and $L_4$
   when $r<p$. On the lower region they are $L_2,M_2,M_3$ for
   $r\geq p$, $L_4,L_1$ for $r_q\leq r\leq p$, and $L_4$ for
   $r<r_q$; the same monotonicity reductions cover $r>q$. Applying
   the dyadic summation argument from the proofs of those
   propositions, including \eqref{est-dyadic-sum}, bounds each
   frequency region by $\sum_{\lambda>2^3}\lambda^{-\vartheta}$ for
   some $\vartheta>0$, after first choosing $\delta$ and then
   $\epsilon$ sufficiently small.
   Combining these estimates with the fact that the three
   regions considered in Propositions~\ref{prop-top}--\ref{prop-below}
   cover $\mathfrak P(r)$ proves the proposition.
\end{proof}

Having established the boundary variation estimates, we now combine
them with the interior variation bounds from Section
\ref{subsect-Lpq-variation}.

\begin{proof}[Proof of the sufficient parts of Theorems~\ref{thm-two-para} and \ref{thm-Lp-Lq}]
    By \eqref{ali-def-two-parameter-variation}, the interior variation
    term is controlled by Propositions
    \ref{prop-top}--\ref{prop-below}, and the two boundary variation
    terms are controlled by Proposition~\ref{prop-boundary-terms}.
    Moreover,
    $|\mathcal Af(x,1,2^{-3})|\leq\mathcal M_cf(x)$, and $\mathcal M_c$
    satisfies the required $L^p$--$L^q$ estimate on $\mathcal Q$ by
    \cite[Theorem~1.2]{LL}. Since $\mathfrak P(r)\subset\mathcal Q$,
    this proves the sufficient part of Theorem~\ref{thm-Lp-Lq}.

    To obtain Theorem~\ref{thm-two-para}, set $q=p$. On the diagonal
    section of $\mathcal Q$, where $0\leq1/p<1/2$, the conditions
    $L_2>0$ and $L_3>0$ are redundant, and the remaining three
    inequalities reduce to
    \[
        \frac1r<\min\left\{\frac3p,\frac1p+\frac14,\frac12\right\}.
    \]
    In the $(1/p,1/r)$-plane, this is precisely the interior of
    $\mathfrak Q_1$ together with the open segment $(P_1,Q_1)$. At
    $P_1$, where $p=r=\infty$, the required bound follows from the
    $L^\infty$ boundedness of $\mathcal A$ and the definition of
    $V_\infty$. This proves the sufficient part of Theorem
    \ref{thm-two-para} on its stated region.
\end{proof}

\subsection{Variational estimates for the one-parameter operator \texorpdfstring{$\mathcal A_1$}{A1}}

Finally, we prove that $V_r(\mathcal A_1)$ is bounded on
$L^p(\mr^3)$. The proof relies on the following lemma.
\begin{lem}\label{lem-one-para-tori}
    Let $\alpha\in \{0,1\}$. For any $\epsilon>0$, the following hold:
	\begingroup
	\addtolength{\leftmargini}{-2em}
    \begin{enumerate}
         \item[(a)] If $2^3\leq\lambda\leq h\leq\lambda^2$ and $\supp \widehat{f}\subset \mathbb{A}_\lambda\times \mathbb{B}_h$, then
                 \begin{align*}
                   \|\partial_t^\alpha \mathcal{A}_1f\|_{\Lp(\mr^3\times \mathbb{I})}\lesssim
                    \begin{cases}
                      \lambda^{1-\frac{5}{p}+\epsilon}h^{\alpha-1+\frac{2}{p}}\|f\|_{\Lp(\mr^3)},& p\geq6, \\[2pt]
                      \lambda^{\frac{1}{2}-\frac{2}{p}+\epsilon}h^{\alpha+\frac{1}{2p}-\frac{3}{4}}\|f\|_{\Lp(\mr^3)},& 2\leq p\leq 6.
                    \end{cases}
                  \end{align*}

         \item[(b)] If $p\geq 4$, $\lambda\geq2^3$, and $\supp \widehat{f}\subset \mathbb{A}_\lambda\times \mathbb{B}_\lambda^{\circ}$, then
               \begin{align*}
                   \|\partial_t^\alpha \mathcal{A}_1f\|_{\Lp(\mr^3\times \mathbb{I})}\lesssim \lambda^{\alpha-\frac{3}{p}+\epsilon}\|f\|_{\Lp(\mr^3)}.
               \end{align*}

         \item[(c)] If $p\geq 4$, $\lambda\geq2^3$, $h\geq\lambda^2$, and $\supp \widehat{f}\subset \mathbb{A}_\lambda\times \mathbb{B}_h$, then
             \begin{align*}
                    \|\partial_t^{\alpha}\mathcal{A}_1f\|_{\Lp(\mr^{3}\times \mathbb I)}\lesssim
                    \lambda^{-\frac{2}{p}+\epsilon}h^{\alpha-\frac{1}{2}}\|f\|_{\Lp(\mr^3)}.
             \end{align*}
     \end{enumerate}
	\endgroup
\end{lem}
\begin{proof}
    For $\alpha=0$, parts (a)--(c) for $p\geq6$, $p\geq4$, and
    $p\geq4$, respectively, follow from
    \cite[Proposition~4.2 (a)--(c)]{LL} in the present dyadic scale
    notation. In part (a), apply the cited estimate with loss
    $\epsilon/2$. Since $\lambda\leq h\leq\lambda^2$, its factor
    $h^{\epsilon/2}$ is bounded by $\lambda^\epsilon$.

    It remains to prove part (a) for $2\leq p\leq6$. On
    $\mathbb A_\lambda\times\mathbb B_h$, with $\lambda\leq h$, the
    expansion \eqref{ali-expansion-sigmats} and the symbol bounds
    \eqref{asm-B-symbol} give a multiplier of size
    $O((\lambda h)^{-1/2})$, uniformly for $t\in\mathbb I$.
    Plancherel's theorem in the spatial variables therefore gives
    \[
        \|\mathcal A_1f\|_{L^2(\mr^3\times\mathbb I)}
        \lesssim(\lambda h)^{-1/2}\|f\|_{L^2(\mr^3)}.
    \]
    Interpolation with the $p=6$ estimate gives the second line in
    part (a).

    For $\alpha=1$, the identity
    $\mathcal A_1f(x,t)=\mathcal Af(x,t,c_0t)$ shows that
    differentiation acts as $\partial_t+c_0\partial_s$. The phase
    multipliers in \eqref{ali-expansion-sigmats} have size $O(h)$ on
    $\mathbb A_\lambda\times\mathbb B_h$ when $h\geq\lambda$, and
    size $O(\lambda)$ on
    $\mathbb A_\lambda\times\mathbb B_\lambda^\circ$. Parameter
    derivatives falling on the symbols are controlled by
    \eqref{asm-B-symbol} and produce no larger contribution. Applying
    the preceding $\alpha=0$ estimates with these respective bounds
    proves the assertions for $\alpha=1$.
\end{proof}

\begin{proof}[Proof of the sufficient part of Theorem~\ref{thm-one-para}]
    By Theorem~\ref{thm-two-para} and
    \eqref{ali-diagonal-var-comparison}, it remains to treat the part
    of the one-parameter region not already covered by the
    two-parameter result.
    We prove the following estimate
    \[
        \|V_r(\mathcal A_1f)\|_{\Lp(\mr^3)}
        \lesssim\|f\|_{\Lp(\mr^3)}
    \]
    for exponents in the larger range $p\geq4$ and
    $2/p\leq1/r<\min\{1/2,3/p\}$.
    Decompose
    \[
        f=f_{\leq2^3,\leq2^3}
        +\sum_{h>2^3}f_{\leq2^3,h}
        +\sum_{\lambda>2^3}f_{\lambda,\leq\lambda}
        +\sum_{\lambda>2^3}\sum_{\lambda<h<\lambda^2}f_{\lambda,h}
        +\sum_{\lambda>2^3}\sum_{h\geq\lambda^2}f_{\lambda,h}.
    \]
    The fully low-frequency term is controlled by
    Lemmas~\ref{lem-embed-1-var} and \ref{lem-A1-A2} (a) at a fixed
    scale. For $f_{\leq2^3,h}$, apply Lemma~\ref{lem-embed-1-var}
    with $N=h$ and use Lemma~\ref{lem-A1-A2} (b). For the remaining
    high-horizontal-frequency terms, apply that lemma with $N=\lambda$
    to $f_{\lambda,\leq\lambda}$ and with $N=h$ to the last two sums,
    and then use Lemma~\ref{lem-one-para-tori}. H\"older's inequality
    and summation first in $h$ give the required geometric decay.

    For $p\geq6$, the resulting powers of $\lambda$ in the intermediate
    and vertical ranges are, respectively,
    $1/r-3/p+\epsilon$ and
    $-2/p+2/r-1+\epsilon$. The bounded-horizontal and low-vertical
    sums converge because $1/r<1/2$ and $1/r<3/p$, respectively.
    Hence, after choosing $\epsilon>0$ sufficiently small, there is
    $\vartheta>0$ such that
    \[
        \|V_r(\mathcal A_1f)\|_{\Lp(\mr^3)}
        \lesssim
        \Big(1+\sum_{\lambda>2^3}\lambda^{-\vartheta}\Big)
        \|f\|_{\Lp(\mr^3)}
        \lesssim\|f\|_{\Lp(\mr^3)}.
    \]

    For $4\leq p\leq6$, the second line of Lemma
    \ref{lem-one-para-tori} (a) gives the intermediate exponent
    $1/r-3/(2p)-1/4+\epsilon$, while the vertical exponent remains
    $-2/p+2/r-1+\epsilon$. Both are negative because
    $2/p\leq1/r<1/2\leq3/(2p)+1/4$. The bounded-horizontal and
    low-vertical sums converge since $1/r<1/2\leq3/p$. The preceding
    estimate therefore also holds in this range, with a possibly
    different $\vartheta>0$.
\end{proof}

This completes the proofs of the sufficient parts of all three main theorems.
\section{Necessary conditions}\label{sec-nec-cond}
We now prove the necessary parts of Theorems~\ref{thm-one-para}, \ref{thm-two-para}, and \ref{thm-Lp-Lq}.
Here $j$ and $k$ are integer parameters used in the constructions below, rather than dyadic frequency indices. Whenever a construction uses $k$, we write $h:=2^k$, so that $h^{-1}=2^{-k}$.
Throughout this section, $\varepsilon_0>0$ denotes a sufficiently
small constant independent of $k$, and whenever $k\gg1$, we write
$N_k:=\lfloor\varepsilon_0 2^k\rfloor$. We begin with the following proposition.

\subsection{Necessary conditions for
\texorpdfstring{$V_r(\mathcal A_1)$ and $V_r(\mathcal A)$}
{Vr(A1) and Vr(A)}}

We recall the affine functions $L_1,\ldots,L_5$ defined in
\eqref{ali-affine-functions}. 

\begin{prop}\label{prop-nece-r>2}
     Suppose that either $V_{r}(\mathcal{A}_1)$ or $V_r(\mathcal{A})$ is bounded from $L^{p}(\mr^3)$ to $L^q(\mr^3)$.
     Then the following hold:

     \textnormal{(a)} $r\geq 2$.

     \textnormal{(b)} $L_1(p,q,r)\geq 0$.

     \textnormal{(c)} If the boundedness assumption holds for $V_r(\mathcal A)$, then
     $L_4(p,q,r)\geq 0$.
\end{prop}

When $r=\infty$, the necessary
conditions  are immediate for $V_\infty(\mathcal A_1)$.
For $V_\infty(\mathcal A)$, they follow from the corresponding
maximal-function obstruction, since
$L_i(p,q,\infty)\geq0$, $1\leq i\leq5$, throughout
$\overline{\mathcal Q}$. Thus, in the proofs below, we may assume
that $r<\infty$.

\begin{proof}
    (a) For every fixed $\eta>0$, \eqref{ali-diagonal-var-comparison}
    gives
    \[
        V_{r+\eta}(\mathcal A_1f)\leq C_{r,\eta}V_r(\mathcal Af).
    \]
    It therefore suffices to prove that boundedness of
    $V_\rho(\mathcal A_1)$ from $L^p$ to $L^q$ forces $\rho\geq2$.
    Under the direct boundedness assumption for $V_r(\mathcal A_1)$,
    take $\rho=r$. Under the boundedness assumption for
    $V_r(\mathcal A)$, take $\rho=r+\eta$ and then let
    $\eta\downarrow0$; no uniformity of $C_{r,\eta}$ is needed.

    Choose a nonnegative function $\psi\in C_c^\infty(\mr^{2})$ such that $\psi(\bar y)=1$ whenever $\frac12\leq |\bar y|\leq 3$.
    Let $\vartheta\in C_c^\infty(\mr)$ be nonnegative, with $\vartheta=1$ on $[-\frac13,-\frac1{10}]$ and $\supp \vartheta\subset [-\frac12,-\frac1{16}]$.
	For $k\gg 1$, define
	\[f_{k}(\bar y,y_{3}):=2^{\frac k2}e^{-i2^{k-1}y_{3}^{2}}\vartheta(y_{3})\psi(\bar y).\]
	Then $\|f_{k}\|_{L^{p}(\mr^{3})}\lesssim 2^{k/2}$.
	Next we estimate $\mathcal{A}_1f_{k}(x,t)$ for $x$ near the origin. Let
	\[\widetilde{E}:=\{x=(\bar x,x_{3})\in \mr^{3}: |\bar x|\leq\varepsilon_0,\ 0\leq x_{3}\leq\varepsilon_0\}.\]
	For $x\in \widetilde{E}$ and $t\in [1,2]$, we have
	\[
		\mathcal{A}_1f_{k}(x,t)=2\pi\,2^{\frac k2}\int_{0}^{2\pi}e^{-i2^{k-1}(x_{3}-c_{0}t\sin\theta)^{2}}\vartheta(x_{3}-c_{0}t\sin\theta)\,d\theta.
	\]
	This is a one-dimensional oscillatory integral. On the support of the
	amplitude, the phase $\theta\mapsto(x_3-c_0t\sin\theta)^2$, with
	$(x_3,t)\in[0,\varepsilon_0]\times[1,2]$, has the unique critical point
	$\theta=\pi/2$, which is nondegenerate. More precisely, if
	$\phi_{t,x_3}(\theta)=(x_3-c_0t\sin\theta)^2$, then
	\[\phi_{t,x_3}''(\pi/2)=2c_0t(x_3-c_0t),\] whose absolute value is bounded
	below uniformly on the indicated $(t,x_3)$-rectangle. The other critical point, at $3\pi/2$, lies outside the support of the amplitude. Hence, the stationary phase method (see, for example,
	\cite[Theorem~7.7.5]{Hor}) yields, uniformly in $(t,x_3)$,
	\[\mathcal{A}_1f_{k}(x,t)=b(t,x_3)\,e^{-i2^{k-1}(c_{0}t-x_{3})^{2}}+O(2^{-k}),\]
	where $b$ is smooth, $|b(t,x_3)|\sim1$, and
	$|\partial_tb(t,x_3)|\lesssim1$ uniformly for
	$(x_3,t)\in[0,\varepsilon_0]\times[1,2]$.

    For $1\leq n\leq N_k$ and $0\leq x_3\leq\varepsilon_0$, define
    \[
        t_n(x_3):=2^3\big(\sqrt{2^{-6}+n\pi2^{1-k}}+x_3\big).
    \]
   Then, for $1\leq n\leq N_k-1$,
\[
e^{-i2^{k-1}(c_0t_{n+1}(x_3)-x_3)^2}
=-e^{-i2^{k-1}(c_0t_n(x_3)-x_3)^2}.
\]
For the same range of $n$, we have
$t_n(x_3),t_{n+1}(x_3)\in[1,2]$ and
$|t_{n+1}(x_3)-t_n(x_3)|\sim2^{-k}$. Consequently, the mean value
	    theorem and the bounds on $b$ give
	    $b(t_{n+1}(x_3),x_3)$ $=b(t_n(x_3),x_3)+O(2^{-k})$.
    Combining these estimates with the stationary-phase expansion, we obtain
    \[
        \big|
        \mathcal A_1f_k(x,t_{n+1}(x_3))
        -\mathcal A_1f_k(x,t_n(x_3))
        \big|
        =2|b(t_n(x_3),x_3)|+O(2^{-k})
        \gtrsim1
    \]
    uniformly for $x\in\widetilde E$ and $1\leq n\leq N_k-1$, provided that $k$ is sufficiently large. Therefore,
    \[
        V_r(\mathcal A_1f_k)(x)
        \geq
        \bigg(
        \sum_{n=1}^{N_k-1}
        \big|
        \mathcal A_1f_k(x,t_{n+1}(x_3))
        -\mathcal A_1f_k(x,t_n(x_3))
        \big|^r
        \bigg)^{1/r}
        \gtrsim2^{k/r}.
    \]
    Since $\widetilde E$ is fixed and $\|f_k\|_{\Lp(\mr^3)}\lesssim2^{k/2}$, it follows that
    \[
        \frac{
        \|V_r(\mathcal A_1f_k)\|_{L^q(\mr^3)}
        }{
        \|f_k\|_{\Lp(\mr^3)}
        }
        \gtrsim2^{k(1/r-1/2)}.
    \]
    Hence boundedness forces $1/r-1/2\leq0$, and therefore $r\geq2$.

	    (b) First suppose that $V_r(\mathcal A_1)$ is bounded from $L^p$ to $L^q$.
	    Let $k\gg 1$, $t_n:=1+n2^{-2-k}$ and $s_n:=2^{-3}t_n$ for $1\leq n\leq 2^k$. Define
	    \[
	        E_{2m}:=\{\Phi_{t_{2m},u}(\theta,\phi):\theta\in[0,\pi/8],\ \phi\in[0,2\pi),\ u\in[s_{2m}-\varepsilon_0 2^{-k},s_{2m}+\varepsilon_0 2^{-k}]\},
	    \]
	    and set $f_k:=\sum_{m=1}^{2^{k-1}}\chi_{E_{2m}}$.

	    On the indicated parameter region, the Jacobian of
	    $(\theta,\phi,u)\mapsto\Phi_{t_{2m},u}(\theta,\phi)$ is
	    $u(t_{2m}+u\cos\theta)\sim1$, and hence $|E_{2m}|\sim2^{-k}$.
		    The central symmetry of the torus and the uniform local
		    parametrization imply that, for
		    $x\in\widetilde E:=\{x\in\mr^3:|x|\leq\varepsilon_0^2 2^{-k}\}$,
    \[
        \mathcal Af_k(x,t_{2m},s_{2m})
        \geq\int_{[0,2\pi)^2}\chi_{E_{2m}}\big(x-\Phi_{t_{2m},s_{2m}}(\theta,\phi)\big)\,d\theta\,d\phi
        \geq C.
    \]

    We next verify the separation. For $y\in E_{2m}$, put $R=|\bar y|$. Then
    \[
        (R-t_{2m})^2+y_3^2=u^2,
        \qquad
        \frac{R-t_{2m}}u\geq\cos(\pi/8).
    \]
    Suppose that $E_{2m}$ meets
    $x-\mathbb T_{t_n,s_n}$ for some $n\neq2m$. For
    $y\in E_{2m}\cap(x-\mathbb T_{t_n,s_n})$, put
    $R=|\bar y|$ and $d=t_n-t_{2m}$. Subtracting the corresponding
    circle equations gives
    \[
        2d(R-t_{2m})
        =(1-c_0^2)d^2-2c_0^2t_{2m}d
        +u^2-c_0^2t_{2m}^2+O(\varepsilon_0^2 2^{-k}),
    \]
    where the error is due to the translation by $x$. Since
    $|d|\geq2^{-k-2}$ and
    $|u-c_0t_{2m}|\leq\varepsilon_0 2^{-k}$, it follows that
    \[
        \frac{R-t_{2m}}u
        =\frac{(1-c_0^2)d/2-c_0^2t_{2m}}{c_0t_{2m}}
        +O(\varepsilon_0)
        <\cos(\pi/8).
    \]
    Here $|d|\leq1/4$ and $1\leq t_{2m}\leq5/4$; when $d<0$,
    the main term is even
    smaller. This contradicts the preceding lower bound.

    The same calculation, now allowing the minor-radius parameter to vary
within $\varepsilon_0 2^{-k}$ of $s_n$, shows that the sets $E_{2m}$
are pairwise disjoint. Therefore only the matching even index contributes, and no set contributes at an odd index. Since $s_n=c_0t_n$,
    \[
        \mathcal A_1f_k(x,t_{2m})\geq C,
        \qquad
        \mathcal A_1f_k(x,t_{2m-1})=0.
    \]
	    Moreover, $\|f_k\|_{\Lp(\mr^3)}\lesssim1$.  	    Thus, for $x\in\widetilde E$,
	    \[
	        V_r(\mathcal A_1f_k)(x)\geq\left(\sum_{n=1}^{2^k-1}|\mathcal A_1f_k(x,t_{n+1})-\mathcal A_1f_k(x,t_n)|^r\right)^{1/r}\geq C2^{k/r},
	    \]
	    which gives
	    \[
	        \|V_r(\mathcal A_1f_k)\|_{L^q(\mr^3)}\gtrsim2^{k/r}|\widetilde E|^{1/q}=C2^{k(1/r-3/q)}.
	    \]
	    Consequently, boundedness forces $1/r\leq3/q$. If instead
	    $V_r(\mathcal A)$ is bounded, then
	    \eqref{ali-diagonal-var-comparison} gives the boundedness of
	    $V_{r+\eta}(\mathcal A_1)$ for every fixed $\eta>0$. Applying the
	    preceding conclusion with $r+\eta$ in place of $r$ and letting
	    $\eta\downarrow0$ gives $1/r\leq3/q$. This proves $L_1(p,q,r)=\frac{3}{q}-\frac{1}{r}\geq 0$.

    (c) Let $k\gg1$. Define
    \[
        f_k:=\sum_{j=0}^{2^{k-5}}(-1)^j\chi_{E_j},
    \]
    where
    \[
        E_j:=\{(\bar x,x_3)\in\mr^3:
        \big||\bar x|-(1-2^{-3}-2^{-5}+j2^{-k})\big|
        \leq2^{-k-5},\
        |x_3|\leq2^{-2-k/2}\}.
    \]
	    For $0\leq m,n\leq2^{k-5}$, set
	    \[
	        t_m:=1+m2^{-k}, \quad
	        s_n:=2^{-3}+n2^{-k}, \quad
	        j_{m,n}:=2^{k-5}+m-n.
		    \]

    Let
    \[
        \widetilde E:=
        \{(\bar x,x_3)\in\mr^3:
        |\bar x|\leq\varepsilon_0^2 2^{-k},\
        |x_3|\leq\varepsilon_0 2^{-k/2}\}.
    \]
    We verify the relevant support properties. Recall that $h^{-1}=2^{-k}$.
	    If $x-\Phi_{t_m,s_n}(\theta,\phi)$ belongs to one of the shells,
	    then its vertical support gives
	    $|s_n\sin\theta|\leq(2^{-2}+\varepsilon_0)h^{-1/2}$, while its
	    radial support forces $\cos\theta\leq-1/2$. Using
	    $s_n(1+\cos\theta)=s_n\sin^2\theta/(1-\cos\theta)$, we obtain
    \[
        \big|
        |\overline{x-\Phi_{t_m,s_n}(\theta,\phi)}|
        -(t_m-s_n)
        \big|
	        <\frac{3}{8h},
    \]
	where the bar denotes the horizontal component.
    For $0\leq m\leq n\leq2^{k-5}$, one has
    $0\leq j_{m,n}\leq2^{k-5}$. Since $2^{-5}=2^{k-5}h^{-1}$, the
    definitions of $t_m$, $s_n$, and $j_{m,n}$ give
    $t_m-s_n=1-2^{-3}-2^{-5}+j_{m,n}h^{-1}$. Thus the radial center of
    $E_{j_{m,n}}$ is exactly $t_m-s_n$. The shell centers are separated by
    $h^{-1}$ and have radial half-width $1/(32h)$. Hence, only
    $E_{j_{m,n}}$ can contribute.
	   Conversely, if $|\theta-\pi|\leq2^{-4}h^{-1/2}$ and
$0\leq\phi<2\pi$, then
$x-\Phi_{t_m,s_n}(\theta,\phi)$ lies in the matching shell.
This parameter region has measure comparable to $h^{-1/2}$.  Therefore,
	    for $x\in\widetilde E$ and $0\leq m\leq n\leq2^{k-5}$,
	    \[
	        (-1)^{j_{m,n}}\mathcal Af_k(x,t_m,s_n)
	        \gtrsim h^{-1/2}=2^{-k/2}.
	    \]

    Replacing $(m,n)$ by $(m+1,n)$ or $(m,n+1)$ changes
    $j_{m,n}$ by $1$ or $-1$, whereas changing both indices leaves
    its parity unchanged. Thus, once the signs in the rectangular difference are taken into
    account, all four terms have the same sign. Hence, for $0\leq m<n<2^{k-5}$,
    \[
    \begin{aligned}
        \big|
        &\mathcal Af_k(x,t_m,s_n)
        -\mathcal Af_k(x,t_{m+1},s_n) \\
        &-\mathcal Af_k(x,t_m,s_{n+1})
        +\mathcal Af_k(x,t_{m+1},s_{n+1})
        \big|
        \gtrsim2^{-k/2}.
    \end{aligned}
    \]
    Including the points $t_m$ and $s_n$ in partitions of
    $\mathbb I$ and $\mathbb I_{-3}$, respectively, and summing over the
    $\approx 2^{2k}$ pairs
    $0\leq m<n<2^{k-5}$, we obtain
    \[
        \widetilde{V}_r(\mathcal Af_k)(x)
        \gtrsim2^{k(2/r-1/2)}.
    \]
    Since $|\widetilde E|\sim2^{-5k/2}$,
    \[
        \|\widetilde{V}_r(\mathcal Af_k)\|_{L^q(\mr^3)}
        \gtrsim2^{k(2/r-1/2-5/(2q))}.
    \]
    Moreover, the shells $E_j$ are pairwise disjoint and
    \[
        \|f_k\|_{\Lp(\mr^3)}
        =\Big(\sum_{j=0}^{2^{k-5}}|E_j|\Big)^{1/p}
        \lesssim2^{-k/(2p)}.
    \]
    Boundedness therefore yields $L_4(p,q,r)=-\frac{1}{2p}+\frac{5}{2q}-\frac{2}{r}+\frac{1}{2}\geq0$.
\end{proof}

\subsection{Additional necessary conditions for
\texorpdfstring{$V_r(\mathcal A)$}{Vr(A)}}

Proposition~\ref{prop-nece-r>2} gives $r\geq2$ and
$L_1(p,q,r)\geq0$ for both $V_r(\mathcal A_1)$ and
$V_r(\mathcal A)$, as well as $L_4(p,q,r)\geq0$ for
$V_r(\mathcal A)$.
Three further constructions establish the remaining necessary conditions.

\begin{prop}\label{prop-nec-Lp-Lq}
	Suppose that $V_r(\mathcal{A})$ is bounded from $L^p(\mr^3)$ to
	$L^q(\mr^3)$. Then the following hold:

	\textnormal{(a)} $L_2(p,q,r)\geq 0$.

	\textnormal{(b)} $L_5(p,q,r)\geq 0$.

	\textnormal{(c)} $L_3(p,q,r)\geq 0$.
\end{prop}
	\begin{proof}
	        (a) Let $k\gg1$ and define $f_k:=\chi_{E_{1,k}}$, where
        \[
            E_{1,k}:=\{(\bar{x},x_3)\in \mr^3:
            \big||\bar{x}|-(1-2^{-3}-2^{-5}+2^{-k})\big|
            \leq 2^{-k-5}, |x_3|\leq 2^{-2-k/2}\}.
        \]
        For $0\leq m\leq 2^{k-5}$, set
	    \[
	        t_m:=1+m2^{-k},\qquad
	        s_m:=2^{-3}+2^{-5}+(m-1)2^{-k}.
	    \]
	    Then $t_m-s_m=1-2^{-3}-2^{-5}+2^{-k}$.
        If
        \[
            x\in \widetilde{E}:=
            \{(\bar{x},x_3)\in \mr^3:
            |\bar{x}|\leq\varepsilon_0^2 2^{-k},
            |x_3|\leq\varepsilon_0 2^{-k/2}\},
        \]
	    then
	    \[
	    \begin{aligned}
	        \mathcal{A}f_k(x,t_m,s_m)
	        &=\int_{[0,2\pi)^2}
	        \chi_{E_{1,k}}(x-\Phi_{t_m,s_m}(\theta,\phi))
	        \,d\theta\,d\phi \\
	        &\geq C2^{-k/2}.
	    \end{aligned}
	    \]
	    On the other hand, $\mathcal{A}f_k(x,t_m,s_n)=0$ whenever
	    $m\neq n$.

        Indeed, the support calculation in the proof of
        Proposition~\ref{prop-nece-r>2} (c) applies with $h^{-1}=2^{-k}$.
		        Here $t_m-s_n=(t_m-s_m)+(m-n)h^{-1}$. Thus, when
		        $m\neq n$, the value $t_m-s_n$ differs from the radial center of
		        $E_{1,k}$ by at least $h^{-1}$, and the support calculation excludes
		        every off-diagonal pair. When $m=n$, the region
		        $|\theta-\pi|\lesssim h^{-1/2}$ has measure comparable to
		        $h^{-1/2}$.
	        The same
	        calculation also gives $|E_{1,k}|\sim h^{-3/2}$ and
	        $|\widetilde E|\sim h^{-5/2}$.

        Thus, for $0\leq m<2^{k-5}$ and $x\in \widetilde{E}$, we obtain
	        \[
	        \begin{aligned}
	            \big|
	            &\mathcal{A}f_{k}(x,t_m,s_m)
	            -\mathcal{A}f_{k}(x,t_{m+1},s_m) \\
            &\quad -\mathcal{A}f_{k}(x,t_m,s_{m+1})
	            +\mathcal{A}f_{k}(x,t_{m+1},s_{m+1})
	            \big|
	            \geq C2^{-k/2}.
	        \end{aligned}
	        \]
	    Taking partitions containing the points $t_m$ and $s_m$, we obtain
	    \[
	    \begin{aligned}
	        \widetilde{V}_r(\mathcal{A}f_k)(x)
	        &\geq
	        \Big(
	        \sum_{m=0}^{2^{k-5}-1}
	        |\mathcal{A}f_k(x,[t_m,t_{m+1}]\times [s_m,s_{m+1}])|^r
	        \Big)^{1/r}\\
	        &\gtrsim2^{k(1/r-1/2)}.
	    \end{aligned}
	    \]
	    Hence
	    \[
	        \|\widetilde{V}_r(\mathcal{A}f_k)\|_{L^q(\mr^3)}
	        \gtrsim2^{k(1/r-1/2-5/(2q))}.
	    \]
	    Moreover,
	    \[
	        \|f_k\|_{\Lp(\mr^3)}=|E_{1,k}|^{1/p}
	        \lesssim2^{-3k/(2p)}.
	    \]
	    Boundedness therefore yields $L_2(p,q,r)= -\frac{3}{2p}+\frac{5}{2q}-\frac{1}{r}+\frac{1}{2}\geq0$, which proves (a).

			For (b), let $k\gg1$ and define
			$f_k:=\sum_{j=1}^{2^{k-4}}(-1)^j\chi_{E_j}$, where
			\[
			\begin{aligned}
				E_j := \Big\{ x \in \mr^3 :
				&\, |x_1 - \bigl(3 - (1 + 2^{-4} + 2^{-3}) + j2^{-k}\bigr)| \le \varepsilon_0 2^{-k}, \\[4pt]
				&\, |x_2| \le \varepsilon_0 2^{-k/2}, \ |x_3| \le \varepsilon_0 2^{-k/2} \Big\}.
			\end{aligned}
			\]
			Let
			\[
				E:=\{x\in \mr^3: 0\leq x_1-3\leq\varepsilon_0,|x_2|\leq\varepsilon_0 2^{-k/2},|x_3|\leq\varepsilon_0 2^{-k/2}\}.
			\]
			For $x\in E$ and $0\leq m,n\leq N_k$, set
	        \[
	            t_m(x):=x_1-3+1+m2^{-k},\qquad
	            s_n:=2^{-3}+n2^{-k}.
	        \]
	        Set $j_{m,n}:=2^{k-4}-m-n$. Our choice of
	        $\varepsilon_0$ ensures that $t_m(x)\in\mathbb I$,
	        $s_n\in\mathbb I_{-3}$, and
	        $1\leq j_{m,n}\leq2^{k-4}$ for all the indicated indices. Since
	        $2^{-4}=2^{k-4}2^{-k}$, we also have
	        \[
	            x_1-t_m(x)-s_n
	            =3-(1+2^{-4}+2^{-3})+j_{m,n}2^{-k}.
	        \]
			Thus, for $x\in E$, we have
			\[
					\mathcal{A}f_k(x,t_m(x),s_n)=c_{m,n}(x)(-1)^{j_{m,n}}2^{-k},
	            \qquad c_{m,n}(x)\sim1,
			\]
				where the comparison constants in $c_{m,n}(x)\sim1$ are uniform in $x,m,n$.

	            To verify this, recall that $h^{-1}=2^{-k}$.
	            The transverse support conditions imply
	            $
	            \operatorname{dist}(\theta,\pi\mathbb Z)
	            +\operatorname{dist}(\phi,\pi\mathbb Z)
	            \lesssim\varepsilon_0 h^{-1/2}.
	            $
	            The longitudinal support condition excludes every branch on
	            which either variable is close to an odd multiple of $\pi$.
	            By periodicity, we may therefore choose representatives satisfying
	            $|\theta|+|\phi|\lesssim\varepsilon_0 h^{-1/2}$. On this region,
            \[
                0\leq
                (t_m+s_n)-(t_m+s_n\cos\theta)\cos\phi
                \lesssim\theta^2+\phi^2
                \lesssim\varepsilon_0^2 h^{-1}
                \ll\varepsilon_0 h^{-1}.
            \]
	            The preceding identity and the above bound on the
	            longitudinal error show that only $E_{j_{m,n}}$ can contribute,
	            since the box centers are separated by $h^{-1}$.
	            Conversely,  for every $(\theta,\phi)$ in the rectangle
$|\theta|\leq ch^{-1/2}$, $|\phi|\leq ch^{-1/2}$ with a
	            sufficiently small $c>0$, 
the point $x-\Phi_{t_m(x),s_n}(\theta,\phi)$ lies in the matching
box; this rectangle has measure comparable to $h^{-1}$. 
	            This proves the asserted uniform bounds for $c_{m,n}(x)$.

            Increasing either $m$ or $n$ by one decreases $j_{m,n}$ by $1$,
            whereas increasing both decreases it by $2$. Hence, once the signs
            in the rectangular difference are taken into account, all four
            contributions to a given increment have the same sign. The absolute
            value of each rectangular increment is therefore bounded below by
            $ch^{-1}$.
	            Moreover, the boxes are pairwise
	            disjoint, $|E_j|\sim h^{-2}$, and $|E|\sim h^{-1}$.
	            For each fixed $x\in E$, take partitions containing the
	            points $t_m(x)$ and $s_n$. Then
			\[
			\begin{aligned}
				V_r(\mathcal{A}f_k)(x)
				&\geq C\left(\sum_{m=1}^{N_k}\sum_{n=1}^{N_k}
				|\mathcal{A}f_k(x,[t_{m-1}(x),t_m(x)]
				\times[s_{n-1},s_n])|^r\right)^{1/r}\\
				&\gtrsim2^{k(2/r-1)}.
			\end{aligned}
			\]
			Consequently,
			\[
				\|V_r(\mathcal{A}f_k)\|_{L^q(\mr^3)}
				\gtrsim2^{k(2/r-1)}|E|^{1/q}
				\gtrsim2^{k(2/r-1-1/q)}.
			\]
			Moreover,
			\[
				\|f_k\|_{\Lp(\mr^3)}
				\lesssim\Big(\sum_{j=1}^{2^{k-4}}|E_j|\Big)^{1/p}
				\lesssim2^{-k/p}.
			\]
			Boundedness therefore yields $L_5(p,q,r)=-\frac{1}{p}+\frac{1}{q}-\frac{2}{r}+1\geq0$, which proves (b).

			For (c), let $k\gg1$ and define $f_k:=\chi_{E_{1,k}}$, where
			\[
			\begin{aligned}
				E_{1,k} := \Big\{ x \in \mr^3 :
				&\, |x_1 - \bigl(3 - (1 + 2^{-4} + 2^{-3}) + 2^{-k}\bigr)| \le \varepsilon_0 2^{-k}, \\[4pt]
				&\, |x_2| \le \varepsilon_0 2^{-k/2}, \ |x_3| \le \varepsilon_0 2^{-k/2} \Big\}.
			\end{aligned}
			\]
			Let $E$ be as in part (b). For $x\in E$ and
			$0\leq m,n\leq N_k$, set
	        \[
	            t_m(x):=x_1-3+1+m2^{-k},
            \qquad
            s_n:=2^{-3}+2^{-4}-(n+1)2^{-k}.
        \]
	        Our choice of $\varepsilon_0$ ensures that
	        $s_n\in\mathbb I_{-3}$ for $0\leq n\leq N_k$, with
	        $s_{N_k}<\cdots<s_0$.
	        With
	        $h^{-1}=2^{-k}$,
	        \[
	        x_1-t_m(x)-s_n
	        =3-(1+2^{-4}+2^{-3})+h^{-1}+(n-m)h^{-1}.
	        \]
	        Hence only $m=n$ matches the support of $E_{1,k}$. The
	        box-separation argument in part (b), together with the angular
	        rectangle of measure comparable to $h^{-1}$ in the matching case,
	        gives
	        \[
	        \mathcal{A}f_k(x,t_m(x),s_m)\geq c2^{-k},
	        \qquad
	        \mathcal{A}f_k(x,t_m(x),s_n)=0 \quad(m\neq n).
	        \]
	        The same argument gives $|E_{1,k}|\sim h^{-2}$ and
	        $|E|\sim h^{-1}$.
	        Taking partitions containing
	        $t_0(x)<\cdots<t_{N_k}(x)$ and $s_{N_k}<\cdots<s_0$, we obtain
	        \[
	        \begin{aligned}
	        V_r(\mathcal{A}f_k)(x)
	        &\geq C\Big(\sum_{m=1}^{N_k}
	        |\mathcal{A}f_k(x,[t_{m-1}(x),t_m(x)]
	        \times[s_m,s_{m-1}])|^r\Big)^{1/r}\\
	        &\gtrsim2^{k(1/r-1)}.
	        \end{aligned}
	        \]
	        Hence
	        \[
	            \|V_r(\mathcal{A}f_k)\|_{L^q(\mr^3)}
	            \gtrsim2^{k(1/r-1)}|E|^{1/q}
	            \gtrsim2^{k(1/r-1-1/q)}.
	        \]
	        Moreover,
	        \[
	            \|f_k\|_{\Lp(\mr^3)}
	            \lesssim|E_{1,k}|^{1/p}
	            \lesssim2^{-2k/p}.
	        \]
	        Boundedness therefore yields $L_3(p,q,r)=-\frac{2}{p}+\frac{1}{q}-\frac{1}{r}+1\geq0$, which proves (c).
		\end{proof}

\medskip
\noindent\emph{Completion of the necessary parts.}
Propositions~\ref{prop-nece-r>2} and \ref{prop-nec-Lp-Lq} yield
$r\geq2$ and $L_i(p,q,r)\geq0$ for $1\leq i\leq5$.
Moreover, $V_r(\mathcal A)$ dominates $\mathcal M_c$, and
\cite[Theorem~1.2]{LL} shows that $\mathcal M_c$ is unbounded whenever
$(1/p,1/q)\notin\overline{\mathcal Q}\setminus\{P_4\}$. Combining this unboundedness result with
the inequalities above proves the necessary inclusion in
Theorem~\ref{thm-Lp-Lq}.

When $q=p$, the condition $L_5(p,p,r)\geq0$ is equivalent to
$r\geq2$. Under the necessary conditions $r\geq2$ and $p>2$, the
conditions $L_2(p,p,r)\geq0$ and $L_3(p,p,r)\geq0$ are redundant.
Thus the conditions $r\geq2$, $L_1(p,p,r)\geq0$, $L_4(p,p,r)\geq0$, and $p>2$ are equivalent to
$(1/p,1/r)\in\mathfrak Q_1\setminus[Q_1,P_4]$, proving the
necessary part of Theorem~\ref{thm-two-para}.
Finally, Proposition~\ref{prop-nece-r>2} (a) and (b) give
\[
\frac1r\leq\min\left\{\frac12,\frac3p\right\}
\]
for $\mathcal A_1$, while boundedness of the corresponding one-parameter maximal operator requires $p>2$.
These are exactly the necessary conditions in
Theorem~\ref{thm-one-para}.

\section*{Acknowledgments}
S.~Lee and S.~Zhao were supported by the National Research Foundation
of Korea under Grant No.~RS-2024-00342160.
J.~Lee was supported by KIAS Individual Grant No. MG098901.
F.~Zhang was supported by the China Scholarship Council (CSC) and the National Natural Science Foundation of China (Grant No. 12571106).

\end{document}